\documentclass[10pt,reqno]{amsart}
\usepackage[utf8]{inputenc}
\usepackage[T1]{fontenc}
\usepackage{lmodern}
\usepackage{xspace}
\usepackage{amsfonts,amsmath,amssymb,amsthm}
\usepackage[small]{titlesec}
\usepackage{comment}
\usepackage{mathrsfs}

\usepackage{fancyhdr}

\usepackage{hyperref}

\usepackage{color}
\usepackage[a4paper,body={160mm,235mm},centering]{geometry}
\def\p{{\partial}}

\def\F{{\mathcal{F}}}
\def\I{{\mathbb{I}}}

\def\bL{{\mathbb{L}}}
\def\bB{{\mathbb{B}}}
\def\P{{\mathbb{P}}}
\def\E{{\mathbb{E}}}
\def\N{{\mathbb{N}}}

\def\R{{\mathbb{R}}}
\def\gF{{\mathbf{F}}}
\newcommand \A[1]{{\bf (#1)}}

\def\leftB{[\![}
\def\rightB{]\!]}

\newcommand{\mW}{\mathcal{W}}  
\newcommand{\mZ}{\mathcal{Z}} 
\newcommand{\dotafter}[1]{#1.}
\titleformat{\section}[hang]
{\normalfont\large\bfseries}{\thesection.}{.5em}{\dotafter}[]
\titleformat{\subsection}[runin]
{\normalfont\bfseries}{\thesubsection.}{.4em}{}[.]
\titlespacing*{\subsection}{0pt}{3ex plus 1ex minus .2ex}{1em}
\titleformat{\paragraph}[runin]{\normalfont\bfseries}{\theparagraph.}{.4em}{}[.]
\theoremstyle{plain}
\newtheorem{THM}{Theorem}
\newtheorem{lem}[THM]{Lemma}
\newtheorem{PROP}[THM]{Proposition}
\newtheorem{cor}[THM]{Corollary}

\theoremstyle{definition}
\newtheorem{df}{Definition}

\theoremstyle{remark}
\newtheorem{REM}{Remark}

\begin{document}
\title[Stable SDEs with Besov drift]{On Multidimensional stable-driven Stochastic Differential Equations with Besov drift}

\author[P.-\'E. Chaudru de Raynal]{Paul-\'Eric Chaudru de Raynal}
\address{Nantes Université, CNRS, Laboratoire de Mathématiques Jean Leray, LMJL, F-44000 Nantes, France}
\email{pe.deraynal@univ-nantes.fr}

\author[S. Menozzi]{St\'ephane Menozzi}
\address{Universit\'e d'Evry Val d'Essonne, Paris Saclay, 
Laboratoire de Math\'ematiques et Mod\'elisation d'Evr, UMR CNRS 8070, 23 Boulevard de France 91037 Evry, France\\
  and Laboratory of Stochastic Analysis, HSE,
Prokrovsky Boulevard 11, Moscow, Russian Federation}
\email{stephane.menozzi@univ-evry.fr}

\begin{abstract}
We establish well-posedness results for multidimensional non degenerate $\alpha$-stable driven SDEs with time inhomogeneous singular drifts in $\bL^r-{\mathbb B}_{p,q}^{-1+\gamma}$ with $\gamma<1$ and $\alpha$ in $(1,2]$, where $\bL^r$ and ${\mathbb B}_{p,q}^{-1+\gamma} $ stand for Lebesgue and Besov spaces respectively. Precisely, we first prove the well-posedness of the corresponding martingale problem and then give a precise meaning to the dynamics of the SDE. {\color{black}This allows us in turn to define an \emph{ad hoc} notion of weak solution, for which well-posedness holds as well.}
Our results rely on the smoothing properties of the underlying PDE, which is investigated by combining a perturbative approach with duality results between Besov spaces.
\end{abstract}

\maketitle
\thispagestyle{fancy}

\bigskip

\section{Introduction}
\subsection{Statement of the problem and related literature}
We are here interested in providing a well-posedness theory for the following \emph{formal} $d$-dimensional  stable driven SDE. \textcolor{black}{For a fixed $T> 0 $, $t\in [0,T] $}:
\begin{equation}
\label{SDE}
X_t=x+\int_0^t F(s,X_s)ds+\mathcal W_t,
\end{equation}
where in the above equation $(\mW_s)_{s\ge 0} $ is a $d $-dimensional symmetric $\alpha$-stable process, for some $\alpha$ in $(1,2]$ (thus including Brownian noise). 

The main point here comes from the fact that the drift $F$ is only supposed to belong to the space $\mathbb{L}^r(\textcolor{black}{[0,T]},\bB_{p,q}^{-1+\gamma}(\R^d,\R^d))$, where $\bB_{p,q}^{-1+\gamma}(\R^d,\R^d)$ denotes a Besov space. {\color{black} In a nutshell, when $p=q=\infty$, for any non integer $\beta > 0$, Besov spaces coincide with H\"older spaces $\bB_{\infty,\infty}^{\beta}(\R^d,\R^d) = \mathcal C^{\beta}(\R^d,\R^d)$; when $\beta <0$, this somehow indicates that the H\"older modulus blows up at rate $\beta$. The parameters $p$ and $q$ are related to the integrability of such a modulus. We refer to Section 2.6.4 of \cite{trie:83} \textcolor{black}{and Section \ref{SEC_BESOV} below} for rigorous definition}. 

 \textcolor{black}{The}  parameters $(p,q,\gamma,\textcolor{black}{r})$ s.t. $1/2<\gamma<1$, $p,q,r\ge 1$ \textcolor{black}{will have to satisfy} some constraints \textcolor{black}{to} be specified later on \textcolor{black}{in order to give a meaning to \eqref{SDE}}. Importantly, assuming the parameter $\gamma$ to be strictly less than $1$ implies that $F$ can even not be a function, but just a distribution, so that it is not clear that the integral part in \eqref{SDE} has any meaning, at least as this. This is the reason why, at this stage, we talk about  ``\emph{formal} $d$-dimensional stable SDE'' or ``\emph{formal} SDE \eqref{SDE}''. There are many approach\textcolor{black}{es} to tackle such a problem which mainly depend on the choice of the parameters $p,q,\gamma,r,\alpha$ and \textcolor{black}{the dimension} $d$. Let us now try to review some of them.\\

\emph{The Brownian setting: $\alpha=2$.} There already exists a rather large literature about singular/distributional SDEs of type \eqref{SDE}. Let us first mention the work by Bass and Chen \cite{bass_stochastic_2001} who derived in the Brownian scalar case the strong well-posedness of \eqref{SDE} when the drift writes (still formally) as $F(t,x)=F(x)=aa'(x)$, for a spatial function $a$ being $\beta$-H\"older continuous with $\beta>1/2 $ and for a multiplicative noise associated with $a^2$, i.e. the additive noise $\mathcal W_t$ in \eqref{SDE} must be replaced by $\int_0^t a(X_s) d\mathcal W_s$. The key point in this setting is that the underlying generator associated with the SDE writes as $\textcolor{black}{L= (1/2) \partial_x \big(a^2 \partial_x \big)} $. From this specific divergence form structure, the authors manage to use the theory of Dirichlet forms of Fukushima \textit{et al.} (see \cite{fuku:oshi:take:10}) to give a proper meaning to \eqref{SDE}. Importantly, the formal integral corresponding to the drift has to be understood as a Dirichlet process. Also, in the particular case where the distributional derivative of $a$ is a signed Radon measure, the authors give an explicit expression of the drift of the SDE in terms of the local time (see  Theorem 3.6 therein). In the multi-dimensional Brownian case, Bass and Chen have also established weak well-posedness of SDE of type \eqref{SDE} when the homogeneous drift belongs to the Kato class, see \cite{bass:chen:03}.

Many authors have also recently investigated SDEs of type \eqref{SDE} in both the scalar and multidimensional Brownian setting for time inhomogeneous drifts in connection with some physical applications. From these works, it clearly appears that handling time inhomogeneous distributional drift can be a more challenging question. Indeed, in the time homogeneous case, denoting by $\gF $ an antiderivative of $F$, one can observe that the generator of \eqref{SDE} can be written in the form $(1/2) \exp(-2\gF(x))\partial_x\big( \exp(2\gF(x))\partial_x\big) $ and the dynamics can again be investigated within the framework of Dirichlet forms (see e.g. the works by Flandoli, Russo and Wolf, \cite{flan:russ:wolf:03}, \cite{flan:russ:wolf:04}). The crucial point is that in the time inhomogeneous case such connection breaks down. 

In this time inhomogeneous framework, we can mention the work by Flandoli, Issoglio and Russo \cite{flandoli_multidimensional_2017} for drifts in fractional Sobolev spaces. The authors establish therein 
the existence and uniqueness of what they call \emph{virtual solutions} to \eqref{SDE}: such solutions are defined through the diffeomorphism induced by the Zvonkin transform in \cite{zvonkin_transformation_1974} which is precisely designed  to get rid of the \textit{bad} drift through It\^o's formula. Namely,  \textcolor{black}{they investigated the smoothness properties of the underlying 
PDE (with the drift as source term and \textcolor{black}{null} terminal condition) which {\color{black}\emph{formally}} writes\footnote{The PDE investigated in \cite{flandoli_multidimensional_2017} slightly differs from the one introduced here (additional potential term therein). However, the resulting Zvonkin transforms are \textit{somehow} equivalent. We have thus chosen to present their result according to the approach we adopted here for the sake of clarity.}, 
\begin{equation}\label{PDE_FIRST_ZVONKIN}
\left\{\begin{array}{l}
\partial_t u+F\cdot D u+\frac12 \Delta u=-F, \text{ on } [0,T)\times \R^d\\
u(T,\cdot)=0.\qquad 
\end{array}\right.
\end{equation}
We said \emph{formal} because above, it is not clear that the product $F\cdot D u$ is meaningful, since product between two distributions can only be defined under suitable constraints.}

In their work, \textcolor{black}{they} managed to prove that this product indeed makes sense and that this PDE admits a unique \emph{mild} solution (\textcolor{black}{see Definition 8 therein}), i.e. 
\begin{equation}\label{PDE_FIRST_ZVONKIN_mild}
u_t = \int_t^Tds P_{s-t}[\{\textcolor{black}{F(s,\cdot)} + F(s,\cdot)\cdot D u_s\}],
\end{equation}

%
where $(P_t)_t$ denotes the usual heat \textcolor{black}{semi-group}, which belongs to a suitable function space, allowing them to define the Zvonkin transform $\Phi(t,x)=x+u(t,x) $. \textcolor{black}{The authors then introduce the notion of \textit{virtual solution}. Namely, this is a process $X$ defined on a stochastic basis such that:
$$X_t=x+u(0,x)-u(t,X_t)
+\int_0^1 (\nabla u(s,X_s)+I) d\mW_s.$$ 
The main advantage of such a definition is that the \textit{bad} drift does not explicitly appears. It actually would from a formal expansion of $u$ solving \eqref{PDE_FIRST_ZVONKIN} through It\^o's formula. This correspondence also gives that, for reasonable smooth drifts, classical solutions also are virtual solutions. Existence and uniqueness for virtual solutions are then established from the fact that $\Phi $ can be shown, from the well-posedness of the mild solution \eqref{PDE_FIRST_ZVONKIN_mild}, to be a $C^1$-diffeomorphism for 
$T$ small enough.} 

\textcolor{black}{
The SDE
\begin{equation}
\label{DYN_ZVONKIN_TRANSFORM}
Y_t=\Phi(0,x)
+\int_0^t D \Phi(s,\Phi^{-1}(s,Y_s)) d\mW_s,
\end{equation}
 has indeed itself a unique weak solution from the smoothness of $u$ solving \eqref{PDE_FIRST_ZVONKIN} and  $X_t=\Phi^{-1}(t,Y_t) $ is a virtual solution for which uniqueness in law holds}. \textcolor{black}{As a consequence of this approach only few things can be said about the original dynamics.}

 \textcolor{black}{In this last perspective, we can also refer to the work of Zhang and Zhao \cite{ZZ17}, who \textcolor{black}{established} \textcolor{black}{in the time homogeneous case} the well-posedness of the martingale problem for the generator associated with \eqref{SDE}, which \textcolor{black}{can in their framework} \textcolor{black}{additionally} contain  a non trivial smooth enough diffusion coefficient \textcolor{black}{(see also Remarks \ref{REM_DIFF_PRELI} and \ref{REM_COEFF_DIFF} below)}. Therein, they obtained as well as some Krylov type density estimates in Bessel potential spaces for the solution. Also, they manage\textcolor{black}{d to derive a} more precise description of the limit drift in the \textit{formal} dynamics in \eqref{SDE}, which is interpreted  as a suitable limit of a sequence of mollified drifts, i.e. $\lim_n \int_0^t F_m(s,X_s) ds $ for a sequence of smooth functions $(F_m)_{m\ge 1} $ converging to $F$ in a suitable sense.}

The key point in these works, who heavily rely on PDE arguments, is to establish: (i) that the product $F\cdot D  u$ in \eqref{PDE_FIRST_ZVONKIN} is meaningful as a distribution (thus with the same regularity as $F$); (ii) that the semi-group $(P_t)_t$ associated with the noise maps the quantity in the bracket in right hand side of \eqref{PDE_FIRST_ZVONKIN_mild} onto a suitable function space, say at least $\mathcal C^{1^+}(\R^d,\R^d)$  to define $Du$ properly. Let illustrate how such constraints translate, assuming for a while that $F$ is time homogeneous and belongs to the Besov-H\"older space $\bB^{-1+\gamma}_{\infty,\infty}(\R^d,\R^d) = \mathcal C^{-1+\gamma}(\R^d,\R^d)$. From the smoothing effect of the noise (parabolic bootstrap), one expects that the semi-group maps $ \mathcal C^{-1+\gamma}(\R^d,\R^d)$ onto $ \mathcal C^{-1+\gamma+2}(\R^d,\R^d)=\mathcal C^{1+\gamma}(\R^d,\R^d)$ (Schauder estimates). The second constraint (ii) then gives $-1+\gamma+2>1 \Leftrightarrow \gamma >0$. On the other hand, as $Du$ belongs to $ \mathcal C^{\gamma}(\R^d,\R^d)$, the first constraint (i) gives, from Bony's paraproduct rule, that the sum of the regularity indexes of $F$ and $Du$ must be strictly positive: $-1+\gamma+\gamma >0 \Leftrightarrow \gamma >1/2$.

This is indeed the threshold appearing in \cite{flandoli_multidimensional_2017} and \cite{ZZ17} as well as the one previously obtained in \cite{bass_stochastic_2001}. \textcolor{black}{This is precisely the threshold that will guarantee well-posedness of the corresponding Martingale problem and weak well-posedness of the associated dynamics, with the Definitions of Section \ref{SEC_1.3_TITI} below (see Definitions \ref{DEF_MPB} and \ref{WEAK-DEF}), for a drift $F\in \bL^\infty(\textcolor{black}{[0,T]},\bB_{\infty,\infty}^{-1+\gamma}\textcolor{black}{(\R^d,\R^d)} )$ in the present work. 
}

To bypass such a limit \textcolor{black}{(\emph{i.e.} $\gamma>1/2$)}, one therefore has to use a suitable theory in order to \textcolor{black}{first} give a meaning to the product $F\cdot Du$. This is, for instance, precisely the aim of either rough paths, regularity structures or paracontrolled calculus. However, as a price to pay to enter this framework, one has to add some structure to the drift assuming that this latter can be enhanced into a rough path structure. In the scalar Brownian setting, and in connection with the KPZ equation, Delarue and Diel \cite{dela:diel:16} used such specific structure to extend the previous results for an inhomogeneous drift which can be viewed as the generalized derivative of $\gF$ with H\"older regularity index greater than $1/3$ (i.e. assuming that $F$ belongs to $\bL^{\infty}([0,T],\bB^{(-\textcolor{black}{2}/3)^+}_{\infty,\infty}\textcolor{black}{(\R,\R)}$).  Importantly, in \cite{dela:diel:16} the authors derived a very precise description of the meaning of the \emph{formal} dynamics \eqref{SDE}: they show that the drift of the solution may be understood as stochastic-Young integral \textcolor{black}{involving the} mollification of the distribution by the transition density of the underlying noise. As far as we know, it appears to us that such a description is the \textcolor{black}{most} accurate that can be found in the literature on stochastic processes (see \cite{catellier_averaging_2016} for a pathwise version and Remark 18 in \cite{dela:diel:16} for some comparisons between the two approaches). With regard to the martingale problem, the result of \cite{dela:diel:16} has then been extended to the multidimensional setting by Cannizzaro and Choukh \cite{canni:chouk:18}, but nothing is said therein about the dynamics.\\

\emph{The pure jump case: $\alpha<2$.} In the pure jump case, there are a few works concerning the well-posedness \textcolor{black}{of} \eqref{SDE} in the singular/distributional case. Even for drifts that are functions, strong uniqueness was shown rather recently. Let us distinguish two cases: the   
\textit{sub-critical case} $\alpha \ge 1 $, in this case the noise dominates the drift (in term of self-similarity index $\alpha $)  and the \textit{super-critical} case \textcolor{black}{$\alpha<1 $} where the noise does not dominate. In the first case, we can refer for bounded H\"older drifts to Priola \cite{prio:12} who proved that strong uniqueness holds (for time homogeneous) functions  $F$ in \eqref{SDE} which are $\beta$-H\"older continuous provided $\beta >1-\alpha /2$. In the second case, the strong well-posedness has been established under the same previous condition by Chen \textit{et al.} \cite{CZZ17}. \textcolor{black}{Those results are multi-dimensional}. \textcolor{black}{In any cases, the threshold obtained does not allow the Authors to consider singular/distributional drift, so that these results are not comparable with the present work.}
 
\textcolor{black}{On the other hand, in} the current distributional framework, \textcolor{black}{and  in the scalar case}, the martingale problem associated with the formal generator of \eqref{SDE}
has been recently investigated by Athreya, Butkovski and Mytnik \cite{athr:butk:mytn:18} for $\alpha> 1$ and a time homogeneous $F \in \bB_{\infty,\infty}^{-1+\textcolor{black}{\gamma}} \textcolor{black}{(\R,\R)}$ under the condition: \textcolor{black}{$-1+\textcolor{black}{\gamma}>(1-\alpha)/2 \Leftrightarrow \gamma > (3-\alpha)/2$}.  After specifying how the associated dynamics can be understood, viewing namely the drift as a Dirichlet process (similarly to what was already done in the Brownian case in \cite{bass_stochastic_2001}), they eventually manage to derive strong uniqueness under the previous condition. Note that results in that direction have also been derived by Bogachev and Pilipenko in \cite{bo:pi:15} for drifts belonging to a certain Kato class \textcolor{black}{in the multidimensional setting}.

Again, the result obtained by Athreya, Butkovski and Mytnik \textcolor{black}{relies} on \textcolor{black}{the Zvonkin} transform and hence \textcolor{black}{requires to have} a suitable theory for the associated PDE. In \textcolor{black}{our} pure-jump \textcolor{black}{time inhomogeneous} framework, it \emph{formally} writes
\begin{equation}\label{PDE_FIRST_ZVONKIN_S}
\left\{\begin{array}{l}
\partial_t u+F\cdot D u+L^\alpha u
=-F, \text{ on } [0,T)\times \R,\\
u(T,\cdot)=0,
\end{array}\right.
\end{equation}
where $L^\alpha$ is the generator of a non-degenerate $\alpha $-stable process. 

Considering again a time homogeneous $F$ in $\mathcal C^{-1+\gamma}(\R,\R)$, one may \textcolor{black}{reproduce} the analysis done before in the Brownian setting to deduce that the solution $u$ is expected to be in $\mathcal C^{-1+\gamma+\alpha}(\R,\R)$ (as the smoothing effect of the \textcolor{black}{semi-group} associated to the generator $L^\alpha$ is now of order $\alpha$). The two constraints exposed in the Brownian setting ((i) to define the product $F\cdot Du$; (ii) to define the gradient $Du$) translate into $-1+\gamma+\alpha >1 \Leftrightarrow \gamma > 2-\alpha$ for (ii) and $-1+\gamma-1+\gamma+\alpha - 1>0 \Leftrightarrow \gamma >(3-\alpha)/2$ for (i). \textcolor{black}{This is precisely the threshold that will guarantee weak well-posedness and pathwise uniqueness in the scalar case  for a drift $F\in \bL^\infty(\textcolor{black}{[0,T]},\bB_{\infty,\infty}^{-1+\gamma}\textcolor{black}{(\R^d,\R^d)} )$ in the present work}.

\subsection{Aim of the paper} 
In the current work, we aim at investigating a \textcolor{black}{rather} large framework by considering the $d$-dimensional case $d\ge1$, with a distributional, \textcolor{black}{potentially} singular \textcolor{black}{in} time, inhomogeneous drift (in $\bL^r([0,T],\bB^{-1+\gamma}_{p,q}\textcolor{black}{(\R^d,\R^d)}))$ when the noise driving the SDE is \textcolor{black}{a} symmetric $\alpha$-stable process, $\alpha$ in $(1,2]$. \textcolor{black}{This setting thus includes} both the Brownian and pure-jump case. \textcolor{black}{In the latter case, we will also be able to consider driving noises with singular spectral measures}.  As \textcolor{black}{previously} done for the aforementioned result\textcolor{black}{s}, our strategy rel\textcolor{black}{ies} on the \textcolor{black}{idea by Zvonkin}. \textcolor{black}{The core} of the analysis \textcolor{black}{therefore consists in obtaining} suitable \textit{a priori} estimates on an associated underlying PDE of type \eqref{PDE_FIRST_ZVONKIN} or \eqref{PDE_FIRST_ZVONKIN_S}. Namely, we will provide a Schauder \textcolor{black}{like} theory for the mild solution of such PDE for a large class of data. This result is also part of the novelty of our approach since these estimates are obtain\textcolor{black}{ed} thanks to a rather robust methodology based on heat-kernel estimate\textcolor{black}{s} on the transition density of the driving noise together with duality results between Besov spaces viewed through their thermic characterization (see Section \ref{SEC_BESOV} below and  Triebel \cite{trie:83} for additional properties on Besov spaces and their characterizations). This approach does not distinguish the pure-jump and Brownian setting provided the heat-kernel estimates hold. \textcolor{black}{It has for instance also been successfully applied in various frameworks, to derive Schauder estimates and strong uniqueness for a degenerate Brownian chain of SDEs (see \cite{chau:hono:meno:18:S}, \cite{chau:hono:meno18}) or Schauder estimates for super-critical fractional operators \cite{chau:meno:prio:19}}.\\

{\color{black} More precisely, 
 our first main result consists in deriving the well-posedness of the martingale problem introduced in Definition \ref{DEF_MPB}  under suitable conditions on the parameters $p,q,r$ and $\gamma $, see Theorem \ref{THEO_WELL_POSED}. 

Then, under slightly reinforced conditions on $p,q,r$ and $\gamma $, we are able to reconstruct the dynamics for the canonical process associated with the solution of the martingale problem, see Theorem \ref{THEO_DYN}, specifying how the Dirichlet process associated with the drift writes. In the spirit of \cite{dela:diel:16}, we in particular exhibit a main contribution in this drift.

Inspired by the dynamics exhibited for the Martingale solution, we define next an ad hoc notion of weak solution in Definition \ref{WEAK-DEF} and prove the associated well-posedness result, see Theorem \ref{THEO_STRONG}. Therein, we also manage to derive pathwise uniqueness in the scalar case, extending partially the previous results of \cite{athr:butk:mytn:18}.

Eventually, we manage in Proposition \ref{PROP_DRIFT} to collect all the results we have at hand to specify the dynamics of both the Martingale and weak solution. These results could be useful to investigate the numerical approximations of those singular SDEs (see equations \eqref{DYNAMICS} and \eqref{DECOMP_DRIFT}) and the recent work by De Angelis \textit{et al.} \cite{dean:germ:isso:19}
}

Let us conclude by \textcolor{black}{mentioning} that, while finishing the preparation of the present manuscript, we discovered a brand new preprint of Ling and Zhao \cite{ling:zhao:19} which somehow presents some overlaps with our results. Therein, the Authors investigate \textit{a priori} estimates for the elliptic version of the PDE of type \eqref{PDE_FIRST_ZVONKIN} or \eqref{PDE_FIRST_ZVONKIN_S}  with (homogeneous) drift belonging to H\"older-Besov spaces with negative regularity index (i.e. in $\textcolor{black}{\bB_{\infty,\infty}^{-1+\gamma}} \textcolor{black}{(\R^d,\R^d)}$) and including a non-trivial diffusion coefficient provided the spectral measure of the driving noise is absolutely continuous. As an application, they derive the well-posedness of the associated martingale problem and prove that the drift can be understood as a Dirichlet process. They also obtained quite sharp regularity estimates on the density of the solution and succeeded in including the limit case $\alpha=1$. 

In comparison with their results, we here manage to handle the case of \textcolor{black}{an inhomogeneous and singular in time} drift \textcolor{black}{which can also have} additional space singularities, since the integrability indexes of the parameter $p,q$ for the Besov space are not supposed to be $p=q=\infty$ (\textcolor{black}{recall that we assume $F\in \bL^r([0,T] ,\bB_{p,q}^{-1+\gamma}\textcolor{black}{(\R^d,\R^d)})$}). Although we did not include it, we could also handle in our framework an additional non-trivial diffusion coefficient under their standing assumptions, we refer to Remarks \ref{REM_DIFF_PRELI} and \ref{REM_COEFF_DIFF} below concerning this point. It also turns out that we obtain more accurate version of the dynamics of the solution which is here, as mentioned above, tractable enough for practical purposes. 
We eventually mention that, as a main difference with our approach, the controls in \cite{ling:zhao:19} are mainly obtained through  Littlewood-\textcolor{black}{Paley} decompositions whereas we rather exploit the thermic characterization and the parabolic framework for the PDE. In this regard, we truly think that the methodology to derive the a priori estimates in both works can be seen as complementary.\\

\subsection{Overview of the paper} We state in this part our main assumptions and results, together with rigorous definition about the meaning of solution of the \emph{formal} SDE \eqref{SDE}.
\label{SEC_1.3_TITI}

\bigskip
\noindent\textbf{Framework.} \textcolor{black}{As already said, we consider the problem of the solvability, in a sense to be \textcolor{black}{specified later} on, of the \emph{formal} SDE \eqref{SDE} with drift $F\in \mathbb L^{r}([0,T], \mathbb B^{-1+\gamma}_{p,q}(\R^d,\R^d))$, with $p,q,r\ge1$ and $\gamma \in (1/2,1)$. }
Concerning the driving noise in \eqref{SDE}, we will denote by $L^\alpha $ its generator. When $\alpha=2 $, $L^2=(1/2) \Delta $ where $\Delta$ stands for the usual Laplace operator on $\R^d $. In the pure-jump stable case $\alpha \in (1,2) $,
 for all $\varphi\in C_0^\infty(\R^d,\R) $:
\begin{equation}
\label{GENERATEUR_STABLE}
L^\alpha \varphi(x) ={\rm p.v.}\int_{\R^d} \big[\varphi(x+z)-\varphi(x)\big]\nu(dz),
\end{equation}
where, writing in polar coordinates $z= \rho \xi,\ (\rho,\xi)\in \R_+\times {\mathbb S}^{d-1}$, the L\'evy measure decomposes as $\nu(dz)=\mu(d\xi)/\rho^{1+\alpha} $ with $\mu $  a symmetric non degenerate measure on the sphere ${\mathbb S}^{d-1} $. Precisely, we assume:
\begin{trivlist} 
\item[\A{UE}] There exists $\kappa\ge 1$ s.t. for all $\lambda \in \R^d $:
\begin{equation}
\label{NON_DEG}
\kappa^{-1}|\lambda |^\alpha \le \int_{{\mathbb S}^{d-1}} | \textcolor{black}{\lambda \cdot \xi}|^\alpha \mu(d\xi)\le \kappa|\lambda|^\alpha.
\end{equation}
\end{trivlist}
Observe in particular that a rather large class of spherical measures $\mu $ satisfy \eqref{NON_DEG}. From the Lebesgue measure, which actually leads, \textcolor{black}{up to a normalizing constant}, to $L^\alpha=-(-\Delta)^{\alpha/2} $ (usual fractional Laplacian of order $ \alpha$ corresponding to the \textcolor{black}{generator of the isotropic} stable process), to sums of Dirac masses in each direction, i.e. $\mu_{{\rm Cyl}} =\sum_{j=1}^d c_j(\delta_{e_j}+\delta_{-e_j})$, with $(e_j)_{j\in \leftB 1,d\rightB} $ standing for the canonical basis vectors, which for $c_j= 1/2 $ then yields $L^\alpha =-\sum_{j=1}^d (-\partial_{x_j}^2)^{\alpha/2} $ corresponding to the cylindrical fractional Laplacian of order $\alpha $ associated with the sum of scalar symmetric $\alpha $-stable processes in each direction. In particular, it is clear that under \A{UE}, the process $\mW $ admits a smooth density in positive time (see e.g. \cite{kolo:00}). Correspondingly, $L^\alpha$ generates a semi-group that will be denoted from now on by $P_t^\alpha=\exp(tL^\alpha)$.
Precisely, for all $\varphi \in B_b(\R^d,\R)$ (space of bounded Borel functions), and all $t>0$:
\begin{equation}
\label{DEF_HEAT_SG}
P_{t}^\alpha[\varphi](x) := \int_{\R^d} dy p_\alpha(t,y-x) \varphi(y),
\end{equation}
where $p_\alpha(t,\cdot) $ stands for the density of $\mW_t $. Further properties associated with the density $p_\alpha $, in particular concerning the integrability properties of its derivatives, are stated in Section \ref{SEC_MATH_TOOLS}.\\

{\color{black}
\bigskip
\noindent\textbf{Notion of solution of the \emph{formal} SDE \eqref{SDE}. }In order to foster an appropriate notion of solution for the \emph{formal} SDE \eqref{SDE}, let us start with an informal discussion which can be found as well in the work of Cannizzaro and Choukh \cite{canni:chouk:18}. We reproduce it for \textcolor{black}{the} sake of clarity. Let $\Omega_2 = \mathcal C([0,T],\R^d) $ and  $\Omega_\alpha = \mathcal D([0,T],\R^d)$ (space of c\`adlag function) when $0< \alpha<2$. The Stroock and Varadhan formulation of the Martingale Problem associated with an operator $F\cdot D + L^\alpha$ reads as: find a probability measure $\mathbb P^\alpha$ on the space $\Omega_\alpha$ equipped with the canonical filtration so that
\begin{enumerate}
\item[(a)] $\displaystyle \mathbb P^\alpha(X_0 = x) = 1$,
\item[(b)] For all $\phi$ in $\mathcal E$ \textcolor{black}{(where $\mathcal E$ here stands for a rich enough function space)}:
\begin{equation}\label{PMG}
\left(\phi(t,X_t) - \int_0^t  (\p_s + F\cdot D + L^\alpha) \phi(s,X_s) ds\right)_{0 \le t \le T},
\end{equation}
is a (square integrable if $\alpha=2$) martingale under $\mathbb P^\alpha$.
\end{enumerate} 
Above, the class $\mathcal E$ has to be chosen so that it is sufficiently rich to characterize a Markov process through the martingale formulation, see e.g. \cite{EK:86}. In our current setting, the main issue comes from the fact that the operator \textcolor{black}{involves} a distributional part (the drift term $F$) so that even if the products in $(F\cdot D+L^\alpha) \phi$ are well defined, this term could only be a distribution with same regularity as $F$. To avoid such a consideration the idea consists in \textcolor{black}{taking }$\mathcal E$ as the set of maps for which $(\p_t + F\cdot D+ L^\alpha)\phi$ is indeed a function $f$, \emph{i.e.} $\mathcal E$ should be the set of solutions of the  Cauchy Problem $\mathscr C(F,L^\alpha,f,g,T)$:
\begin{eqnarray}\label{EDP}
(\p_t +  F\cdot D +  L^\alpha) u = f,\quad u_T= g,
\end{eqnarray}
for $f$ and $g$ in large enough classes $ \mathcal F$ and $\mathcal G$. Having this in mind, one may thus rewrite the associated Martingale Problem as: find a probability measure $\mathbb P^\alpha$ on the space $\Omega_\alpha$ equipped with its canonical filtration so that
\begin{enumerate}
\item[(a)] $\displaystyle \mathbb P^\alpha(X_0 = x) = 1$,
\item[(b)] For all $f,g$ in $\mathcal F, \mathcal G$, 
\begin{equation}\label{EDP_PMG}
\left(u(t,X_t) - \int_0^t f(s,X_s) ds - u(0,x)\right)_{0\le t \le T},
\end{equation}
with  $u$ the solution (in a sense to be specified) of the Cauchy Problem $\mathscr C(F,L^\alpha,f,g,T)$, is a  (square integrable if $\alpha=2$) martingale under $\mathbb P^\alpha$.\\
\end{enumerate}

\begin{REM}[On the link between both Martingale formulation \eqref{PMG} and \eqref{EDP_PMG} in a favorable setting]
Let us try to briefly illustrate the link between both formulations in a favorable setting. Assume first that $F$ is bounded, continuous in time and is $\gamma$-H\"older continuous in space. A natural choice for $\mathcal E$ is $C_b^{1,\alpha+\gamma}([0,T]\times \R^d,\R^d)$ so that $(\p_t+  F\cdot D + L^\alpha)\phi$ belongs to $C_b^{0,\gamma}([0,T]\times \R^d,\R^d)$. \textcolor{black}{This thus} means that $\mathcal F$ should be this \textcolor{black}{latter} function space  while $\mathcal G$ should be the space $C_b^{\alpha+\gamma}(\R^d)$. As a consequence, both formulations are equivalents. Indeed, under these conditions, for any  $f,g \in \mathcal F, \mathcal G$, Schauder estimates imply that there exists a unique element $u \in  C^{1,\alpha+\gamma}_b([0,T]\times \R^d,\R^d)$ that solves $ \mathscr C(F,L^\alpha,f,g,T)$ in a classical sense. Therefore \eqref{PMG} implies \eqref{EDP_PMG} and, conversely, \eqref{EDP_PMG} implies \eqref{PMG} choosing for any $\phi \in C^{1,\alpha+\gamma}_b([0,T]\times \R^d,\R^d)$: $f = (\p_t +  F\cdot D + L^\alpha )\phi$ and $g=\phi(T,\cdot)$.\\
\end{REM}

\bigskip
\noindent\textbf{Definition of the Martingale solutions and associated well-posedness results. }Let us come back to our general case, $F \in \mathbb L^r([0,T],\mathbb B^{-1+\gamma}_{\infty,\infty}(\R^d,\R^d))$. The main point is now to notice that, in order to obtain the martingale formulation \eqref{EDP_PMG}, we do not need to work with classical solutions. Up to a regularization argument and thanks to It\^o's formula, it is indeed enough to work with \emph{mild} solutions $u$ in $C^{0,1}([0,T] \times \R^d,\R)$ with bounded gradient. In our setting, the natural candidate for the \emph{mild} solution \emph{formally} writes
$$\forall (t,x)\in [0,T]\times \R^d,\quad u(t,x) = P^\alpha_{T-t}[g](x) - \int_t^T ds P^\alpha_{s-t}[\{f - F\cdot Du\}](s,x),$$
for e.g. continuously differentiable function $g$ and continuous in time and bounded in space function $f$. This form requires the last term on the right hand side to be defined as a function, whereas the drift $F$ is only assumed to be a distribution. We are thus again \textcolor{black}{led} to find appropriate conditions \textcolor{black}{that guarantee that}: (i) $F\cdot Du$ is well defined; (ii) the product is mapped onto the space of continuous in time and continuously differentiable in space functions with bounded gradient by the semi-group $(P^\alpha_t)_t$. 

This suggests that the gradient $Du$ should \textcolor{black}{belong} to a suitable Lebesgue-Besov space that depend\textcolor{black}{s} both on the driving noise and the drift. The following lemma, whose proof can be found in Subsection \ref{proof-pde}, is \textcolor{black}{useful to investigate} wether the above integral makes sense as a function and \textcolor{black}{also gives some hints concerning  the function space in which} the solution $u$ should be \textcolor{black}{sought}:  
\begin{lem}\label{lemme:defprod} Let $p,q,r \ge 1$, $\alpha \in (1,2]$ and $\gamma \in (1/2,1)$ be such that they satisfy a \emph{good relation}, \emph{i.e.}
\begin{equation}\label{good_relation}\tag{GR}
p,q,r\ge 1,\quad \alpha \in  \left( \frac{1+ [d/p]}{1-[1/r]},2\right],\quad \gamma \in \left(\frac{3-\alpha + [d/p] + [\alpha/r]}{2}, 1\right).
\end{equation}
Define then 
\begin{equation}\label{DEF_THETA}
\theta:=\gamma-1+\alpha-\frac dp-\frac \alpha r.
\end{equation}
Let $G$ in $\mathbb L^{r}([0,T], \mathbb B^{-1+\gamma}_{p,q}(\R^d,\R^d))$ and $v$ belongs to $\mathbb L^{\infty}([0,T],\mathbb B^{\theta-1-\varepsilon}_{\infty,\infty}(\R^d,\R^d))$ for any $0<\varepsilon<<1$.
Then, the map
$$\mathfrak r^v :  [0,T]\times \R^d \ni (t,x) \mapsto  \int_t^T ds P^\alpha_{s-t}[G\cdot v](s,x),$$
with $P^\alpha$ the semi-group generated by $L^\alpha$, belongs to $C^{0,1}_b([0,T]\times \R^d,\R)$. Moreover, the product $G\cdot v$ makes sense as an element of $\mathbb L^{r}([0,T], \mathbb B^{-1+\gamma}_{p,q}(\R^d,\R))$.
\end{lem}
\textcolor{black}{We point out that the parameter $\theta $ would correspond to the parabolic bootstrap, which thus also reflects the impact of the integrability indexes $p,r$ (but not $q$)}.

We are now in position to define the \emph{mild} solution of the \emph{formal} Cauchy problem and then state the associated well-posedness theorem.  
\begin{df}\label{DEF_PDESOL}
Let $\alpha \in (1,2] $, $f :\R_+ \times \R^d \to \R$ and $g:\R^d\to \R$. For any given fixed $T > 0$, we say that $u:[0,T]\times \R^d \to \R$ is a mild solution of the \emph{formal} Cauchy problem $\mathscr C(F,L^\alpha,f,g,T)$
\begin{equation}\label{Asso_PDE_PMG}
(\p_t + F\cdot D + L^\alpha)u(t,x) = f(t,x) \text{ on } [0,T) \times \R^d,\quad u(T,\cdot) =g(\cdot) \text{ on } \R^d,
\end{equation}
if it belongs to $\mathcal C^{0,1}([0,T]\times \R^d, \R)$ with $Du$ in $\mathcal C^0_b([0,T],\mathbb B^{\theta-1-\varepsilon}_{\infty,\infty}(\R^d,\R^d))$ for any $0<\varepsilon<<1$ where $\theta$ is given by \eqref{DEF_THETA} and satisfies
\begin{equation}\label{repsolPDE}
\forall (t,x) \in [0,T]\times \R^d,\quad u(t,x) = P_{T-}^\alpha[g](x) - \int_t^T ds P_{s-t}^\alpha[\left\{ f - F \cdot D u\right\}](s,x),
\end{equation}
with $(P^\alpha_t)_t$ the semi-group generated by $(L^\alpha_t)_t$. 
\end{df}

Notice that, due to definition of the space in which the solution should be \textcolor{black}{sought}, Lemma \ref{lemme:defprod} implies that the above mild formulation is meaningful. We now have the following well-posedness result proved in Section \ref{SEC_PDE_PROOF} (see in particular Subsection \ref{proof-pde}).

\begin{THM}\label{THE_PDE} Let $p,q,r\ge 1 $, $\alpha \in (1,2]$ and $\gamma \in (1/2,1)$ satisfy a \emph{good relation} \eqref{good_relation}. For all $f$ in $\mathcal C([0,T], \mathbb B^{\theta-\alpha}_{\infty,\infty}(\R^d,\R))$ and $g \in \mathcal C^1(\R^d,\R)$ with $Dg \in \bB_{\infty,\infty}^{\theta-1}(\R^d,\R^d)$, where $\theta$ is given by \eqref{DEF_THETA}, the \emph{formal} Cauchy problem $\mathscr C(F,L^\alpha,f,g,T)$ admits a unique solution in the sense of Definition \eqref{DEF_PDESOL}. Moreover it satisfies that for all  $(t\le s)$ in $[0,T]^2$, $ x$ in $\R^d$:
\begin{eqnarray*}
|u(t,x)-u(s,x)|&\le& C |t-s|^{\frac{\theta}{\alpha}},\notag\\
|Du(t,x)-Du(s,x)|&\le& C |t-s|^{\frac{\theta-1}{\alpha}}.\notag
\end{eqnarray*}
\end{THM}

We can therefore define the associated Martingale Problem and the corresponding well-posedness result proved in Section \ref{SEC_PREUVE} (see subsection \ref{SDE_2_PDE}, points $(i)$ to $(iii)$). 

\begin{df}\label{DEF_MPB} Let $\Omega_2 = \mathcal C([0,T],\R^d) $ and  $\Omega_\alpha = \mathcal D([0,T],\R^d)$ when $0< \alpha<2$. For any $\alpha \in (0,2\textcolor{black}{]}$, we say that a probability measure $\mathbb P^\alpha$ on $\Omega_\alpha$ equipped with its canonical filtration is a solution of the Martingale Problem associated with $(F,L^\alpha,x)$ for $x \in \R^d$ if   
\begin{enumerate}
\item[(i)] $\displaystyle \mathbb P^\alpha(X_0 = x) = 1$,
\item[(ii)] $\displaystyle \forall f \in \mathcal C([0,T], \mathcal S(\R^d,\R))$,  $g \in \mathcal C^1(\R^d,\R)$ with $Dg \in \bB_{\infty,\infty}^{\theta-1}(\R^d,\R^d)$,
$$\left(u(t,X_t) - \int_0^t f(s,X_s) ds - u(0,x)\right)_{0\le t \le T}$$ 
is a  (square integrable if $\alpha=2$) martingale under $\mathbb P^\alpha$ where $u$ is the mild solution of the Cauchy Problem $\mathscr C(F, L^\alpha,f,g,T)$ in the sense of Definition \ref{DEF_PDESOL}.
\end{enumerate}
\end{df}

\begin{THM}\label{THEO_WELL_POSED}
Let $p,q,r\ge 1 $, $\alpha \in (1,2]$ and $\gamma \in (1/2,1)$ satisfy a \emph{good relation} \eqref{good_relation}. Then, the Martingale Problem associated with  $(F, L^\alpha,x)$ for $x \in \R^d$, is well-posed in the sense of Definition \ref{DEF_MPB}. \textcolor{black}{Moreover, the canonical process under $\mathbb P^\alpha$ is strong Markov.}
\end{THM}

Note that, according to the notations of the previous paragraph, we chose above to work with $\mathcal E = \mathcal S(\R^d)$, where $\mathcal S$ stands for the class of Schwartz functions. This is mainly motivated by our approach based on Besov spaces as this class is continuously embedded into any Besov spaces $\bB^{s}_{l,m}(\R^d,\R)$, $s\in \R$, $1\le l,m\le \infty$, see e.g. paragraph 2.3.3 in \cite{trie:83}. Note that it is as well dense in  $\bB^{s}_{l,m}(\R^d,\R)$ provided $p,q<\infty$.

\bigskip
\noindent\textbf{Building the dynamics.}
One may \textcolor{black}{then} wonder if something can be said about the dynamics of the underlying process. In other words, the next step consists in linking the Martingale Problem and the \emph{formal} SDE \eqref{SDE}. \textcolor{black}{For sufficiently smooth drifts, the starting point to build a dynamics consists in recovering the noise from the  canonical process associated with the martingale solution. Since in the current framework the difficulty consists in specifying the drift, it therefore seems natural to have a noise at hand to precisely recover the drift. This is precisely what leads to consider an \textit{enlarged} martingale problem which \textit{includes} the noise}. Similar issues were e.g. discussed in Kurtz \cite{K11}.

\textcolor{black}{For singular drifts} this procedure was performed  in the specific setting of \cite{dela:diel:16}. We here extend this connection to our general setting and explain how one can \textcolor{black}{relate} the Martingale solution with the dynamics of \eqref{SDE}, and how this \textcolor{black}{latter} has to be understood. To do so, we will however be led to slightly reinforce the \emph{good relation} \eqref{good_relation}. Namely, 
 we say that $p,q,r \ge 1$, $\alpha \in (1,2]$ and $\gamma \in (1/2,1)$ satisfy a \emph{good relation for the dynamics} if the following relation holds:
\begin{equation}\label{good_relation_dyn}\tag{GR-D}
p,q,r\ge 1,\quad \alpha \in  \left( \frac{1+ [d/p]}{1-[1/r]},2\right],\quad \gamma \in \left(\frac{3-\alpha + [2d/p] + [2\alpha/r]}{2}, 1\right).
\end{equation}
Note that when $p=r=\infty$, the above condition and the previous one are equivalent. 

Having such a condition at hand, our main strategy consists in following the ideas of \cite{dela:diel:16}. Namely, we show that the canonical process can be decomposed into two parts: a stable driving noise, \emph{i.e.} an $\alpha$-stable process with generator $L^\alpha$ and a drift term, defined as a kind of non-linear stochastic Young integral of a regularized version of the initial drift by the density of the driving process. To the best of our knowledge, non-linear Young integrals ha\textcolor{black}{ve} been introduced in  \cite{catellier_averaging_2016} and further extend\textcolor{black}{ed} from a  probabilistic perspective in  \cite{dela:diel:16}. The rigorous definition of this last object is the following.

\begin{df}\label{Stochastic_young}  Let $\tau>0$, $(\tilde \Omega, \tilde{\mathcal F},(\tilde{\mathcal F}_t)_{0 \le t \le \tau}, \tilde {\mathbb P})$ be a filtered probability space and let $(\psi_t)_{0 \le t \le \tau }$ be a progressively measurable process on it. Let $(A(s,t))_{0 \le s \le t \le \tau}$ be a continuous and progressively measurable map in the sense that for any $0 \le s \le t$,
$$\tilde\Omega\times \{s'\in [0,s],\ t'\in [0,t],\ s'\le t'\} \ni (\omega,s',t') \mapsto A(s',t')$$
is $\tilde{\mathcal F}_t \otimes \mathcal B(\{s'\in [0,s],\ t'\in [0,t],\ s'\le t'\})$ measurable and 
$$\{s'\in [0,\tau],\ t'\in [0,\tau],\ s'\le t'\} \ni (s,t) \mapsto A(s,t)$$
is continuous. For $\ell \ge 1$, we call $\bL^\ell$-\emph{stochastic Young integral} of $\psi$ with respect to the pseudo increment $A$ the limit in $\bL^\ell(\tilde \Omega,\tilde {\mathbb P})$
\begin{equation}
\lim\limits_{\substack{\Delta {\rm{\ partition\ of\ [0,\tau]}} \\ |\Delta| \to 0}} \sum_{t_i \in \Delta}\psi_{t_i}A(t_i,t_{i+1}) = : \int_0^\tau \psi_{\textcolor{black}{t}} A(t,t+dt),
\end{equation}
when it exists.
\end{df} 

Having such tools at hand, the strategy consists in 
building simultaneously the Martingale solution and the noise $(X,\mW)$ as the solution of \textcolor{black}{a} kind of enlarged Martingale Problem \textcolor{black}{whose entries for the underlying canonical process} $(X,\mW)$ are, respectively, associated with the solution of the martingale problem $(L^\alpha + F\cdot D)$ and the corresponding driving noise. This will allow to build the drift as the difference between them. Indeed, having such a canonical process at hand, we decompose the increment of the process $X$ as 
$$X_{t+h}-X_t = \{\E[X_{t+h}-X_t|\mathcal F_t]\} + \{X_{t+h}-X_t- \E[X_{t+h}-X_t|\mathcal F_t]\},$$
where $\mathcal F_t := \sigma(X_s,\mW_s,\ 0 \le s \le t)$. Clearly, the first difference in the above right hand side stands for a drift term, while the second stands for a martingale part.
It thus suffices to \textcolor{black}{relate} both parts with (i) the original drift $F$ and (ii) the $\alpha$-stable noise \textcolor{black}{$\mW $} previously built. This is the purpose of the next result proved in Section \ref{SEC_RECON_DYN} (see in particular Subsection \ref{SEC_DYNAMICS} and \textcolor{black}{Proposition \ref{PROP_REG_PARTIELLE} below)}.

\begin{THM}\label{THEO_DYN} Let $p,q,r \ge 1$, $\alpha \in (1,2]$ and $\gamma \in (1/2,1)$ 
satisfy a \emph{good relation for the dynamics} \eqref{good_relation_dyn}. It then holds that there exists a probability measure $\mathbf P^\alpha$ on $\mathcal C([0,T],\R^{2d})$ when $\alpha=2$ and  $\mathcal D([0,T],\R^{2d})$ when $0< \alpha<2$ such that the canonical process, denoted by $(X,\mW)$, satisfies
\begin{itemize}
\item[(i)] The law of $X$ under $\mathbf P^\alpha$ is a solution of the Martingale problem associated with $(F, L^\alpha,x)$,  $x\in \R^{d}$ and the law of $\mW$ under $\mathbf P^\alpha$ is a Brownian motion if $\alpha=2$ and an $\alpha$-stable process if $\alpha<2$.
\item[(ii)] The dynamics of the canonical process reads
\begin{equation}
\label{DYNAMICS}
X_t=x+\int_0^t \mathscr F(s,X_s,ds)+{\mathcal W}_t,\quad \mathbf P^\alpha-a.s.
\end{equation} 
where for any $0\leq v \leq s \leq T$, $x\in \R^d$,
\begin{eqnarray}\label{DECOMP_DRIFT}
\mathscr F(v,x,s-v)&:=&\int_v^s dr \int_{\R^d}dy F(r,y) p_\alpha(r-v,y-x),
\end{eqnarray}
with $p_\alpha$ the (smooth) density of $\mW$ and where the integral in \eqref{DYNAMICS} is understood as a\textcolor{black}{n} $\bL^\ell$-\emph{stochastic Young integral}, {\color{black} for any} $1 \le \ell < \alpha$, in the sense of Definition \ref{Stochastic_young}.
\end{itemize}
\end{THM}  
In fact, we prove a stronger result concerning the dynamics. We show that it is possible to define a stochastic Young integral against the dynamics, leading in turn to use It\^o calculus. This is done for a 
suitable class of progressively measurable processes $\psi$ satisfying appropriate H\"older regularity conditions. 
For any $\mathfrak q'\ge 1$, any $\beta\in (0,1)$, we set
 \begin{eqnarray}\label{def:predproc}
&& \mathcal H_{\mathfrak q'}^{\beta}(\Omega,\mathcal F,(\mathcal F_t)_t,\mathbb P) : = \bigg\{(\psi_t)_{t\in [0,T]} \text{ progressively measurable, } \\
 && \qquad \sup_{t \in [0,T]}\E^{\frac 1{\mathfrak q'}}[|\psi|^{\mathfrak q'}] + \sup_{s\neq t \in [0,T]}\frac{\E^{\frac{1}{\mathfrak q'}}[|\psi_s-\psi_t|^{\mathfrak q'}]}{|t-s|^{\beta}} < +\infty\bigg\}.\notag
\end{eqnarray}
As a corollary of the proof of Theorem  \ref{THEO_DYN}, we have the following result.
\begin{cor}[Associated $\bL^\ell$-stochastic Young integral, $1\leq \ell < \alpha$ ]\label{INTEG_STO} Under the above assumptions, one can define a stochastic Young integral w.r.t. the quantities in \eqref{DYNAMICS}.  Namely, for any $1 \leq \ell,\mathfrak q < \alpha$, for which there exists $\mathfrak q' \ge 1$ satisfying $1/\mathfrak q'+1/\mathfrak q=1/\ell$, one has
\begin{equation}
\int_0^t \psi_s dX_s= \int_0^t \psi_s \mathscr F(s,X_s,ds)+ \int_0^t \psi_s \textcolor{black}{d} {\mathcal W}_s,
\end{equation}
for any $\psi \in  \mathcal H_{\mathfrak q'}^{1-1/\alpha-\varepsilon_2}$, for all $0< \varepsilon_2 < (\theta-1)/\alpha$ and where the first term in the above right hand side is defined as a $\bL^\ell$-stochastic Young integral.
\end{cor}

\bigskip
\noindent\textbf{Further properties and weak formulation.}
The previously described dynamics for the Martingale solution strongly suggests that a notion of weak solution associated with the \emph{formal} SDE \eqref{SDE} can \textcolor{black}{somehow be considered}. This leads to the following definition.

\begin{df}\label{WEAK-DEF} We call weak solution of the \emph{formal} SDE \eqref{SDE} a pair $(Y,\mZ)$ of adapted processes on a filtered probability space $(\Omega,\mathcal F, \{\mathcal F_t\}_{t \ge 0}, \mathbb P)$ such that $\mZ$ is an $\{\mathcal F_t\}_{t \ge 0}$ $\alpha$-stable process and $(Y,\mZ)$ satisfies 
\begin{equation}
\label{DYNAMICS-DEF}
Y_t=x+\int_0^t \mathscr F(s,Y_s,ds)+{\mathcal Z}_t,\quad \P-a.s.,\qquad  \E|\int_0^t \mathscr F(s,Y_s,ds)| <+\infty 
\end{equation} 
for any $t$ in $[0,T]$ and where for any $0\leq v \leq s \leq T$, $x\in \R^d$,
\begin{eqnarray}\label{DECOMP_DRIFT-DEF}
\mathscr F(v,x,s-v)&=&\int_v^s dr \int_{\R^d}dy F(r,y) p_\alpha(r-v,y-x) 
\end{eqnarray}
with $p_\alpha$ the (smooth) density of $\mZ$ and where the integral in \eqref{DYNAMICS-DEF} is understood as a\textcolor{black}{n} $\bL^1$-\emph{stochastic Young integral}, in the sense of Definition \ref{Stochastic_young}.\\

We say that weak uniqueness holds for \eqref{SDE} if for any two weak solutions $(Y,\mZ)$, $(\Omega,\mathcal F, \{\mathcal F_t\}_{t \ge 0}, \mathbb P)$ and $(\tilde Y,\tilde \mZ)$, $(\tilde \Omega, \tilde{\mathcal F}, \{\tilde {\mathcal F}_t\}_{t \ge 0}, \tilde{ \mathbb P})$ with the same initial condition, then $(Y_t)_{t\ge 0}\overset{({\rm law})}{=} (\tilde Y_t)_{t\ge 0}$.
\end{df}
We then have the following well-posedness result whose proof is postponed to Section \ref{futher_prop} (see Subsection \ref{WEAKSOL}) \textcolor{black}{and thoroughly exploits that we can apply some stochastic calculus arguments to the dynamics and the smoothness of the underlying PDE (see Corollary \ref{INTEG_STO} and Theorem \ref{THE_PDE} above)}.

\begin{THM}\label{THEO_STRONG} Let $p,q,r \ge 1$, $\alpha \in (1,2]$ and $\gamma \in (1/2,1)$ 
satisfy a \emph{good relation for the dynamics} \eqref{good_relation_dyn}. Then,
\begin{itemize}
\item[(i)] the \emph{formal} SDE \eqref{SDE} admits a unique weak solution in the sense of Definition \ref{WEAK-DEF};
\item[(ii)] if $d=1$, pathwise uniqueness holds, \emph{i.e.} the paths of two weak solutions defined on the same probability basis $(\Omega,\mathcal F, \{\mathcal F_t\}_{t \ge 0}, \mathbb P, \mathcal Z)$ coincide a.s. whenever they start from the same initial condition.
\end{itemize}
Moreover, one can define an associated $\bL^1$-stochastic Young calculus \emph{i.e.} Corollary \ref{INTEG_STO} hold with $\ell=1$ therein.
\end{THM}
\begin{REM}[About the connections between the Martingale and weak solutions] \label{REM_MEASURABILITY}
\hspace*{.5cm}
\begin{itemize}
\item Although we obtain well-posedness for both the Martingale problem and the weak formulation associated with the \emph{formal} SDE \eqref{SDE}, we do not claim that equivalence holds between them. In fact, what we are able to prove is that weak existence implies the existence of a Martingale solution and that uniqueness for the Martingale problem implies weak uniqueness, see Remark \ref{equivWeakMart}. It is nevertheless more involved to prove that existence of a Martingale solution gives the weak existence and that weak uniqueness implies uniqueness for  the Martingale problem. The main issue relies on the fact that the very definition of the Martingale problem does not allow one to define an associated stochastic calculus, while the very definition of a weak solution does. Without the help of It\^o's formula, we are not able to go from the Martingale solution to the weak one, which prevents us from obtaining the equivalence between both formulations.
\item Pay attention that, in the above result (ii), we do not claim that strong uniqueness holds. This mainly comes from a measurability argument. In \cite{athr:butk:mytn:18}, the Authors built the drift as a Dirichlet process and then recover the noise part of the dynamics as the difference between the solution and the drift allowing them in turn to work under a more standard framework (in term of measurability), and thus to use the Yamada-Watanabe Theorem. Here, we mainly recover the noise in a canonical way, through the martingale problem, and then build the drift as the difference between the solution and the noise. Such a construction allow\textcolor{black}{s} us to give a precise meaning to the drift and the loss of measurability can be seen as the price to pay for it. Nevertheless, at this stage, one may restart with the approach of Athreya \emph{et al.} \cite{athr:butk:mytn:18} to define an ad hoc noise as the difference between the process and the drift (which reads as a Dirichlet process), identify the objects obtained with the two approaches and then obtain suitable measurability conditions to apply \textcolor{black}{the} Yamada-Watanabe Theorem.
\end{itemize}
\end{REM}

We emphasize that the above results give, \textcolor{black}{to the best of our knowledge}, the mo\textcolor{black}{st} accurate description of the dynamics of the  \emph{formal} SDE \eqref{SDE} we found in the literature on SDE with distributional drift. We emphasize as well that such a description includes all \textcolor{black}{those} we found and discussed in the Introduction. A crucial point is that, moreover, the above dynamics (for whenever the Martingale solution or the weak solution) coincide with the classical one when the drift is time-space H\"older continuous. All those facts are collected in the following proposition whose proof is given in Section \ref{futher_prop} (see Subsection \ref{futher_prop_drift}).

\begin{PROP}\label{PROP_DRIFT} Either the Martingale solution or the weak solution of the formal SDE \eqref{SDE} is:
\begin{itemize}
\item[(i)] A virtual solution of the \emph{formal} SDE \eqref{SDE}; 
\item[(ii)] A Dirichlet process;
\item[(iii)] It holds that for any smooth approximating sequence $(F_m)_{m\ge 1}$ such that 
$$\lim_{m\to +\infty}\| F - F_m\|_{\bL^r(\bB^{-1+\gamma}_{p,q})} = 0\footnote{See Remark \ref{APPROX} for the case when $p$ and/or $r$ are/is $+\infty$.},$$
\begin{equation}
\lim_{m \to \infty} \left\| \int_0^t \mathscr F(s,X_s,ds) - \int_0^t F_m(s,X_s)ds \right\|_{\bL^\ell }= 0,\quad 1\le \ell <\alpha,
\end{equation}
with $\ell=1$ for the weak solution;
\item[(iv)] If $F$ is time-space $\beta$-H\"older continuous, the previous construction coincide with the \textcolor{black}{\textit{``usual"}} drift:
$$\int_0^t \mathscr F(s,X_s,ds) = \int_0^t F(s,X_s)ds,\quad a.s..$$
\end{itemize}
\end{PROP}
}
\textcolor{black}{We eventually mention that the previous explicit representation and properties of the drift} could  also be useful in order to derive numerical approximations for the SDE \eqref{DYNAMICS}. We can, to this end, mention the recent work by De Angelis \textit{et al.} \cite{dean:germ:isso:19} who considered in the Brownian scalar case some related issues.\\

\bigskip
\noindent\textbf{Organization of this paper.}
The paper is organized as follows. Section \ref{SEC_PREUVE} is dedicated to the proof \textcolor{black}{of} Theorem  \ref{THEO_WELL_POSED} (well-posedness of the Martingale problem). To do so, we collect some material (i.e. intermediate results) on the PDE, which is used all along this work. These intermediate results are proved in Section \ref{SEC_PDE_PROOF} and allow, in turn, \textcolor{black}{to prove}  Theorem \ref{THE_PDE}. As mentioned before, Section 
\ref{SEC_PDE_PROOF} is dedicated to the PDE analysis. In Section \ref{SEC_RECON_DYN}, we mainly reconstruct the dynamics associated with the Martingale solution. Proof of Theorem \ref{THEO_DYN} can be found therein. Having this dynamics at hand, we then investigate the weak solution of the \emph{formal} SDE \eqref{SDE} and further properties of the obtained drift in Section \ref{futher_prop}. This last \textcolor{black}{section} thus contains the proofs of Theorem \ref{THEO_STRONG} and Proposition \ref{PROP_DRIFT}.


\bigskip
\noindent\textbf{Notations.} \textcolor{black}{Throughout the document}, we denote by $c,c'...$ some positive constants depending \textcolor{black}{on the non-degeneracy constant $\kappa$ in} \A{UE} and on the set of parameters $\{\alpha,p,q,r,\gamma\}$. \textcolor{black}{The notation $C,C'...$ is used when the constants also depend in a non-decreasing way on time $T$. Other possible dependencies are also explicitly indicated. We also sometimes shorten $\bL^r([0,T] ,\bB_{p,q}^{-1+\gamma}\textcolor{black}{(\R^N,\R^M)})$, for $N,M$ in $\mathbb N$, with the notation $\bL^r([0,T] ,\bB_{p,q}^{-1+\gamma})$ or $\bL^r(\bB_{p,q}^{-1+\gamma})$ when there are no ambiguities.}

\section{Well-posedness of the Martingale problem : proof of Theorem \ref{THEO_WELL_POSED}}\label{SEC_PREUVE}
In this Section, we mainly prove Theorem \ref{THEO_WELL_POSED}. To do so, we need some additional material on the PDE, which is collected in Subsection \ref{SEC_PDE} below. This material is of crucial importance in our work as it will be used  to prove Theorems \ref{THE_PDE}, \ref{THEO_DYN}, \ref{THEO_STRONG} and Proposition \ref{PROP_DRIFT} as well. All these PDE results are proved in Section \ref{SEC_PDE_PROOF} below (see Subsection \ref{proof-pde}). The proof of Theorem \ref{THEO_WELL_POSED} is then derived in Subsection \ref{SDE_2_PDE}.

\subsection{The underlying PDE}\label{SEC_PDE}
\textcolor{black}{As underlined  in Definitions \ref{DEF_PDESOL} and \ref{DEF_MPB}, it turns out that the well-posedness of the Martingale Problem associated with $(F,L^\alpha,x)$, $x\in \R^d$, heavily relies on the \textcolor{black}{construction} of a suitable theory for the Cauchy problem $\mathscr C(F,L^\alpha,f,g,T)$ (see Definition \ref{DEF_PDESOL}}) for some data $f$ and $g$  belonging to some \textcolor{black}{appropriate function} spaces to be specified later on.
\textcolor{black}{We recall that, because of the scalar product $F \cdot D u$ \textcolor{black}{therein},} the \textcolor{black}{ aforementioned PDE is
only stated formally. Only the mild formulation, given by \eqref{repsolPDE}, is licit thanks to Lemma \ref{lemme:defprod}.}\\

Hence, as a key intermediate tool we need to introduce what we will later on call the \emph{mollified} Cauchy problem.
Namely, denoting by $(F_m)_{m \in \mathbb N^*}$ a sequence of smooth functions such that $\|F-F_m\|_{\mathbb L^r([0,T], \mathbb B_{p,q}^{-1+\gamma})}\to 0$ when $m\to \infty$,
the \emph{mollified} Cauchy problem $\mathscr C(F_m,L^\alpha,f,g,T)$ reads as
\begin{eqnarray}\label{Asso_PDE_MOLL}
\p_t  u_m(t,x) + L^\alpha u_m(t,x) + F_m(t,x) \cdot D u_m(t,x) &=&{\color{black}f(t,x)},\quad \text{on }[0,T]\times \R^d, \notag \\
u_m(T,x) &=& g(x),\quad \text{on }\R^d.
\end{eqnarray}

{\color{black}
\begin{REM}[Smooth approximating sequence of the drift]\label{APPROX}
When $p,q,r<\infty$, such a sequence can be obtained from  $C_0^\infty([0,T]\times \R^d,\R^d)$ functions (or in the Schwartz class in space), see e.g. Theorem 4.1.3 in \cite{adam:hedb:96}. When $p=\infty$, one can approximate $F$ by sequence of $C_b^\infty([0,T]\times \R^d,\R^d)$ functions in $ \bL^{r'}([0,T], \mathbb B_{p,q}^{-1+\gamma'})$ for any $\gamma'<\gamma$, $ r'=r$ if $r<+\infty$ and $r'<+\infty $ if $r=\infty $  (taking e.g. their convolution with a gaussian kernel with variance $m^{-1}{\rm{Id}}$), see e.g. Definition 2.2 and Lemma 1.3 in \cite{athr:butk:le:mytn:21} again for the spatial part. Up to an abuse of notation, one can denote by $\gamma,r$ these indexes and still use the analysis done below.
\end{REM}
}

\textcolor{black}{We start with the following control which, in some sense, is the counterpart of Theorem \ref{THE_PDE} for the \emph{mollified} Cauchy problem.}

\begin{PROP}\label{PROP_PDE_MOLL} Assume that the parameters $p,q,r$, $\alpha$ and $\gamma$ satisfy a \emph{good relation} \eqref{good_relation}. \textcolor{black}{Let $f,g$ be smooth functions where $g$ has as well at most linear growth}.
Let  $(u_m)_{m \geq 1}$ denote the sequence of classical solutions of the \emph{mollified} PDE \eqref{Asso_PDE_MOLL} i.e. of $\big(\mathscr C(F_m,L^\alpha,f,g,T)\big)_{m \ge 1}$. Then, there exist positive constants  $C:=C(\|F\|_{\bL^r(\bB_{p,q}^{\textcolor{black}{-1+\gamma}})})$, $C_T:=C(T,\|F\|_{\bL^r(\bB_{p,q}^{\textcolor{black}{-1+\gamma}})})$, \textcolor{black}{depending on the known parameters $\gamma,p,q,r $ and $\kappa $ in \A{UE}}, s.t. for all $m\ge 1 $:
\begin{eqnarray}
\forall x \in \R^d,\quad |u_m(t,x)|&\le &C(1+|x|),\notag\\
\|D u_m\|_{\bL^{\infty}\left(\bB^{\theta-1- \varepsilon}_{\infty,\infty}\right)} &\le&  C_T(\|Dg\|_{\bB_{\infty,\infty}^{\theta-1}}+\|f\|_{\textcolor{black}{\bL^\infty}(\bB_{\infty,\infty}^{\theta-\alpha})}),\label{CTR_SCHAUDER_LIKE}\\
 \forall  0\le t\le s\le T,\ x\in \R^d,\ |u_m(t,x)-u_m(s,x)|&\le& C |t-s|^{\frac{\theta}{\alpha}},\notag\\
  |Du_m(t,x)-Du_m(s,x)|&\le& C |t-s|^{\frac{\theta-1}{\alpha}},\notag
\end{eqnarray}
where $\varepsilon <<1$ can be chosen as small as desired, $T \mapsto C_T$ is a non-decreasing function and where, \textcolor{black}{from the definition in \eqref{DEF_THETA}}, $\theta-1=\gamma - 2+\alpha - d/p - \alpha/r >0 $.

{\color{black}Moreover, the sequence $(u_m,Du_m)_{m \ge 1}$ converges toward the solution $(u,D_u)$,  in the sense of Definition \ref{DEF_PDESOL}, of the Cauchy problem $\mathscr C(F,L^\alpha,f,g,T)$ uniformly on compact subsets of $[0,T]\times \R^d$.}
\end{PROP}

\begin{REM}\label{REM_SHAU_EST_MOLL}
Let us mention that, when the terminal condition $g$ is bounded, then the solution $u_m$ is itself bounded, i.e. $\sup_{m \ge 1}|u_m(t,x)|\le C $. 
\end{REM}

\begin{REM}[On the spatial smoothness of the mollified PDE] \textcolor{black}{From the conditions on $\gamma,\alpha $ and the definition of $\theta $ in \eqref{DEF_THETA}}, we carefully point out that:
\begin{equation*}
\theta=\gamma - 1+\alpha - \frac dp - \frac{\alpha}r>1.
\end{equation*}
\textcolor{black}{This} reflects the spatial 
smoothness of the underlying PDE. In particular, the condition $\theta>1 $ provides a pointwise gradient estimate for the solution of the mollified PDE. This key condition rewrites:
 $\theta>1\iff \gamma-2+\alpha-[d/p]-[\alpha/r]> 0 $. It will be implied assuming that $\gamma>[3-\alpha+d/p+\alpha/r]/2 $, since in this case $[3-\alpha+d/p+\alpha/r]/2 -2+\alpha-[d/p]-[\alpha/r]> 0 \iff \alpha >[1+d/p]/[1-1/r]$. 
\end{REM}
\begin{REM}[On the corresponding parabolic bootstrap]\label{REM_DIFF_PRELI}
Observe that, when $p=r=+\infty $, we almost have a Schauder type result, namely $\theta=\gamma-1+\alpha$ in \eqref{CTR_SCHAUDER_LIKE} and we end up with the corresponding parabolic bootstrap effect for both the solution of $\mathscr C(F_m,L^\alpha,f,g,T)$ and $\mathscr C(F,L^\alpha,f,g,T)$,
up to the small exponent $\varepsilon $ which can be chosen arbitrarily small. 
\end{REM}

\begin{REM}[About additional diffusion coefficients] It should be noted at this point that we are confident about the extension of the results to differential operators $L^\alpha$ involving non-trivial diffusion coefficient, provided this \textcolor{black}{latter} is H\"older-continuous in space. Sketches of proofs in this direction are given in the Remark \ref{REM_COEFF_DIFF} in Subsection \ref{proof-pde}.
However, we avoid investigating this direction for the sake of clarity and in order to focus on the more (unusual) drift component.
\end{REM}

\textcolor{black}{The following Proposition and Corollary provide a Zvonkin type theory for the \emph{mollified} and \emph{formal} Cauchy problem given respectively by $\mathscr C(F,L^\alpha,\textcolor{black}{-F^k_m},0,T)$ and $\mathscr C(F,L^\alpha,\textcolor{black}{-F^k},0,T)$,
 where for any $k$ in $\{1,\ldots,d\}$, $F^k$ denotes the $k^{\rm{th}}$ component of $F$ and $(F_m^k)_{m\ge 1}$ denotes its mollification, see Remark \ref{APPROX}.}
\begin{PROP}[Zvonkin type theory for the \emph{mollified} PDE]\label{COR_ZVON_THEO}
Let $p,q,r$, $\alpha$ and $\gamma$ satisfy a \emph{good relation} \eqref{good_relation} and let $k$ in $\{1,\ldots,d\}$. 
 \textcolor{black}{There exists a positive constant  $C_T:=C(T,\|F\|_{\bL^r(\bB_{p,q}^{-1+\gamma})})$ s.t. for each $k$ and all $m\ge 1$, the sequence of classical solutions $(u_m^k)_{m \geq 1}$ of $\big(\mathscr C(F_m,L^\alpha,F_m^k,0,T)\big)_{m\ge1}$
 satisfies}: 
\begin{eqnarray*}
\|u_m^k\|_{\bL^\infty(\bL^{\infty})}+ \|D u_m^k\|_{\bL^{\infty}\left(\bB^{\theta-1- \varepsilon}_{\infty,\infty}\right)} \le  C_T,
\end{eqnarray*}
where $T \mapsto C_T$ is a non-decreasing function and for which the two last lines in \eqref{CTR_SCHAUDER_LIKE} hold as well. {\color{black}Moreover, for any $k$ in $\{1,\ldots,d\}$, the sequence $(u_m^k,Du_m^k)_{m \ge 1}$ converges towards the solution $(u^k,Du^k)$,in the sense of Definition \ref{DEF_PDESOL}, of the Cauchy problem $\mathscr C(F,L^\alpha,-F^k,0,T)$, uniformly on compact subsets of $[0,T]\times \R^d$.}
\end{PROP}

\textcolor{black}{
\begin{cor}[Zvonkin type theory for the \emph{formal} PDE]\label{COR_ZVON_THEO_UNIF}
Let $k \in \{1,\ldots,d\}$. Under the above assumptions, the \emph{formal} Cauchy problem $\mathscr C(F,L^\alpha,-F^k,0,T)$
admits a unique solution $u^k$ in the sense of Definition \ref{DEF_PDESOL} which moreover satisfies that  there exists a positive constant  $C_T:=C(T,\|F\|_{\bL^r(\bB_{p,q}^{-1+\gamma})})$ s.t. 
\begin{eqnarray*}
\|u^k\|_{\bL^\infty(\bL^{\infty})}+ \|D u^k\|_{\bL^{\infty}\left(\bB^{\theta-1- \varepsilon}_{\infty,\infty}\right)} \le  C_T,
\end{eqnarray*}
where $T \mapsto C_T$ is a non-decreasing function.
\end{cor}
}
\begin{REM}
Of course, \textcolor{black}{in order to use the Zvonkin type theory} to derive strong well-posedness in the multidimensional setting some controls of the second order derivatives are needed. This is what \textcolor{black}{Krylov and R\"ockner} did in \cite{krylov_strong_2005} in the Sobolev setting.  Let us also specify that, in connection with Theorem \ref{THEO_STRONG}, in the scalar setting weak and strong uniqueness are somehow closer since, from the PDE viewpoint, they do not require to go up to second order derivatives. Indeed, the strategy is then to develop for two weak solutions {\color{black}$(X^1,\mW), (X^2,\mW)  $} of \eqref{DYNAMICS}, a regularized version of $|X_t^1-X_t^2|$, which somehow makes appear a kind of ``\textit{local-time}" term which is handled through the H\"older controls on the gradients (see the proof of Theorem  \ref{THEO_STRONG}-(ii) in Subsection \ref{WEAKSOL}-(ii) and e.g. Proposition 2.9 in [ABP18]), whereas in the multidimensional setting, for strong uniqueness, the second derivatives get in.  
\end{REM}

\subsection{From PDE to SDE results: proof of Theorem \ref{THEO_WELL_POSED}}\label{SDE_2_PDE}

It is quite standard to derive well-posedness results for a probabilistic problem through PDE estimates. When the drift is a function, such a strategy goes back to e.g. Zvonkin \cite{zvonkin_transformation_1974} or Stroock and Varadhan \cite{stro:vara:79}. 
This approach has been made quite systematic in the distributional setting by Delarue and Diel in \cite{dela:diel:16} who provide a very robust framework. \textcolor{black}{To investigate the meaning and well-posedness of \eqref{SDE}, we adapt their procedure to the current setting}. Points \emph{(i)} to \emph{(iii)} allow to derive the rigorous proof of Theorem \ref{THEO_WELL_POSED} provided Proposition \ref{PROP_PDE_MOLL}, Proposition \ref{COR_ZVON_THEO}, Corollary \ref{COR_ZVON_THEO_UNIF}  and Theorem \ref{THE_PDE} hold. 

\noindent\emph{(i) Tightness of the sequence of probability measure induced by the solution of the mollified SDE \eqref{SDE}.} Here, we consider the regular framework induced by the \emph{mollified} PDE \eqref{Asso_PDE_MOLL}. Note that in this regularized framework, for any $m$,  the Martingale Problem associated with {\color{black} $(F_m,L^\alpha,x)$, $x\in \R^d$,} is well posed. We denote by $\mathbb{P}^\alpha_m$ the associated solution. Let us generically denote by $(X_t^m)_{t\ge 0} $ the associated canonical process. Note that the underlying space where such a process is defined differs  according to the values of $\alpha$: when $\alpha=2$ the underlying space is $\mathcal C([0,T],\R^d) $ while it is  $\mathcal D([0,T],\R^d)$ when $\alpha<2$. Recall that we denoted it by $\Omega_\alpha$.

Let $u_m=(u^1_m,\ldots,u^d_m)$ where each $\textcolor{black}{u_m^k}$, $k \in \{1,\ldots,d\}$ is the solution of {\color{black} the mollified Cauchy problem $\mathscr C(F_m,L^\alpha,-F^{{\color{black}k}}_m,0,T)$.} 
Let us define for any $s\geq v$ in $ [0,T]^2$ and for any $\alpha \in (1,2]$ the process 
\begin{equation}\label{def_de_m_alpha}
M_{v,s}(\alpha,u_m,X^m)= \left\lbrace\begin{array}{llll}
\displaystyle \int_v^s D u_m(r,X_r^m)\cdot dW_r,\\
\text{ where } W\text{ is a Brownian motion, if }\alpha=2;\\ 
\displaystyle \int_v^s \int_{\R^d \backslash\{0\} } \{u_m(r,X_{r^-}^{\textcolor{black}{m}}+x) - u_m(r,X_{r^-}^{\textcolor{black}{m}}) \}\tilde N(dr,dx), \\
\text{ where  }\ \tilde N\text{ is \textcolor{black}{the} compensated Poisson measure}, \\
\text{if }\alpha<2.
\end{array}
\right.
\end{equation}
Note that this process makes sense for any $m\ge 1$, thanks to regularity estimates on $u_m$ from Proposition \ref{COR_ZVON_THEO}.
Next, applying It\^o's formula  to $\textcolor{black}{(X_r^m + u_m(r,X_r^m))_{r\in [v,s]}}$  we obtain
\begin{equation}\label{rep_canoproc_PDE}
X_s^m-X_v^m=M_{v,s}(\alpha,u_m,X^m)+ \mW_s-\mW_v {\color{black} +} [u_m(v,X_v^m)-u_m(s,X_s^m)].
\end{equation}
In order to prove that $(\P^\alpha_m)_{m\in \N^*} $ actually forms a tight sequence of probability measures on \textcolor{black}{$\Omega_\alpha$}, whose limit is denoted by $\P^\alpha$,
it is sufficient to prove that there exists $c,\tilde p$ and $\eta>0$ such that $\E^{\P^{\textcolor{black}{2}}_{\textcolor{black}{m}}}[|X^m_s - X^m_v|^{\tilde p}] \leq c|v-s|^{1+\eta}$ \textcolor{black}{or} $\E^{\P^\alpha_{\textcolor{black}{m}}}[|X^m_s - X^m_0|^{\tilde p}] \leq cs^{\eta}$\textcolor{black}{, for $\alpha \in (1,2)$} thanks to the Kolmogorov (resp. Aldous) Criterion. We refer e.g. for the latter to Proposition 34.9 in Bass \cite{bass:11}. Writing
\begin{eqnarray*}
&&[u_m(v,X_v^m)-u_m(s,X_s^m)] \\
&=& u_m(v,X_v^m)-u_m(v,X_s^m) +u_m(v,X_s^m)-u_m(s,X_s^m),
\end{eqnarray*}
the result follows in small time, thanks to {\color{black} Proposition} \ref{COR_ZVON_THEO} (choosing $1< \tilde p<\alpha$ in the pure jump setting).\\

\noindent\emph{(ii) Identification of the limit probability measure.} Let us now prove that the limit, denoted by $\mathbb P^\alpha$, is indeed a solution of the martingale problem associated with {\color{black}$(F,L^\alpha,x)$, $x \in \R^d$, in the sense of Definition \ref{DEF_MPB}}. Let $f\in C^0([0,T],\mathcal S(\R^d))$ {\color{black} and $g$ be a continuous function with gradient in $\bB_{\textcolor{black}{\infty,\infty}}^{\theta-1}(\R^d,\R^d)$ and let for any $m\ge1$ $u_m$ be the classical solution of the
Cauchy problem $\mathscr C(F_m,L^\alpha,f,g,T)$.
}
Applying It\^o's formula for each $u_m(t,X_t^m)$ we obtain that 
\begin{eqnarray*}
u_m(t,X_t^m) - u_m(0,x_0) - \int_0^t f(s,X_s^m) d s  =  M_{\textcolor{black}{0},t}(\alpha,u_m,X^m),  
\end{eqnarray*} 
where $M(\alpha,u_m,X^m)$ is defined by \eqref{def_de_m_alpha}. 
%
Thus, using the convergence result of $(u_m,Du_m)_{m \ge 1}$ to the solution $(u,Du)$ of $\mathscr C(F,L^\alpha,f,g,T)$ on every compact subsets of $[0,T] \times \R^d$ from Proposition \ref{PROP_PDE_MOLL} together with a uniform control of the moment of $X^m$ (which also follows from \eqref{rep_canoproc_PDE} and above conditions on $u_m$), we deduce that
\begin{equation}\label{eq:mgprop}
\left(u(t,X_t) \textcolor{black}{-} \int_0^t f(s,X_s) d s - u(0,x) \right)_{0\leq t \leq T},
\end{equation}
is a  $\mathbb{P}^\alpha$-martingale (\textcolor{black}{square integrable when $\alpha=2 $}) by letting the regularization procedure tend to infinity.\\

\noindent\emph{(iii) Uniqueness of the limit probability measure.} We now come back to the canonical space {\color{black} $\Omega_\alpha$,}
 and let  $\mathbb{P}^\alpha$ and $\tilde{\mathbb{P}}^\alpha$ be two solutions of the Martingale Problem associated with {\color{black} $(F,L^\alpha,x)$, $x \in \R^d$}.
Thus, for all $f \in C^0([0,T],\mathcal S(\R^d))$ we have, 
{\color{black} for the solution $u$ of the Cauchy problem $\mathscr C(F,L^\alpha,f,0,T)$}
\begin{equation*}
{\color{black} -}u(0,x) =  \E^{\mathbb{P}^\alpha}\left[ \int_0^T f(s,X_s) d s\right] = \E^{\tilde{\mathbb{P}}^\alpha}\left[ \int_0^T f(s,X_s) d s\right],
\end{equation*}
so that the marginal law\textcolor{black}{s} of the canonical process are the same under $\mathbb{P}^\alpha$ and $\tilde{\mathbb{P}}^\alpha$. We extend the result on  $\R_+$ thanks to regular conditional probabilities, see  Chapter 6.2 in \cite{stro:vara:79} . Uniqueness then follows from Corollary 6.2.4 of  \cite{stro:vara:79}. {\color{black} The strong Markov property follows from Theorem 4.2 in \cite{EK:86}.}\\

\section{PDE analysis}\label{SEC_PDE_PROOF}
{\color{black}This part is dedicated to the proofs of Proposition  \ref{PROP_PDE_MOLL}, Proposition \ref{COR_ZVON_THEO}, Theorem \ref{THE_PDE} as well as Corollary \ref{COR_ZVON_THEO_UNIF} and Lemma \ref{lemme:defprod}.} It is thus the core of this paper as these results allow to recover, \textcolor{black}{specify} and extend, most of the previous results on SDEs with distributional drifts \textcolor{black}{discussed} in the introduction. Especially, as they are handled, the proofs are essentially the same in the diffusive ($\alpha=2$) and pure jump ($\alpha<2$) setting as they only require heat kernel type estimates on the density of the associated underlying noise. We first start by introducing the mathematical tools in Subsection \ref{SEC_MATH_TOOLS}. Then, we provide a primer on the \emph{formal} Cauchy problem $\mathscr C(F,L^\alpha,f,g,T)$ 
by investigating the smoothing properties of the Green kernel associated with the stable noise in Subsection \ref{SEC_GREEN}. Uniform (w.r.t. the mollification) estimates of the solution of the Cauchy problem $\mathscr C(F_m,L^\alpha,f,g,T)$ are investigated in Subsection \ref{sec_cor_pde}. Eventually, we derive in Subsection \ref{proof-pde} {\color{black}the proofs of Proposition  \ref{PROP_PDE_MOLL}, Proposition \ref{COR_ZVON_THEO}, Theorem \ref{THE_PDE} as well as Corollary \ref{COR_ZVON_THEO_UNIF} and Lemma \ref{lemme:defprod}}. \textcolor{black}{We importantly point out that, from now on and in all the current section, we assume without loss of generality that $T\le 1$.}

\subsection{Mathematical tools}\label{SEC_MATH_TOOLS}
In this part, we give the main mathematical tools needed to prove Proposition \ref{PROP_PDE_MOLL} \textcolor{black}{and} Theorem \ref{THE_PDE}. 

\paragraph{Heat kernel estimates for the density of the driving process}
Under \A{UE}, it is rather well known that the following properties hold for the density $p_\alpha $ of $\mW $. For the sake of completeness we provide a complete proof.
\begin{lem}[Bounds and Sensitivities for the stable density]
\label{SENS_SING_STAB}
There exists $C:=C(\A{\textcolor{black}{UE}})$ s.t. for all $\ell \in \{1,2\} $, $t>0$, and $ y\in \R^d $:
\begin{equation} 
\label{SENSI_STABLE}
|D_y^\ell p_\alpha(t,y)|\le \frac{C}{t^{\ell/\alpha}} q_\alpha(t,y),\ |\partial_t^\ell p_\alpha(t,y)|\le \frac{C}{t^{\ell}} q_\alpha(t,y),
\end{equation}
where $\big( q_\alpha(t,\cdot)\big)_{t>0}$ is a  family of probability densities on $\R^d $ such that $q_\alpha(t,y) = {t^{-d/ \alpha}} \, q_\alpha (1, t^{- 1/\alpha} y)$, $ t>0$, $ \in \R^d$ and  for all $\gamma \in [0,\alpha) $, there exists a constant $c:=c(\alpha,\eta,\gamma)$ s.t. 
\begin{equation}
\label{INT_Q}
\int_{\R^N}q_\alpha(t,y)|y|^\gamma dy
\le C_{\gamma}t^{\frac{\gamma}{\alpha}},\;\;\; t>0.
\end{equation}

\end{lem}
 \begin{REM}
\label{NOTATION_Q} 
{
From now on, for the family of stable densities $\big(q(t,\cdot)\big)_{t>0} $, 
we also use the notation  $q(\cdot):=q(1,\cdot) $,  i.e. without any specified argument $q(\cdot)$ stands for the density $q\textcolor{black}{(t,\cdot)}$ at time $\textcolor{black}{t=1}$.}
\end{REM} 
\begin{proof}
\textcolor{black}{We focus here on the pure jump case $\alpha\in (1,2) $. Indeed, for $\alpha=2 $ the density of the driving Brownian motion readily satisfies the controls of \eqref{SENSI_STABLE} with $q_\alpha $ replaced by a suitable Gaussian density}.

Let us recall that, for a given fixed $t>0$, we can use an It\^o-L\'evy  decomposition
 at the associated  characteristic stable time scale for $\mW $ (i.e. the truncation is performed at the threshold $t^{\frac {1} {\alpha}} $) 
to write $\mW_t:=M_t+N_t$
where $M_t$ and $N_t $ are independent random variables. 
More precisely, 
 \begin{equation} \label{dec}
 N_s = \int_0^s \int_{ |x| > t^{\frac {1} {\alpha}} }
\; x  N(du,dx), \;\;\; \; M_s = \mW_s - N_s, \;\; s \ge 0,
 \end{equation} 
where $N$ is the  Poisson random measure associated with the process $\mW$; for the considered fixed $t>0$,
 $M_t$ and $N_t$ correspond to
 the \textit{small jumps part } and
\textit{large jumps part} respectively. 
A similar decomposition has been already used in
 \cite{wata:07},        \cite{szto:10} 
 and \cite{huan:meno:15}, \cite{huan:meno:prio:19} (see in particular Lemma 4.3 therein). It is useful to note that the cutting threshold in \eqref{dec} precisely yields for the considered $t>0$ that:
\begin{equation} \label{ind}
N_t  \overset{({\rm law})}{=} t^{\frac 1\alpha} N_1 \;\; \text{and} \;\;
M_t  \overset{({\rm law})}{=} t^{\frac 1\alpha} M_1.
\end{equation}  
To check the assertion about $N$ we start with 
$$
\E [e^{i \textcolor{black}{( \lambda \cdot N_t )}}] = 
\exp \Big(  t
\int_{{\mathbb S}^{d-1}} \int_{t^{\frac 1\alpha}}^{\infty}
 \Big(\cos (\textcolor{black}{ \lambda \cdot (r\xi)})  - 1  \Big) \, \frac{dr}{r^{1+\alpha}} \mu_{S}(d\xi) \Big), \;\; \lambda \in \R^d
$$
(see  \cite{sato:99}). Changing variable $r/t^{1/\alpha} =s$
we get that $\E [e^{i \langle \lambda , N_t \rangle}]$ $= \E [e^{i \langle \lambda , t^{ 1/\alpha} N_1 \rangle}]$ for any $\lambda \in \R^d$ and this shows the assertion (similarly we get
the statement for $M$).
The density of $\mW_t$ then writes
\begin{equation}
\label{DECOMP_G_P}
p_\alpha(t,x)=\int_{\R^d} p_{M}(t,x-\xi)P_{N_t}(d\xi),
\end{equation}
where $p_M(t,\cdot)$ corresponds to the density of $M_t$ and $P_{N_t}$ stands for the law of $N_t$. 
{From Lemma A.2 in \cite{huan:meno:prio:19} (see as well
Lemma B.1  }
in \cite{huan:meno:15}),  $p_M(t,\cdot)$  belongs to the Schwartz class ${\mathscr S}(\R^N) $ and satisfies that for all $m\ge 1 $ and all $\ell \in \{0,1,2\} $, there exist constants  $\bar C_m,\ C_{m}$ s.t. for all $t>0,\ x\in  \R^d  $:
\begin{equation}
\label{CTR_DER_M}
|D_x^\ell p_M(t,x)|\le \frac{\bar C_{m}}{t^{\frac{\ell}{\alpha} }} \, p_{\bar M}(t,x),\;\; \text{where} \;\; p_{\bar M}(t,x)
:=
\frac{C_{m}}{t^{\frac{d}{\alpha}}} \left( 1+ \frac{|x|}{t^{\frac 1\alpha}}\right)^{-m}
\end{equation}
where $C_m$ is chosen in order that {\it $p_{\bar M}(t,\cdot ) $ be a probability density.}


We carefully point out that, to establish the indicated results, since we are led to consider potentially singular spherical measures,  we only focus on integrability properties similarly to \cite{huan:meno:prio:19} and not on pointwise density estimates as for instance in \cite{huan:meno:15}. The main idea thus consists in exploiting 
{\eqref{dec},}  \eqref{DECOMP_G_P} and \eqref{CTR_DER_M}.
The derivatives on which we want to obtain quantitative bounds  will be expressed through derivatives of $p_M(t,\cdot)$, which also give the corresponding time singularities. However, as for general stable processes, the integrability restrictions come from the large jumps (here $N_t $) and only depend on its index $\alpha$.
A crucial point then consists in observing  that the convolution $\int_{\R^d}p_{\bar M}(t,x-\xi)P_{N_t}(d\xi) $ actually corresponds to the density of the random variable 
\begin {equation} \label{we2}
\bar \mW_t:=\bar M_t+N_t,\;\; t>0 
\end{equation}
 (where $\bar M_t $ has density $p_{\bar M}(t,.)$ and is independent of $N_t $; 
 {to have such decomposition one can define each $\bar \mW_t$ on a product probability space}). Then, the integrability properties of $\bar M_t+N_t $, and more generally of all random variables appearing below, come from those of $\bar M_t $ and $N_t$. 

One can easily check that $p_{\bar M}(t,x) = {t^{-\frac d\alpha}} \, p_{\bar M} (1, t^{-\frac 1\alpha} x),$ $ t>0, \, $ $x \in \R^d.$  Hence 
$$
\bar M_t  \overset{({\rm law})}{=} t^{\frac 1\alpha} \bar M_1,\;\;\; N_t  \overset{({\rm law})}{=} t^{\frac 1\alpha}  N_1.
$$
By independence of $\bar M_t$ and $N_t$, using the Fourier transform, one can easily prove that 
\begin{equation} \label{ser1}
\bar \mW_t  \overset{({\rm law})}{=} t^{\frac 1\alpha} \bar \mW_1.
\end{equation} 
Moreover, 
$
\E[|\bar \mW_t|^\gamma]=\E[|\bar M_t+N_t|^\gamma]\le C_\gamma t^{\frac\gamma \alpha}(\E[|\bar M_1|^\gamma]+\E[| N_1|^\gamma])\le C_\gamma t^{\frac\gamma \alpha}, \; \gamma \in (0,\alpha).
$ 
This shows that the density of $\bar \mW_t$ verifies \eqref{INT_Q}. The controls on the spatial derivatives are derived similarly using 
\eqref{CTR_DER_M} for $\ell\in \{1,2\} $ and the same previous argument. The bound for the time derivatives follow from the Kolmogorov equation $\partial_t p_\alpha(t,z)=L^\alpha p_\alpha (t,z)  $ and \eqref{DECOMP_G_P} using the fact that for all $x\in \R^d,\ |L^\alpha p_M(t,x)|\le C_m t^{-1 }\bar p_M(t,x) $ (see again Lemma 4.3 in \cite{huan:meno:prio:19} for details).
\end{proof}

\paragraph{Thermic characterization of Besov norm}\label{SEC_BESOV}{\color{black}
In the sequel, we will intensively use the thermic characterisation of Besov spaces, see e.g.  Section 2.6.4 of Triebel \cite{trie:83}. The \emph{thermic} terminology   comes from the fact that such a norm \textcolor{black}{involves convolution with a suitable heat kernel}. 

In the following, we denote by ${\mathcal S}(\R^d) $ the Schwartz class. For $f \in {\mathcal S}'(\R^d) $, $\varphi \in C_0^\infty(\R^d)$ (smooth function with compact support) s.t. $\varphi(0)\neq 0 $ we set $\varphi(D)f := (\varphi \hat f)^{\vee} $ where $\hat f$ and $(\varphi \hat f)^\vee $ respectively denote the Fourier transform of $f$ and the inverse  Fourier transform of $\varphi \hat f $. 

The thermic characterisation of Besov spaces \textcolor{black}{we will use in the current work} reads as follows: for $\vartheta \in \R, \textcolor{black}{m}\in (0,+\infty] ,\textcolor{black}{l} \in (0,\infty] $, $\bB_{\textcolor{black}{l},\textcolor{black}{m}}^\vartheta(\R^d,\R):=\{f\in {\mathcal S}'(\R^d): \|f\|_{{\mathcal H}_{\textcolor{black}{l},\textcolor{black}{m}}^\vartheta, \tilde \alpha}<+\infty \} $, for any $\tilde \alpha \in [1,2]$, with

\begin{eqnarray}
\label{strong_THERMIC_CAR_DEF_STAB}
\|f\|_{{\mathcal H}_{\textcolor{black}{l},\textcolor{black}{m}}^\vartheta,\tilde{\alpha}}&:=&\|\varphi(D) f\|_{\bL^{\textcolor{black}{l}}(\R^d)}+ \Big(\int_0^1 \frac {dv}{v} v^{(n-\frac \vartheta{\tilde{\alpha}} )\textcolor{black}{m}}    \|\partial_v^n \tilde p_\alpha (v,\cdot)\star f\|_{\bL^{\textcolor{black}{l}}(\R^d)}^{\textcolor{black}{m}} \Big)^{\frac 1{\textcolor{black}{m}}}\\
&=:&\textcolor{black}{\|\varphi(D) f\|_{\bL^{\textcolor{black}{l}}(\R^d)}+\mathcal{T}_{\textcolor{black}{l},\textcolor{black}{m}}^{\vartheta}[f]},\notag
\end{eqnarray}
\textcolor{black}{where in the above definition $``\star $'' stands for the spatial convolution}.

The parameter $n$ is an integer s.t. $n> \vartheta/ \tilde \alpha $ and $\tilde{p}_{\tilde \alpha}$ denotes the isotropic $\tilde \alpha$ stable heat kernel on $\R^d$ (or the gaussian heat kernel if $\tilde \alpha=2$). {\color{black}In Section 2.6.4. of \cite{trie:83}, the thermic characterization is presented either with the Gaussian heat kernel ($ \tilde \alpha=2$) or with the Cauchy-Poisson kernel ($\tilde \alpha=1 $). It actually turns out from the main characterization of Besov spaces, see Section 2.5.1 in \cite{trie:83}, that many kernels can actually be used. We chose the isotropic stable one, with stability index associated with the one of the driving noise in \eqref{SDE}, \emph{i.e.} $\tilde \alpha=\alpha$, since it precisely allows to benefit from the stability by convolution when the underlying stable semi-group, around which we perform the Duhamel expansion (see \eqref{DUHAMEL}, \eqref{DEF_GREEN}), is involved. This is particularly well adapted to the computations in the proof of Lemma  \ref{LEM_BES_NORM}, see Appendix \ref{SEC_APP_TEC}. In particular $\tilde p_{\alpha}$ satisfies the bounds of Lemma \ref{SENS_SING_STAB} and in that case the upper-bounding density can be specified.  Namely, in that case \eqref{SENSI_STABLE} holds with $q_\alpha(t,x)=C_\alpha t^{-d/\alpha}(1+|x|/t^{1/\alpha})^{-(d+\alpha)}$.\\

\textcolor{black}{In the following we call \emph{thermic part} the second term in the right hand side of \eqref{strong_THERMIC_CAR_DEF_STAB} denoted by $\mathcal{T}_{\textcolor{black}{l},\textcolor{black}{m}}^{\vartheta}[f]$.}
}
}

Importantly, it is  well known that $\bB_{\textcolor{black}{l},\textcolor{black}{m}}^\vartheta(\R^d,\R)$ and $\bB_{\textcolor{black}{l}',\textcolor{black}{m}'}^{-\vartheta}(\R^d,\R)$ where $\textcolor{black}{l}',\textcolor{black}{m}' $ are the conjugates of $\textcolor{black}{l},\textcolor{black}{m}$ \textcolor{black}{can be put} in duality. 
\textcolor{black}{Namely,  see e.g.  Theorem 4.1.3 in \cite{adam:hedb:96} or \textcolor{black}{Proposition 3.6 in \cite{lemar:02}},  for $(\textcolor{black}{l},\textcolor{black}{m})\in \textcolor{black}{[}1,\infty]^2 $ }
\textcolor{black}{ and for $(f,g)\in \bB_{l,\textcolor{black}{m}}^{\vartheta}(\R^d,\R)\times \bB_{\textcolor{black}{l}',\textcolor{black}{m}'}^{-\vartheta}(\R^d,\R) $ which are also functions:}
\begin{equation}
\label{EQ_DUALITY}
|\int_{\R^d} f(y) g(y) dy|\le \|f\|_{\bB_{\textcolor{black}{l},\textcolor{black}{m}}^\vartheta}\|g\|_{\bB_{\textcolor{black}{l}',\textcolor{black}{m}'}^{-\vartheta}}. 
\end{equation}


\begin{REM}\label{REM_PART_NONTHERMIC}
As it will be clear in the following, the first part of the r.h.s. in \eqref{strong_THERMIC_CAR_DEF_STAB} will be the easiest part to handle (in our case) and will give a negligible contributions. For that reason, we will only focus on the estimation of the \emph{thermic part} of the Besov norm below. See Remark \ref{GESTION_BESOV_FIRST} in the proof of Lemma \ref{LEM_BES_NORM} in Appendix \ref{SEC_APP_TEC} \textcolor{black}{for details}.
\end{REM}

{\color{black}
\begin{REM} One may wonder why we chose to work with the Thermic characterization of Besov spaces instead of the one \textcolor{black}{deriving from the} dyadic Littlewood-Paley decomposition. Note first of all that such characterizations are equivalent, we can e.g. refer to Section 2.5.1 of \cite{trie:83} (main Theorem) or  Chapters 3 and 5 in \cite{lemar:02}. The main point here is that we are led to handle \textcolor{black}{a} mild (or Duhamel) formulation of a parabolic PDE which itself involves convolutions of distribution\textcolor{black}{s} by \textcolor{black}{a} heat kernel (associated with the driving noise). \textcolor{black}{Such a framework hence naturally leads to consider the thermic characterization and to choose therein a heat kernel which is also compatible with the one of the driving noise.}
Let us mention that the Littlewood-Paley was in this same context of Schauder type estimates successfully used by Zhang and his co-authors, see e.g. \cite{hao:wu:zhan:20} in the degenerate non-local kinetic case. The thermic approach seems more natural and direct to us for our goal.
\end{REM}
}


\paragraph{Auxiliary estimates}
We here provide some useful estimates whose proof\textcolor{black}{s are} postponed to Appendix \ref{SEC_APP_TEC}. We refer to the next Section \textcolor{black}{\ref{SEC_GREEN}} for a flavor of \textcolor{black}{those} proof\textcolor{black}{s} as well as for application\textcolor{black}{s} of such result\textcolor{black}{s}.

\begin{lem}\label{LEM_BES_NORM} Let $\Psi : [0,T]\times \R^d \to \R^d$. Assume that for all $s$ in $[0,T]$ the map $y \mapsto \Psi(s,y)$ is in $\mathbb B_{\infty,\infty}^\beta(\R^d)$ for some $\beta \in \textcolor{black}{(0,1]}$. Define for any $\alpha$ in $(1,2]$, for all $\eta \in \{0,1,\alpha\}$, the  \textcolor{black}{spatial operator} $\mathscr D^\eta$ by
\begin{equation}
\textcolor{black}{{\mathscr D}^\eta := 
\left\lbrace\begin{array}{lll}
{\rm Id} \text{ if }\eta =0,\\
\nabla \text{ if }\eta =1,\\
L^\alpha \text{ if }\eta =\alpha,
\end{array}\right.
}
\end{equation}
and let $p_\alpha(t,\cdot)$ be the density of $\mW_t$ defined in \eqref{DECOMP_G_P}. Then, there exists a constant $C := C(\A{UE},T)>0$ such that for any $\gamma$ in $(1-\beta,1)$, any $p',q'\ge 1$, all $t<s$ in $[0,T]^2$, for all $x$ in $\R^d$
\begin{equation}\label{Esti_BES_NORM}
\| \Psi(s,\cdot) \mathscr D^\eta p_\alpha(s-t,\cdot-x) \|_{\bB_{p',q'}^{\textcolor{black}{1-\gamma}}} \leq \|\Psi(s,\cdot)\|_{\bB_{\infty,\infty}^\beta} \frac{C}{(s-t)^{\left[\frac{1-\gamma}{\alpha}+\frac d{p\alpha}+\frac {\eta}{\alpha}\right]}},
\end{equation}
where $p$ is the conjugate of $p'$. Also, for any $\gamma$ in $(1-\beta,1]$ all $t<s$ in $[0,T]^2$, for all $x,x'$ in $\R^d$ it holds \textcolor{black}{that for all $\beta'\in (0,1) $},
\begin{eqnarray}\label{Esti_BES_HOLD}
&&\| \Psi(s,\cdot) \big(\mathscr D^\eta p_\alpha(s-t,\cdot-x) - \mathscr D^\eta p_\alpha(s-t,\cdot-x')\big) \|_{\bB_{p',q'}^{\textcolor{black}{1-\gamma}}} \notag\\
&\leq& \|\Psi(s,\cdot)\|_{ \bB_{\infty,\infty}^\beta}\frac{C}{(s-t)^{\left[\frac{1-\gamma}{\alpha}+\frac d{p\alpha}+\frac {\eta+\beta'}{\alpha}\right]}}|x-x'|^{\beta'},
\end{eqnarray}
up to a modification of $C\textcolor{black}{:=C(\A{UE},T,\beta')}$.
\end{lem}

\subsection{A primer on PDE the \emph{formal} Cauchy problem $\mathscr C(F,L^\alpha,f,g,T)$: 
reading almost optimal regularity through Green kernel estimates}
\label{SEC_GREEN}
Equation \eqref{Asso_PDE_PMG} can be, still formally, rewritten as
\begin{eqnarray}
\p_t  u(t,x) + L^\alpha u(t,x)  &=& f(t,x)- F(t,x) \cdot D u(t,x),\quad \text{on }[0,T]\times \R^d, \notag \\
u(T,x) &=& g(x),\quad \text{on }\R^d,\label{ALTERNATE_PDE}
\end{eqnarray}
viewing the first order term as a source (depending here on the solution itself). In order to understand what type of smoothing effects can be expected for rough source we first begin by investigating the smoothness of the following equation:
\begin{eqnarray}
\label{LIN_WITH_ROUGH_SOURCE}
\p_t  w(t,x) + L^\alpha w(t,x)  &=& \textcolor{black}{\Phi}(t,x),\quad \text{on }[0,T]\times \R^d,\notag  \\
w(T,x) &=& 0,\quad \text{on }\R^d,
\end{eqnarray}
The parallel with the initial problem \eqref{Asso_PDE_PMG}, rewritten in \eqref{ALTERNATE_PDE}, is rather clear. We will aim at applying the results obtained below for the solution of \eqref{LIN_WITH_ROUGH_SOURCE} to  $\textcolor{black}{\Phi}=\textcolor{black}{f}-F\cdot Du $ (where the roughest part of the source will obviously be $ F\cdot Du$).

Given a map $\textcolor{black}{\Phi}$ in $\bL^r(\bB_{p,q}^{-1+\gamma})$ we now  specifically concentrate on the gain of regularity which can be obtained through the fractional operator $L^\alpha $  for the solution $w$ of \eqref{LIN_WITH_ROUGH_SOURCE} w.r.t. the data $\textcolor{black}{\Phi}$.  Having a lot of parameters at hand, this will provide a primer to understand what could be, at best, attainable for the  \textit{target} PDE \eqref{ALTERNATE_PDE}-\eqref{Asso_PDE_PMG}.

The solution of \eqref{LIN_WITH_ROUGH_SOURCE} corresponds to the Green kernel associated with $\textcolor{black}{\Phi}$ defined as:
\begin{equation} \label{DEF_GREEN}
G^\alpha \textcolor{black}{\Phi}(t,x) := \int_t^T ds \int_{\R^d} dy \textcolor{black}{\Phi}(s,y) p_\alpha(s-t,y-x).
\end{equation}
Since to address the well-posedness of the martingale problem we are led to contol, in some sense, gradients, we will here try to do so for the Green kernel introduced in \eqref{DEF_GREEN} solving the linear problem \eqref{LIN_WITH_ROUGH_SOURCE} with \textit{rough} source. Namely for a multi-index $\eta\in \N^d, |\eta|:=\sum_{i=1}^d \eta_i\le 1 $, we want to control $\textcolor{black}{D_x^\eta}  G^\alpha \textcolor{black}{\Phi}(t,x) $


Avoiding harmonic analysis techniques, which could in some sense allow to average non-integrable singularities, our approach allows to obtain \emph{almost optimal regularity} thresholds that could be attainable on $u$.  Thanks to the H\"older inequality (in time) and the duality on Besov spaces (see equation \eqref{EQ_DUALITY}) we have that:
\begin{eqnarray*}
\left|\textcolor{black}{D_x^\eta} G^\alpha \textcolor{black}{\Phi}(t,x)\right| &=& \left|\int_t^T ds \int_{\R^d} dy \textcolor{black}{\Phi}(s,y) \textcolor{black}{D_x^\eta}p(s-t,y-x)\right|\\
&\leq&  \| \textcolor{black}{\Phi} \|_{\bL^r((t,T],\bB_{p,q}^{-1+\gamma})} \|\textcolor{black}{D_x^\eta}p_\alpha(\cdot-t,\cdot-x) \|_{\bL^{r'}((t,T],\bB_{p',q'}^{1-\gamma})},
\end{eqnarray*}
where $p'$, $q'$ and $r'$ are the conjugate exponents of $p$, $q$ and $r$. Let us first focus, for $s\in (t,T]$ on the thermic part of $\|\textcolor{black}{D_x^\eta}p_\alpha(s-t,\cdot-x) \|_{\bB_{p',q'}^{1-\gamma}} 
$. We have with the notations of Section \ref{SEC_BESOV}:
\begin{eqnarray}
&&\Big(\mathcal{T}_{p',q'}^{1-\gamma}[\textcolor{black}{D_x^\eta}p_\alpha(s-t,\cdot-x)]\Big)^{q'}\notag\\
&=&\int_0^1 \frac{dv}{v} v^{(1-\frac{1-\gamma}{\alpha})q'} \| \partial_v\tilde p_\alpha(v,\cdot) \star \textcolor{black}{D_x^\eta}p_\alpha(s-t,\cdot-x)\|_{\bL^{p'}}^{q'}\notag\\
&=& \int_0^{(s-t)^{}} \frac{dv}{v} v^{(1-\frac{1-\gamma}{\alpha})q'} \| \partial_v\tilde p_\alpha(v,\cdot) \star \textcolor{black}{D_x^\eta}p_\alpha(s-t,\cdot-x)\|_{\bL^{p'}}^{q'}\notag\\
&&+ \int_{(s-t)^{}}^1 \frac{dv}{v} v^{(1-\frac{1-\gamma}{\alpha})q'} \| \partial_v\tilde p_\alpha(v,\cdot) \star \textcolor{black}{D_x^\eta}p_\alpha(s-t,\cdot-x)\|_{\bL^{p'}}^{q'}\notag\\
&=:&\Big(\mathcal{T}_{p',q'}^{1-\gamma}[\textcolor{black}{D_x^\eta}p_\alpha(\textcolor{black}{s}-t,\cdot-x)]|_{[0,(s-t)]}\Big)^{q'} \label{DECOMP_UPP_DOWN}\\
&&+\Big(\mathcal{T}_{p',q'}^{1-\gamma}[\textcolor{black}{D_x^\eta}p_\alpha(s-t,\cdot-x)]|_{[(s-t),1]}\Big)^{q'}.\notag 
\end{eqnarray}
In the above equation, we split the time interval into two parts. On the upper interval, for which there are no time singularities, we use directly convolution inequalities and the available controls for the derivatives of the heat kernel (see Lemma \ref{SENS_SING_STAB}).
On the lower interval we have to equilibrate the singularities in $v$ and use cancellation techniques involving the sensitivities of  $\textcolor{black}{D_x^\eta} p_\alpha $ (which again follow from Lemma \ref{SENS_SING_STAB}).

Let us begin with the upper part \textcolor{black}{(i.e. the second term in \eqref{DECOMP_UPP_DOWN})}. Using the $\bL^1-\bL^{p'}$ convolution inequality, we have from Lemma \ref{SENS_SING_STAB}:
\begin{eqnarray}
&&\Big(\mathcal{T}_{p',q'}^{1-\gamma}[\textcolor{black}{D_x^\eta}p_\alpha(s-t,\cdot-x)]|_{[(s-t),1]}\Big)^{q'}\notag\\
&\leq&\int_{(s-t)^{}}^1 \frac{dv}{v} v^{(1-\frac{1-\gamma}{\alpha})q'} \|\partial_v \tilde p_\alpha(v,\cdot) \|_{\bL^1}^{q'} \| \textcolor{black}{D_x^\eta}p_\alpha(s-t,\cdot-x)\|_{\bL^{p'}}^{q'}\notag\\
&\leq&\frac{C}{(s-t)^{(\frac d{p\alpha}+\frac {|\eta|}{\alpha})q'}} \int_{(s-t)^{}}^1 \frac{dv}{v} \frac{1}{v^{\frac{1-\gamma}{\alpha}q'}} 
\leq 
\frac{C}{(s-t)^{\left[\frac{1-\gamma}{\alpha}+\frac d{p\alpha}+\frac {|\eta|}{\alpha}\right]q'}}.\label{CTR_COUP_HAUTE_GREEN}
\end{eqnarray}
Indeed, we used for the second inequality that equation \eqref{SENSI_STABLE} \textcolor{black}{and the self similarity of $q_\alpha $ give}:
\begin{eqnarray}
&&\|\textcolor{black}{D_x^\eta}p_\alpha(s-t,\cdot-x)\|_{\bL^{p'}}=\Big(\int_{\R^d} \big(\textcolor{black}{D}_x^\eta p_\alpha(s-t,x) \big)^{p'} dx\Big)^{1/p'}\notag\\
&\le& \frac{C_{p'}}{(s-t)^{\frac{|\eta|}{\alpha}}} \Big((s-t)^{-\frac d \alpha(p'-1)}\int_{\R^d} \frac{dx}{(\textcolor{black}{s-t})^{\frac d\alpha}}\big(q_\alpha(1, \frac{x}{(s-t)^{\frac 1\alpha}})\big)^{p'}\Big)^{1/p'}\notag\\
&\le & C_{p'}(s-t)^{-[\frac{d}{\alpha p}+\frac{|\eta|}\alpha]}\Big( \int_{\R^d} d\tilde x \big(q(1,\tilde x)\big)^{p'}\Big)^{1/p'}\le \bar C_{p'}(s-t)^{-[\frac{d}{\alpha p}+\frac{|\eta|}{\alpha}]},\notag\\
\label{INT_LP_DENS_STABLE}
\end{eqnarray}
recalling that $p^{-1}+(p')^{-1}=1$ and $p\in (1,+\infty] $, $p'\in [1,+\infty) $ for the last inequality.

Hence, the map $s \mapsto \mathcal{T}_{p',q'}^{1-\gamma}[\textcolor{black}{D_x^\eta}p_\alpha(s-t,\cdot-x)]|_{[(s-t),1]}$ belongs to $\bL^{r'}((t,T],\R^+)$ as soon as
\begin{equation}\label{FIRST_COND_ETA}
-r'\left[\frac{1-\gamma}{\alpha}+\frac d{p\alpha}+\frac {|\eta|}{\alpha}\right] >-1 \Longleftrightarrow |\eta|< \alpha(1-\frac 1r)+\gamma-1-\frac dp.
\end{equation}
\textcolor{black}{We now focus on the lower part (i.e. the first term in \eqref{DECOMP_UPP_DOWN})}. Still from \eqref{SENSI_STABLE} (see again the proof of Lemma 4.3 in \cite{huan:meno:prio:19} for details), one derives that there exists $C$ s.t. for all $\beta\in (0,1] $ and all $(x,y,z)\in (\R^d)^2 $,
\begin{eqnarray}
&&|\textcolor{black}{D_x^\eta} p_\alpha(s-t,z-x)- \textcolor{black}{D_x^\eta} p_\alpha(s-t,y-x)|\notag\\
&\le& \frac{C}{(s-t)^{\frac{\beta+|\eta|}{\alpha}}} |z-y|^\beta \Big( q_\alpha(s-t,z-x)+q_\alpha(s-t,y-x)\Big).\label{CTR_BETA_GREENPART}
\end{eqnarray}
Indeed, \eqref{CTR_BETA_GREENPART} is direct if $|z-y|\ge (1/2) (s-t)^{1/\alpha} $ (off-diagonal regime). It suffices to exploit the bound \eqref{SENSI_STABLE} for $\textcolor{black}{D_x^\eta} p_\alpha(s-t,y-x) $ and $\textcolor{black}{D_x^\eta} p_\alpha(s-t,z-x) $ and to observe that $\big(|z-y|/(s-t)^{1/\alpha}\big)^{\beta}\ge 1 $. If now $|z-y|\le (1/2)(s-t)^{1/\alpha} $ (diagonal regime), it suffices to observe from \eqref{CTR_DER_M} that, with the notations of the proof of Lemma \ref{SENS_SING_STAB} (see in particular \eqref{DECOMP_G_P}), for all $\lambda\in [0,1] $:
\begin{eqnarray}
&&|\textcolor{black}{D_x^\eta} D p_M(s-t,y-x+\lambda(y-z))|\notag\\
&\le& \frac{C_m}{(s-t)^{\frac{|\eta|+1}\alpha}}p_{\bar M}(s-t,y-x-\lambda(y-z))\notag\\
&\le& \frac{C_m}{(s-t)^{\frac{|\eta|+1+d}\alpha}}\frac{1}{\Big( 1+\frac{|y-x-\lambda(z-y)|}{(s-t)^{\frac 1\alpha}} \Big)^{m}} \notag\\
&\le& \frac{C_m}{(s-t)^{\frac{|\eta|+1+d}\alpha}}\frac{1}{\Big( \frac 12+\frac{|y-x|}{(s-t)^{\frac 1\alpha}} \Big)^{m}}\le 2\frac{C_m}{(s-t)^{\frac {|\eta|+1}\alpha}} p_{\bar M}(s-t,y-x).\notag\\
\label{MIN_JUMP_GREENPART}
\end{eqnarray}
Therefore, in the diagonal case, \eqref{CTR_BETA_GREENPART} follows from \eqref{MIN_JUMP_GREENPART} and \eqref{DECOMP_G_P} writing 
\textcolor{black}{
\begin{eqnarray*}
|\textcolor{black}{D_x^\eta} p_\alpha(s-t,z-x)- \textcolor{black}{D_x^\eta} p_\alpha(s-t,y-x)| &\le& \int_0^1 d\lambda   | \textcolor{black}{D_x^\eta} D p_\alpha(s-t,y-x+\lambda(y-z)) \cdot (y-z)|\\
& \le& 2 C_m (s-t)^{-(|\eta|+1)/\alpha} q_{\alpha}(s-t,y-x)|z-y|\\
&\le & \tilde C_m (s-t)^{-(|\eta|+\beta)/\alpha} q_{\alpha}(s-t,y-x)|z-y|^\beta,
\end{eqnarray*}
}
for all $\beta \in [0,1] $ (exploiting again that $|z-y|\le (1/2) (s-t)^{1/\alpha} $ for the last inequality). From \eqref{CTR_BETA_GREENPART} we now derive:
\begin{eqnarray}
&&
\| \partial_v \tilde p_\alpha(v,\cdot) \star \textcolor{black}{D_x^\eta}p_\alpha(s-t,\cdot-x)\|_{L^{p'}}\notag\\
&
=&
\Big(\int_{\R^d } dz |\int_{\R^d}dy \partial_v \tilde p_\alpha(v,z-y)\textcolor{black}{D_x^\eta}p_\alpha(s-t,y-x)|^{p'} \Big)^{1/p'}\notag\\
&
=
&
\Big(\int_{\R^d } dz \Big|\int_{\R^d}dy \partial_v \tilde p_\alpha(v,z-y)\notag\\
&&
\times
\Big[\textcolor{black}{D_x^\eta}p_\alpha(s-t,y-x)-\textcolor{black}{D_x^\eta}p_\alpha(s-t,z-x)\Big]\Big|^{p'} \Big)^{1/p'}\notag \\
&
\le 
& 
\frac{1}{(s-t)^{\frac{|\eta|+\beta}{\alpha}}}\Big(\int_{\R^d } dz \Big|\int_{\R^d}dy|\partial_v \tilde p_\alpha(v,z-y)|\ |z-y|^\beta\notag\\
&&
\times\big[q_\alpha(s-t,y-x)  + q_\alpha(s-t,z-x) \big]\Big|^{p'}\Big)^{1/p'}\notag \\
&
\le & 
\frac{C_{p'}}{(s-t)^{\frac{|\eta|+\beta}{\alpha}}}\Bigg[\Big(\int_{\R^d } dz \Big|\int_{\R^d}dy |\partial_v \tilde p_\alpha(v,z-y)|\ |z-y|^\beta q_\alpha(s-t,y-x)\Big|^{p'}\Big)^{1/p'}  \notag \\
&&
+ \Big(\int_{\R^d} dz \big(q_\alpha(s-t,z-x)\big)^{p'}\Big(\int_{\R^{d}} dy |\partial_v \tilde p_\alpha(v,y-z)|\ |y-z|^\beta   \Big)^{p'} \Big)^{1/p'}\Bigg].\notag\\
\label{PREAL_CANC_GREEN}
\end{eqnarray}
From the $\bL^1-\bL^{p'}$  convolution inequality and Lemma \ref{SENS_SING_STAB} \textcolor{black}{(see also \eqref{INT_LP_DENS_STABLE})}  we thus obtain:
$$\|\textcolor{black}{\partial_v} \tilde p_\alpha(v,\cdot) \star \textcolor{black}{D_x^\eta}p_\alpha(s-t,\cdot-x)\|_{\bL^{p'}} \le \frac{C_{p'}}{(s-t)^{\frac{|\eta|+\beta+\frac dp}{\alpha}}} v^{-1+\frac \beta \alpha}.$$
Hence,
\begin{eqnarray}
&&\Big(\mathcal{T}_{p',q'}^{1-\gamma}[\textcolor{black}{D_x^\eta}p_\alpha(s-t,\cdot-x)]|_{[0,(s-t)]}\Big)^{q'}\notag\\
&\leq&\frac{C}{(s-t)^{\left[\frac{d}{p\alpha}+\frac{|\eta|}{\alpha}+\frac{\beta}{\alpha}\right]q'}}\int_0^{(s-t)^{}} \frac{dv}{v} v^{(1-\frac{1-\gamma}{\alpha}-1+\frac{\beta}{\alpha})q'} \notag \\
&\leq& \frac{C}{(s-t)^{\left[\frac{d}{p\alpha}+\frac{|\eta|}{\alpha}+\frac{\beta}{\alpha} + \frac{1-\gamma-\beta}{\alpha}\right]q'}}=\frac{C}{(s-t)^{\left[\frac{d}{p\alpha}+\frac{|\eta|}{\alpha} + \frac{1-\gamma}{\alpha}\right]q'}},\label{CTR_CANC_GREEN}
\end{eqnarray}
provided $\beta+\gamma>1$ for the second inequality (which can be assumed since we can choose  $\beta$ arbitrarily in $(0,1) $). The map $s \mapsto \mathcal{T}_{p',q'}^{1-\gamma}[\textcolor{black}{D_x^\eta}p_\alpha(s-t,\cdot-x)]|_{[0,(s-t)]}$ hence belongs to $\bL^{r'}((t,T],\R^+)$ under the same previous condition on $\eta $ than in \eqref{FIRST_COND_ETA}. \textcolor{black}{Let us eventually mention that the above arguments somehow provide the lines of the proof of Lemma \ref{LEM_BES_NORM} for $\Psi=1 $. The proof in its whole generality is provided in Appendix \ref{SEC_APP_TEC}}.

\begin{REM}[Pointwise gradient estimate on $G^\alpha$]
The condition in \eqref{FIRST_COND_ETA} then precisely gives that the gradient of the Green kernel will exist pointwise (with uniform bound depending on the Besov norm of $\textcolor{black}{\Phi} $) as soon as:
\begin{equation}
\label{COND_GRAD_PONCTUEL}
1< \alpha(1-\frac 1r)+\gamma-1-\frac dp \iff  \gamma >2- \alpha (1-\frac 1r)+\frac dp.
\end{equation}

In particular, provided \eqref{COND_GRAD_PONCTUEL} holds,  the same type of arguments would also lead to a H\"older control of the gradient in space of index $\zeta< \alpha(1-1/r)+\gamma-1- d/p-1 $. The previous computations somehow provide the almost optimal regularity that could be attainable for $u$ (through what can be derived from $w$ solving \eqref{LIN_WITH_ROUGH_SOURCE}). The purpose of the next section will precisely be to prove that these arguments can be adapted to that framework. The price to pay will be some additional constraint on the $\gamma $ \textcolor{black}{because we will precisely have to handle the product $F\cdot Du $ on the way.}
\end{REM}

\begin{REM}[On the second integrability parameter ``$q$'' in the Besov norm]
 Eventually, we emphasize that  the parameter $q$ does not play a key role in the previous analysis. Indeed, \textcolor{black}{none of} the thresholds appearing depend on this parameter. Since for all $\gamma,p$ we have that for all $q< q'$ that {\color{black}$B^{\gamma}_{p,q} \hookrightarrow B^{\gamma}_{p,q'}$} the above analysis suggests that it could be enough to consider the case $q=\infty$. Nevertheless, as it does not provide any additional difficulties, we let the parameter $q$ vary in the following.
\end{REM}

\subsection{Uniform estimates of the solution of the Cauchy problem $\mathscr C(F_m,L^\alpha,f,g,T)$
and associated (uniform) H\"older controls}\label{sec_cor_pde}
%
It is known that, under \A{UE} and for $\vartheta>\alpha ,$  if $g\in \bB_{\infty,\infty}^\vartheta\textcolor{black}{(\R^d,\R)} $ is also bounded and $f\in \bB_{\infty,\infty}^{\vartheta-\alpha}(\R^d,\R) $, for any $m \ge 1$ there exists a unique classical solution $u_m\in \bL^\infty([0,T],\bB_{\infty,\infty}^{\vartheta}(\R^d,\R))$ to the Cauchy problem $\mathscr C(F_m,L^\alpha,f,g,T)$.
This is indeed the usual Schauder estimates for sub-critical stable operators (see e.g. Priola \cite{prio:12} or Mikulevicius and Pragarauskas who also address the case of a multiplicative noise \cite{miku:prag:14}). It is clear that  the following Duhamel representation formula holds for $u_m $. With the notations of \eqref{DEF_HEAT_SG}:
\begin{equation}\label{DUHAMEL}
u_m(t,x) = P_{T-t}^{\alpha} [g] (x) {\color{black} - } G^\alpha f(t,x) + \mathfrak r_m(t,x),
\end{equation}
where the Green kernel $G^\alpha$ is defined by \eqref{DEF_GREEN} and where \textcolor{black}{the remainder term $ \mathfrak r_m$ is defined as follows}:
\begin{equation}\label{REMAINDER}
\mathfrak r_m(t,x) := \int_t^T ds P_{T-s}^{\alpha} \textcolor{black}{[ F_m(s,\cdot) \cdot D u_m(s,\cdot)]}(x).
\end{equation}
It is plain to check that, if we now relax the boundedness assumption on $g$, supposing it can have linear growth,  there exists $C:=C(d)>0$ such that 
\begin{eqnarray*}
&&\left\|\textcolor{black}{D}P_{T-t}^\alpha [g]\right\|_{\bL^\infty([0,T],\bB_{\infty,\infty}^{\vartheta\textcolor{black}{-1}})} + \left\|G^\alpha f\right\|_{\bL^\infty([0,T],\bB_{\infty,\infty}^{\vartheta})} \\
&\leq& C\big(\|f\|_{\bL^\infty ([0,T],\bB_{\infty,\infty}^{\vartheta-\alpha})} + \|Dg\|_{\bB_{\infty,\infty}^{\vartheta-1}}\big).
\end{eqnarray*}
We also refer to the section concerning the smoothness in time below for specific arguments related to a terminal condition with linear growth.\\

In the following, \textcolor{black}{we will extend the previous bounds in order to consider \textit{singular} sources as well}. In order to keep the notations as clear as possible, we drop the superscript $m$ associated with the mollifying procedure for the rest of the section. \textcolor{black}{Note also that the following analysis will allow us to drop the above condition $\vartheta > \alpha$, i.e. the above Duhamel representation holds in our setting with $\vartheta =\theta$.}

\bigskip
\noindent\emph{(i) Gradient bound.} Let us first control the terminal condition. We have, integrating by parts and using usual cancelation arguments,
\begin{eqnarray}
|D P_{T-t}^\alpha [g](x) | &\leq&  \sum_{j=1}^d |\p_{x_j}P_{T-t}^\alpha [g](x) |\leq   \sum_{j=1}^d  \left|  \int_{\R^d} dy  \p_{j}g(y)  p_\alpha(T-t,y-x)  \right|\notag\\
&\leq&    \sum_{j=1}^d  C \| Dg \|_{\bB_{\infty,\infty}^{\theta-1}}.\label{BD_GRAD_TERCOND}
\end{eqnarray}

We now turn to control the Green kernel part. Write
\begin{eqnarray}
|D G^\alpha f(t,x) | &\leq&  \sum_{j=1}^d |\p_{x_j}G^\alpha f(t,x) |\notag\\
&= &  \sum_{j=1}^d  \left| \int_t^T ds \int_{\R^d} dy  f(s,y)  \p_{x_j}p_\alpha(s-t,y-x)  \right|\notag\\
&\leq &   \sum_{j=1}^d  \| f \|_{\bL^{\infty}(\bB_{\infty,\infty}^{\theta-\alpha})} \| \p_{x_j} p_\alpha(\cdot-t,\cdot-x) \|_{\bL^{1}(\bB_{1,1}^{\alpha-\theta})}.\notag
\end{eqnarray}
From the very definition \eqref{DEF_THETA} of $\theta$  we have $\theta-\alpha+1<1$ and $(\theta-\alpha+1)+1>1$. We can thus apply Lemma \ref{LEM_BES_NORM} (see eq. \eqref{Esti_BES_NORM} with $\gamma=\theta-\alpha+1$, $\beta =1$, $\eta=1$ and \textcolor{black}{$\Psi=1 $} therein) to obtain
\begin{equation*}
\| \p_{x_j} p_\alpha(s-t,\cdot-x) \|_{\bB_{1,1}^{\alpha-\theta}\big(\R^d\big)} \leq \frac{C}{(s-t)^{\left[\frac{\alpha-\theta}{\alpha}+\frac {1}{\alpha}\right]}}.
\end{equation*}
\textcolor{black}{Recalling $\theta>1 $}, we thus obtain
\begin{equation}
\left\|DG^\alpha f\right\|_{\bL^\infty} \leq C (T-t)^{\frac{\theta-1}{\alpha}} \|f\|_{\bL^\infty ([0,T],\bB_{\infty,\infty}^{\theta-\alpha})}.\label{BD_GRAD_GREEN}
\end{equation}

Let us now focus on first gradient estimate of $\mathfrak r$. Using \textcolor{black}{the} H\"older inequality and then Besov duality we have,
\begin{eqnarray}
|D \mathfrak r(t,x) | &\leq&  \sum_{j=1}^d |\p_{x_j}\mathfrak r(t,x) |\notag\\
&\leq &  \sum_{j=1}^d \sum_{k=1}^d \left| \int_t^T ds \int_{\R^d} dy  F_k(s,y) \p_{y_k} u(s,y) \p_{x_j}p_\alpha(s-t,y-x)  \right|\notag\\
&\leq &   \sum_{j=1}^d \sum_{k=1}^d  \| F_k \|_{\bL^r(\bB_{p,q}^{-1+\gamma})} \| \p_{k} u \p_{x_j} p_\alpha(\cdot-t,\cdot-x) \|_{\bL^{r'}(\bB_{p',q'}^{1-\gamma})},\label{BD_GRAD}
\end{eqnarray}
so that the main issue consists in establishing the required control on the map $(t,T] \ni s \mapsto \| \p_{k} u(s,\cdot) \p_{x_j} p_\alpha(\cdot-t,\cdot-x) \|_{\bB_{p',q'}^{1-\gamma}}$ for any $j,k$ in $\leftB 1,d\rightB$. Note that since for all $s$ in $[0,T]$ the map $y \mapsto \textcolor{black}{ u(s,y)}$ is in $\bB_{\infty,\infty}^{\vartheta}$ for any $ \vartheta \in (\alpha, \alpha+1]$, we have in particular from the very definition of $\theta$ (see eq. \eqref{DEF_THETA}) and assumptions on $\gamma$ that there exists $\varepsilon >0$ such that $\theta-1-\varepsilon>0$, $\theta-1-\varepsilon + \gamma >1$ and for all $s$ in $[0,T]$ the map $y \mapsto \p_k u(s,y)$ is in $\bB_{\infty,\infty}^{\theta-1-\varepsilon}$. One can hence apply Lemma \ref{LEM_BES_NORM} so that (see eq. \eqref{Esti_BES_NORM} with $\beta =\theta - 1-\varepsilon $, $\eta=1$ \textcolor{black}{and  $\Psi(s,\cdot)=\partial_k u(s,\cdot) $} therein)  
\begin{equation*}
\| \p_k u(s,\cdot) \p_{x_j} p_\alpha(s-t,\cdot-x) \|_{\bB_{p',q'}^{\textcolor{black}{1-\gamma}}} \leq \|\p_k u (s,\cdot)\|_{\bB_{\infty,\infty}^{ \textcolor{black}{\theta-1-\varepsilon}}} \frac{C}{(s-t)^{\left[\frac{1-\gamma}{\alpha}+\frac d{p\alpha}+\frac {1}{\alpha}\right]}}.
\end{equation*}
This map hence belongs to $\bL^{r'}((t,T],\R_+)$ as soon as
\begin{equation}\label{THRESHOLD_GRAD}
-r'\left[\frac{d}{p\alpha}+\frac{1}{\alpha} + \frac{1-\gamma}{\alpha}\right] >-1 \Leftrightarrow  \gamma > 2 -\alpha +\frac{\alpha}{r} + \frac dp,
\end{equation}
which follows from the assumptions on $\gamma$. We then obtain, after taking the $\bL^{r'}((t,T],\R_+)$ norm of the above estimate, that
\begin{eqnarray}\label{upp_bound_Du}
|D \mathfrak r(t,x)| \leq CT^{\frac{\theta-1}{\alpha}} \|D u \|_{\bL^{\infty}(\bB^{ \textcolor{black}{\theta-1-\varepsilon}}_{\infty,\infty})},
\end{eqnarray}
\textcolor{black}{recalling from \eqref{DEF_THETA} that $\theta=\gamma-1+\alpha-d/p-\alpha /r$}.

\bigskip
\noindent\emph{(ii) H\"older norm of the gradient.} As in the above proof we obtain gradient bounds depending on the spatial H\"older norm of $D u$, we now have to precisely estimate this quantity. \textcolor{black}{The main difficulty is induced by the remainder term}:
\begin{eqnarray*}
&&|D \mathfrak r(t,x) -D \mathfrak r(t,x')| \\
&\leq& \sum_{j=1}^d |\p_{j}\mathfrak r(t,x) -\p_{j}\mathfrak r(t,x') |\\
&\leq& 
\sum_{j=1}^d \sum_{k=1}^d \Big| \int_t^T ds \int_{\R^d} dy  F_k(s,y) \left(\p_{y_k} u(s,y) \right.\\
&& \quad \times\left.\left(\p_{x_j}p_\alpha(s-t,y-x) - 
\p_{x_j}p_\alpha(s-t,y-x')\right)\right) \Big|\\
&\leq& 
\sum_{j,k=1}^d\!  \| F_k \|_{\bL^r(\bB_{p,q}^{-1+\gamma})} \| \p_{k} u\left(\p_{x_j} p_\alpha(\cdot\!-t,\cdot\!-x) -\p_{x_j} p_\alpha(\cdot\!-t,\cdot\!-x')\right)\|_{\bL^{r'}(\bB_{p',q'}^{1-\gamma})},
\end{eqnarray*}
using again the H\"older inequality and duality between the considered Besov spaces (see Section \ref{SEC_BESOV}).
Hence, the main issue consists in establishing the required control on the map 
$$(t,T] \ni s \mapsto \| \p_{k} u(s,\cdot) \left(\p_{x_j} p_\alpha(\textcolor{black}{s}-t,\cdot-x) -\p_{x_j} p_\alpha(\textcolor{black}{s}-t,\cdot-x')\right) \|_{\bL^{r'}(\bB_{p',q'}^{1-\gamma})},$$ 
for any $j,k$ in $\leftB 1,d\rightB$. Since $\theta-1-\varepsilon<1$, one can again apply Lemma \ref{LEM_BES_NORM} so that (see eq. \eqref{Esti_BES_HOLD} with $\beta = \theta-1-\varepsilon$, $\beta'=\theta-1-\varepsilon$, $\eta=1$ and \textcolor{black}{$\Psi(s,\cdot)=\partial_k u(s,\cdot) $} therein):
\begin{eqnarray*}
&&\| \p_k u (s,\cdot) \big(\p_{x_j} p_\alpha(s-t,\cdot-x) - \p_{x_j} p_\alpha(s-t,\cdot-x')\big) \|_{\bB_{p',q'}^{\textcolor{black}{1-\gamma}}} \\
&\leq& \|\p_k u(s,\cdot)\|_{ \bB_{\infty,\infty}^{\theta-1-\varepsilon}}\frac{C}{(s-t)^{\left[\frac{1-\gamma}{\alpha}+\frac d{p\alpha}+\frac {1+(\theta-1-\varepsilon)}{\alpha}\right]}}|x-x'|^{\theta-1-\varepsilon}\\
&\le& 
\textcolor{black}{\frac{C\|\p_k u\|_{ \bL^\infty(\bB_{\infty,\infty}^{\theta-1-\varepsilon})}}{(s-t)^{\left[\frac{1-\gamma}{\alpha}+\frac d{p\alpha}+\frac {1+(\theta-1-\varepsilon)}{\alpha}\right]}}|x-x'|^{\theta-1-\varepsilon}}.
\end{eqnarray*}
The above map hence belongs to $\bL^{r'}((t,T],\R_+)$ as soon as
\begin{equation}\label{THRESHOLD_HOLDER}
-r'\left[\frac{d}{p\alpha}\!+\!\frac{1+(\theta-1-\varepsilon)}{\alpha}\! +\! \frac{1-\gamma}{\alpha}\right] >-1 \Leftrightarrow  \theta-1-\varepsilon < \gamma - \left( 2 -\alpha \!+\!\frac{\alpha}{r}\! +\! \frac dp\right),
\end{equation}
which readily follows from the very definition of $\theta$ (see eq. \eqref{DEF_THETA}) and the fact that $\varepsilon >0$. We then obtain
\begin{eqnarray}\label{upp_bound_Du_hold}
|D \mathfrak r(t,x) -D \mathfrak r(t,x') | \leq CT^{\frac {\textcolor{black}{\varepsilon}} \alpha} \|D u \|_{\textcolor{black}{\bL^{\infty}(\bB^{\theta-1-\varepsilon}_{\infty,\infty})}} |x-x'|^{\theta-1-\varepsilon}.
\end{eqnarray}

\begin{REM}{\color{black}
Note that assuming that $\theta$ is fixed, we readily obtain from  \eqref{THRESHOLD_HOLDER} together with the constraint $\theta-1-\varepsilon + \gamma >1$ the initial constraint 
\begin{equation}\label{bound_gamma}
\gamma > \frac{3-\alpha + \frac{d}{p} + \frac{\alpha} r}{2}.
\end{equation}
In comparison with the threshold obtained when investigating the smoothing effect of the Green kernel (see eq. \eqref{COND_GRAD_PONCTUEL} and the related discussion) this \textcolor{black}{additional} regularity allows to \textcolor{black}{handle} the product $F\cdot Du$, in the sense that it allows to give a meaning to this product as a distribution. Indeed, as suggested by  Lemma \ref{LEM_BES_NORM} (replacing therein the heat kernel by a smooth test function), a sufficient condition to define the product $F\cdot Du$ is to obtain estimate on the H\"older modulus of the map $Du$ of order $\beta$ for some $\beta >1-\gamma = -(-1+\gamma)$ (see \eqref{Esti_BES_NORM}). The threshold \eqref{bound_gamma} precisely reflects that constraint and appears as the price to pay to define such a modulus.
}
%
%
%
%
\end{REM} 

Let us eventually estimate the H\"older moduli of the \textcolor{black}{gradients} of the first and second terms in the Duhamel representation \eqref{DUHAMEL}. We first note that, \textcolor{black}{for the} Green kernel, \textcolor{black}{the proof follows from the above lines}. When doing so, we obtain that 
\begin{equation}\label{BD_HOLD_GREEN}
|D G^\alpha f(t,x) - D G^\alpha f(t,x')| \leq  CT^{\textcolor{black}{\frac \varepsilon\alpha}}\|f\|_{\bB_{\infty,\infty}^{\theta-\alpha}} |x-x'|^{\theta-1-\varepsilon}.
\end{equation}
Concerning the terminal condition\textcolor{black}{,} we have \textcolor{black}{on} the one hand, when \textcolor{black}{$(T-t)^{\frac 1\alpha}\leq |x-x'|$ (off-diagonal regime), that}:
\begin{eqnarray}
&&|D P_{T-t}^\alpha [g](x) - D P_{T-t}^\alpha [g](x')|  \notag\\
&= &   \left|  \int_{\R^d} dy  Dg(y)  \big(p_\alpha(T-t,y-x)-p_\alpha(T-t,y-x')\big)  \right|\notag\\
&\leq &     \Bigg|  \int_{\R^d} dy  \big(Dg(y) - Dg(x)\big)  p_\alpha(T-t,y-x) +Dg(x)-Dg(x')\notag\\
&& - \int_{\R^d} dy  \big(Dg(y) - Dg(x')\big)  p_\alpha(T-t,y-x') \Bigg|\notag\\
&\leq &   C \|Dg\|_{\bB_{\infty,\infty}^{\theta-1}} |x-x'|^{\theta-1-\varepsilon},\notag
\end{eqnarray}
\textcolor{black}{recalling that $p_\alpha$ is a density for the first inequality.}
On the other hand,  when \textcolor{black}{$(T-t)^{\frac 1\alpha}> |x-x'|$ (diagonal regime)}, we have using cancellations arguments
\begin{eqnarray*}
&&|D P_{T-t}^\alpha [g](x) - D P_{T-t}^\alpha [g](x')|  \\
&\leq &\big |\int_{\R^{d}} [ p_\alpha(T-t,y-x) -p_\alpha(T-t,y-x') ] Dg(y) dy  \big | \\
&\leq& \big |\int_0^1 d\lambda \int_{\R^{d}}  [D_{x}  p_\alpha\big(T-t,y-\textcolor{black}{(x'+\mu(x-x'))}\big) \cdot (x-x') ]\notag\\
&&\times [D g(y)\textcolor{black}{-Dg(x'+\mu(x-x'))}] dy  \big |\\
&\leq&  \|Dg\|_{\bB_{\infty,\infty}^{\theta-1}} (T-t)^{-\frac 1 \alpha + \frac{\theta-1}{\alpha}}|x-x'| \leq C(T-t)^{\textcolor{black}{\frac \varepsilon\alpha}} \|Dg\|_{\bB_{\infty,\infty}^{\theta-1}} |x-x'|^{\theta-1-\varepsilon}.
\end{eqnarray*}
Hence
\begin{equation}\label{BD_HOLD_TERCOND}
|D P_{T-t}^\alpha [g](x) - D P_{T-t}^\alpha [g](x')| \leq  C(T^{\textcolor{black}{\frac \varepsilon\alpha}}+1)\|Dg\|_{\bB_{\infty,\infty}^{\theta-1}} |x-x'|^{\theta-1-\varepsilon}.
\end{equation}

Putting together estimates \eqref{BD_GRAD_TERCOND}, \eqref{BD_GRAD_GREEN}, \eqref{upp_bound_Du}, \eqref{upp_bound_Du_hold}, \eqref{BD_HOLD_GREEN} and \eqref{BD_HOLD_TERCOND} we deduce that 
\begin{eqnarray}\label{THE_BD_GRAD_EN_THETA_1_EPS}
\forall\, \alpha \in  \left( \frac{1+\frac dp}{1-\frac 1r},2\right],\, \forall \gamma \in \left(\frac{3-\alpha + \frac{d}{p} + \frac{\alpha} r}{2}, 1\right],\ \exists C(T)>0 \text{ s.t. }\notag\\
\| D u \|_{\bL^{\infty}\left(\bB^{\gamma - 2+\alpha - \frac dp - \frac{\alpha}r-\varepsilon}_{\infty,\infty}\right)} < C_T.
\end{eqnarray}
In particular, \textcolor{black}{when $g=0 $},  $\lim C_T =0$ when $T$ tends to $0$.

\bigskip
\noindent\emph{(iii) Smoothness in time for $u$ and $Du$.} We restart here from the Duhamel representation \eqref{DUHAMEL}. Namely,
\begin{equation*}
u(t,x) = P_{\textcolor{black}{T-t}}^{\alpha} [g] (x) {\color{black} - } G_{}^{\alpha} [ f](t,x) + \mathfrak r(t,x),
\end{equation*}
where from \eqref{REMAINDER}, \textcolor{black}{the remainder term writes}:
\begin{equation*}
\mathfrak r(t,x) = \int_t^T ds \int_{\R^d} dy \textcolor{black}{[ F(s,y) \cdot D u(s,y)]} p_\alpha(s-t,y-x).
\end{equation*}
We now want to control for a fixed $x\in \R^d $ and $0\le t< t'\le T $ the difference:
\begin{eqnarray}
\label{DIFF_U}
u(t',x)-u(t,x) &=&  \big(P_{\textcolor{black}{T-t'}}^{\alpha}-P_{\textcolor{black}{T-t}}^{\alpha} \big)    [g] (x) {\color{black} - }\big(G^{\alpha}f(t',x)- G^{\alpha}f(t,x)\big)  \notag\\
&&+ \big(r(t',x)- r(t,x)\big).
\end{eqnarray}

For the first term in the r.h.s. of \eqref{DIFF_U} we write:
\begin{eqnarray*}
\big(P_{\textcolor{black}{T-t'}}^{\alpha}-P_{\textcolor{black}{T-t}}^{\alpha} \big)    [g] (x)=\int_{\R^d} \big[ p_\alpha(T-t',y-x)-p_\alpha(T-t,y-x)\big] g(y)dy\\
=-\int_{\R^d} \int_0^1 d\lambda \big[ \partial_s p_\alpha(s,y-x)\big]\Big|_{s=T-t-\lambda (t'-t)} g(y)dy (t'-t).
\end{eqnarray*}
From the Fubini's theorem and usual cancellation arguments we get:
\begin{eqnarray*}
\big(P_{\textcolor{black}{T-t'}}^{\alpha}-P_{\textcolor{black}{T-t}}^{\alpha} \big)    [g] (x)
\!\!\!&=&\!\!\!-(t'-t) \int_0^1 d\lambda \Big[\int_{\R^d}  \partial_s p(s,y-x)\\
&&\!\!\!\times \big( g(y) -g(x)-Dg(x)\cdot(y-x)\big)dy\Big]  \Big|_{s=T-t-\lambda (t'-t)}. 
\end{eqnarray*}
We indeed recall that, because of the symmetry of the driving process $\mW $, and since $\alpha>1 $, one has for all $s>0$, $\int_{\R^d} p(s,y-x)(y-x)dy=0 $. Recalling as well 
that we assumed $Dg\in \bB_{\infty,\infty}^{\theta-1} $, we therefore derive from Lemma \ref{SENS_SING_STAB}:
\begin{eqnarray*}
&&|\big(P_{\textcolor{black}{T-t'}}^{\alpha}-P_{\textcolor{black}{T-t}}^{\alpha} \big)    [g] (x)|\\
&\le& (t'-t)\int_{0}^1 d\lambda  \Big[\frac{C\|Dg\|_{\bB_{\infty,\infty}^{\theta-1}}}s \int_{\R^d}  q_\alpha(s,y-x) |y-x|^{\theta}dy\Big]\Big|_{s=T-t-\lambda (t'-t)}\notag\\
&\le & C (t'-t)\|Dg\|_{\bB_{\infty,\infty}^{\theta-1}} \int_{0}^1 d\lambda s^{-1+\frac \theta \alpha} \big|_{s=T-t-\lambda (t'-t)},
\end{eqnarray*}
\textcolor{black}{recalling from \eqref{DEF_THETA} that $\theta<\alpha $ for the last inequality}.
Observe now that since $0\le t<t'\le T $, one has $s=T-t-\lambda(t'-t)\ge (1-\lambda)(t'-t) $ for all $\lambda\in[0,1] $. Hence,
\begin{eqnarray}
|\big(P_{\textcolor{black}{T-t'}}^{\alpha}-P_{\textcolor{black}{T-t}}^{\alpha} \big)    [g] (x)|&\le&C (t'-t)\|Dg\|_{\bB_{\infty,\infty}^{\theta-1}} \int_{0}^1 \frac{d\lambda}{(1-\lambda)^{1-\frac\theta\alpha}}(t'-t)^{-1+\frac\theta\alpha} \notag\\
&\le & C (t'-t)^{\frac \theta\alpha}\|Dg\|_{\bB_{\infty,\infty}^{\theta-1}},\label{REG_TEMPS_COND_TERM}
\end{eqnarray}
which is the expected control. We now focus on the remainder term $r$ since the control of the Green kernel is easier and can be derived following the same lines of reasoning. Write
\begin{eqnarray}
\mathfrak r(t',x)- 
\mathfrak r(t,x)&=&\int_{t'}^T ds \big(P_{\textcolor{black}{s-t'}}^{\alpha}-P_{\textcolor{black}{s-t}}^{\alpha} \big)    \textcolor{black}{[ F(s,\cdot) \cdot Du(s,\cdot)]} (x)\notag\\
&&+\int_t^{t'} ds P_{\textcolor{black}{s-t}}^{\alpha} \textcolor{black}{[ F(s,\cdot) \cdot Du(s,\cdot)]} (x).\label{DIFF_R}
\end{eqnarray}

From Lemma \ref{LEM_BES_NORM} (see eq. \eqref{Esti_BES_NORM} with $\beta = \theta-1-\varepsilon$ and $\eta=0$) it can be deduced (see computations in point \emph{(i)} of the current section) that
\begin{equation}\label{REG_TEMPS_REMAINDER}
|\int_t^{t'} ds P_{\textcolor{black}{s-t}}^{\alpha} \textcolor{black}{[ F(s,\cdot) \cdot Du(s,\cdot)]} (x)|\le C |t-t'|^{\frac{\theta}{\alpha}}.
\end{equation}

Let us now focus on 
\begin{eqnarray}
&&\int_{t'}^T ds \big(P_{\textcolor{black}{s-t'}}^{\alpha}-P_{\textcolor{black}{s-t}}^{\alpha} \big)   \textcolor{black}{[ F(s,\cdot) \cdot Du(s,\cdot)]} (x)\notag\\
&=&\int_{t'}^T ds \int_{0}^1 d\lambda \Big\{\partial_w P_{\textcolor{black}{s-w}}^\alpha \textcolor{black}{[ F(s,\cdot) \cdot Du(s,\cdot)]} (x)\Big\} \Big|_{w=t+\lambda (t'-t)}(t'-t)\notag\\
&=& \int_{0}^1 d\lambda \int_{t'}^T ds  \Big\{L^\alpha P_{\textcolor{black}{s-w}}^\alpha \textcolor{black}{[ F(s,\cdot) \cdot Du(s,\cdot)]} (x)\Big\} \Big|_{w=t+\lambda (t'-t)}(t'-t).\notag\\ \label{DEV_SENSI_EN_TEMPS}
\end{eqnarray}
We have
\begin{eqnarray}
&&\int_{t'}^T ds|L^\alpha P_{\textcolor{black}{s-w}}^\alpha\textcolor{black}{[ F(s,\cdot) \cdot Du(s,\cdot)]} (x) | \notag\\
&\leq &   \sum_{k=1}^d \int_{t'}^Tds \Big|\int_{\R^d} dy  F_k(s,y) \p_{y_k} u(s,y) L^\alpha p_\alpha(s-w,y-x) \Big| \notag \\
&\leq &    \sum_{k=1}^d  \| F_k \|_{\bL^r([t',T],\bB_{p,q}^{-1+\gamma})} \| \p_{k} u L^\alpha p_\alpha(\cdot-w,\cdot-x) \|_{\bL^{r'}([t',T],\bB_{p',q'}^{1-\gamma})}.\label{DUALITY_TO_BE_INTEGRATED_FOR_TIME_SMOTHNESS}
\end{eqnarray}
\textcolor{black}{Applying} Lemma \ref{LEM_BES_NORM} (\textcolor{black}{see eq. \eqref{Esti_BES_NORM}} with $\beta = \theta-1-\varepsilon$ and $\eta=\alpha$ therein),  \textcolor{black}{we get}: 
\begin{equation*}
\| \p_{k} u(s,\cdot)  L^\alpha p_\alpha(s-\textcolor{black}{w},\cdot-x) \|_{\bB_{p',q'}^{\textcolor{black}{1-\gamma}}} \leq \|\p_{k} u(s,\cdot)\|_{\bB_{\infty,\infty}^{\textcolor{black}{ \theta-1-\varepsilon}}} \frac{C}{(s-\textcolor{black}{w})^{\left[\frac{1-\gamma}{\alpha}+\frac d{p\alpha}+1\right]}}.
\end{equation*}
Thus, \textcolor{black}{from \eqref{THE_BD_GRAD_EN_THETA_1_EPS} (recall from \eqref{DEF_THETA} that $ \gamma - 2+\alpha - \frac dp - \frac{\alpha}r-\varepsilon=\theta-1-\varepsilon$)}:
\begin{eqnarray}
 \| \p_{k} u L^\alpha p_\alpha(\cdot-w,\cdot-x) \|_{\bL^{r'}([t',T],\bB_{p',q'}^{1-\gamma})} &\le& C(t'-w)^{\frac 1{r'}-\big(\frac{1-\gamma}{\alpha}+\frac d{p\alpha}+1\big)}\notag\\
 & =&C(t'-w)^{\frac \theta \alpha-1}.\label{PREAL_BD_CTR_TEMPS_BAS}
\end{eqnarray}
Therefore, from \eqref{PREAL_BD_CTR_TEMPS_BAS} and \eqref{DUALITY_TO_BE_INTEGRATED_FOR_TIME_SMOTHNESS}, we derive:
\begin{eqnarray*}
\int_{t'}^T ds|L^\alpha P_{\textcolor{black}{s-w}}^\alpha \textcolor{black}{[ F(s,\cdot) \cdot Du(s,\cdot)]} (x) | 
&\leq &    C\sum_{k=1}^d  \| F_k \|_{\bL^r(\bB_{p,q}^{-1+\gamma})}  (t'-w)^{\frac \theta \alpha-1},
\end{eqnarray*}
which in turn, plugged into \eqref{DEV_SENSI_EN_TEMPS}, gives:
\begin{eqnarray}
&&|\int_{t'}^T ds \big(P_{\textcolor{black}{s-t'}}^{\alpha}-P_{\textcolor{black}{s-t}}^{\alpha} \big)   \textcolor{black}{[ F(s,\cdot) \cdot Du(s,\cdot)]} (x)|\notag\\
&\le& 
 \int_{0}^1 d\lambda \int_{t'}^T ds  \Big| L^\alpha P_{\textcolor{black}{s-w}}^\alpha\textcolor{black}{[ F(s,\cdot) \cdot Du(s,\cdot)]} (x)\Big| \Bigg|_{w=t+\lambda (t'-t)}(t'-t)\notag\\ 
 &\le & C\sum_{k=1}^d  \| F_k \|_{\bL^r(\bB_{p,q}^{-1+\gamma})}  \int_0^1 d\lambda (t'-(t+\lambda(t'-t)))^{\frac \theta \alpha -1}(t'-t)\notag\\
 &\le & C\sum_{k=1}^d  \| F_k \|_{\bL^r(\bB_{p,q}^{-1+\gamma})} (t'-t)^{\frac \theta \alpha}.\label{CTR_SENSI_EN_TEMPS}
 \end{eqnarray}
From \eqref{CTR_SENSI_EN_TEMPS}, \eqref{REG_TEMPS_REMAINDER} and \eqref{DIFF_R} we thus obtain:
\begin{equation}
\label{CTR_DIFF_R}
\big|\mathfrak r(t',x)- \mathfrak r(t,x)\big|\le C  \| F \|_{\bL^r(\bB_{p,q}^{-1+\gamma})} (t'-t)^{\frac { \theta} \alpha}.
\end{equation}

The H\"older control of the Green kernel $G^\alpha f$ follows from similar arguments. Indeed, repeating the above proof  it is plain to check that there exists $C\ge 1$ s.t. for all  $0\le t<t'\le T$, $x\in \R^d $:
\begin{equation}
\label{SMOOTH_TIME_GREEN_KERNEL}
\Big|\big(G^{\alpha}f(t',x)- G^{\alpha}f(t,x)\big) \Big|\le C \|f\|_{\bL^{\infty}(\bB_{\infty,\infty}^{\theta-\alpha})} (t'-t)^{\frac{\theta}{\alpha}}.
\end{equation}

The final control of \eqref{CTR_SCHAUDER_LIKE} concerning the smoothness in time then follows plugging \eqref{REG_TEMPS_COND_TERM}, \eqref{CTR_DIFF_R} and \eqref{SMOOTH_TIME_GREEN_KERNEL} into \eqref{DIFF_U}. The control concerning the time sensitivity of the spatial gradient would be obtained following the same lines.\\

\subsection{Proofs of Theorem \ref{THE_PDE}, Proposition \ref{PROP_PDE_MOLL}, Proposition \ref{COR_ZVON_THEO}, Corollary \ref{COR_ZVON_THEO_UNIF} and Lemma \ref{lemme:defprod}}\label{proof-pde}
Points \emph{(i)} to \emph{(iii)} conclude the proof of Proposition \ref{PROP_PDE_MOLL} up to the convergence  assertion. {\color{black}  Let us notice that the previous analysis allows to obtain Proposition \ref{COR_ZVON_THEO} up to this assertion as well. Indeed, in such a case, the map $f_m$ is  the $k^{\rm th}$ coordinate of $-F_m$ and should thus be estimated in term of its $\bL^r([0,T],\bB^{-1+\gamma}_{p,q}(\R^d,\R))$ norm. The associate control can be obtained following exactly the same strategy as the one we used to handle the remainder  term $\mathfrak r$ in the Duhamel representation \eqref{DUHAMEL} except that we do not need to deal with the additional gradient of the solution. 
The fact that the constant therein are decreasing w.r.t. time follows from the fact that $g\equiv 0$, see \eqref{THE_BD_GRAD_EN_THETA_1_EPS}.

Eventually, end of the proof of Proposition \ref{PROP_PDE_MOLL} and the proof of Theorem  \ref{THE_PDE} follows from compactness argument\textcolor{black}{s} together with the Schauder like control of Propositions \ref{PROP_PDE_MOLL} and the previous analysis. Uniqueness follows from the Schauder like control of Propositions \ref{PROP_PDE_MOLL} as well, as the underlying PDE is linear. The representation \eqref{repsolPDE} holds through similar computations. Corollary \ref{COR_ZVON_THEO_UNIF} is derived in the same way through Proposition \ref{COR_ZVON_THEO}, whose proof is concluded following the same lines as for Proposition \ref{PROP_PDE_MOLL}. 

Lemma \ref{lemme:defprod} follows from the above calculations as well. {\color{black} The fact that the product $F\cdot Du$ makes sense is an easy consequence of Lemma \ref{LEM_BES_NORM} (see estimate \eqref{Esti_BES_HOLD} with the heat kernel therein replaced by a smooth test function)  together with the regularity of $Du$.}
}

\begin{REM}[About additional diffusion coefficients]\label{REM_COEFF_DIFF}Let us first explain how, in the diffusive setting, $\alpha=2 $ the diffusion coefficient can be handled. Namely, this would lead to consider for the  PDE with mollified coefficients an additional term in the Duhamel formulation that would write:
\begin{eqnarray}\label{DUHAM_PERT}
u_m(t,x) &=& P_{\textcolor{black}{s-t}}^{\alpha,\xi,m}[g](x) {\color{black} - } \int_t^T ds P_{\textcolor{black}{s-t}}^{\alpha,\xi,m}[\Big\{ f (s,\cdot){\color{black} - } F_m \cdot D u_m (s,\cdot)\notag\\
&&+ \frac 12 {\rm Tr }\big((a_m(s,\cdot)-a_m(s,\xi)) D^2 u_m(s,\cdot)\big)\Big\}](x),
\end{eqnarray}
for an auxiliary parameter $\xi$ which will be taken equal to $x$ after potential differentiations in \eqref{DUHAM_PERT}.  Here, $P_{\textcolor{black}{s-t}}^{\alpha,\xi,m}$ denotes the two-parameter semi-group associated with $\big(\frac 12 {\rm Tr} \big(a_m(v,\xi) D^2\big)\big)_{v\in [s,t]} $ (mollified diffusion coefficient frozen at point $\xi$).
Let us focus on the second order term. Recall from the above proof of Proposition \ref{PROP_PDE_MOLL} that we aim \textcolor{black}{at estimating} the gradient pointwise, deriving as well some H\"older continuity for it. Hence, focusing on the additional term, we write for the gradient part:
\begin{eqnarray*}
&&D_x \int_t^T ds P_{\textcolor{black}{s-t}}^{\alpha,\xi,m}[ \frac 12 {\rm Tr }\big((a_m(s,\cdot)-a_m(s,\xi)) D^2 u_m(s,\cdot)\big)](x)\\
&=&\int_t^T ds \int_{\R^d}D_x p_\alpha^{\xi,m}(t,s,x,y)\frac 12 {\rm Tr }\big((a_m(s,y)-a_m(s,\xi)) D^2 u_m(s,y)\big) dy\\
&=&\frac 12 \sum_{i,j=1}^d\int_t^T ds \int_{\R^d}\Big( D_x p_\alpha^{\xi,m}(t,s,x,y) \big((a_{m,i,j}(s,y)-a_{m,i,j}(s,\xi))\Big)\\
&&\times D_{y_iy_j}u_m(s,y) dy.
\end{eqnarray*}
From the previous Proposition \ref{PROP_PDE_MOLL}, we aim at establishing that  $Du_m $ has H\"older index $\theta-1-\varepsilon=\gamma-2+\alpha- d/p-\alpha /r-\varepsilon$ and therefore $D_{y_iy_j}u_m\in \bB_{\infty,\infty}^{ \theta-2-\varepsilon
} $. Assume for a while that $ p=q=r=+\infty$. The goal is now to bound the above term through Besov duality. Namely, taking $\xi=x $ after having taken the gradient w.r.t. $x$ for the heat kernel, we get:
\begin{eqnarray*}
&&|D_x \int_t^T ds P_{\textcolor{black}{s-t}}^{\alpha,\xi,m}[ \frac 12 {\rm Tr }\big((a_m(s,\cdot)-a_m(s,\xi)) D^2 u_m(s,\cdot)\big)](x)| \Big|_{\xi=x}\\
&\le &\textcolor{black}{\sum_{i,j=1}^d}\int_t^T ds \|\Big( D_x p_\alpha^{\xi,m}(t,s,x,\cdot) \big((a_{m,i,j}(s,\cdot)-a_{m,i,j}(s,\xi))\Big)\|_{\bB_{1,1}^{2+\varepsilon-\theta}} \Big|_{\xi=x} \\
&&\times \|\textcolor{black}{\partial_{i,j}^2} u_m(s,\cdot)\|_{\bB_{\infty,\infty}^{\theta-2-\varepsilon}}.
\end{eqnarray*}
Now, in the considered case $\theta-2-\varepsilon=\gamma-1-\varepsilon$. Recalling that $D_x p_\alpha^{\xi,m}(t,s,x,\cdot)\in \bB_{1,1}^{1/2-\tilde \varepsilon} $ for any $\tilde \varepsilon>0 $ for $\gamma> 1/2=(3-\alpha)/2 $ and $\varepsilon$ small enough, we will indeed have that $D_x p_\alpha^{\xi,m}(t,s,x,\cdot) \big((a_{m,i,j}(s,\cdot)-a_{m,i,j}(s,\xi)) \in \bB_{1,1}^{2+\varepsilon-\theta}$ provided the bounded function $a$ itself has the same regularity, \textcolor{black}{i.e. $2+\varepsilon-\theta $, the integrability of the product deriving from the one of the heat kernel}. Since $ \|\textcolor{black}{\partial_{i,j}^2} u_m(s,\cdot)\|_{\bB_{\infty,\infty}^{\theta-2-\varepsilon}}\le C \|D u_m(s,\cdot)\|_{\bB_{\infty,\infty}^{\theta-1-\varepsilon}}$, see e.g. Triebel \cite{trie:83}, this roughly means that, the same Schauder estimate should hold with a diffusion coefficient $a\in \bL^\infty([0,T],\bB_{\infty,\infty}^{2+\varepsilon-\theta})$. \textcolor{black}{Similar thresholds also appear more generally in \cite{ZZ17}}. The general diffusive case for $p,q,r \ge 1$ and $\gamma $
 satisfying the conditions of Theorem \ref{THEO_WELL_POSED} can be handled similarly through duality arguments. \\
 
 For the pure jump case, we illustrate for simplicity what happens if the diffusion coefficient is scalar. Namely, when 
  $L^{\alpha,\sigma}\varphi(x)={\rm p.v.} \int_{\R^d} \big(\varphi(x+\sigma(x)z)- \varphi(x)\big)\nu(dz)=-\sigma^\alpha(x )(-\Delta)^{\alpha /2} \varphi(x) $, where $\sigma $ is a non-degenerate diffusion coefficient.  Introducing $L^{\alpha,\sigma,\xi}\varphi(x)={\rm p.v.} \int_{\R^d} \big(\varphi(x+\sigma(\xi)z)- \varphi(x)\big)$ $\nu(dz)=-\sigma^\alpha(\xi)(-\Delta)^{ \alpha/ 2} \varphi(x) $, we rewrite for the Duhamel formula, similarly to \eqref{DUHAM_PERT}:
\begin{eqnarray}\label{DUHAM_PERT_JUMP}
u_m(t,x) &=& P_{\textcolor{black}{s-t}}^{\alpha,\xi,m}[g](x) {\color{black} - } \int_t^T ds P_{\textcolor{black}{s-t}}^{\alpha,\xi,m}[\Big\{ f (s,\cdot){\color{black} - } F_m \cdot D u_m (s,\cdot)\notag\\
&&+ (L^{\alpha,\sigma_m}-L^{\alpha,\sigma_m,\xi}) u_m(s,\cdot)\big)\Big\}](x).
\end{eqnarray}
Focusing again on the non-local term, we write for the gradient part:
\begin{eqnarray*}
&&D_x \int_t^T ds P_{\textcolor{black}{s-t}}^{\alpha,\xi,m}[\big(\sigma_m^\alpha(s,\cdot)-\sigma_m^\alpha(s,\xi)) \Delta^{\frac \alpha 2} u_m(s,\cdot)\big)](x)\\
&=&-\int_t^T ds \int_{\R^d}D_x p_\alpha^{\xi,m}(t,s,x,y)\big(\sigma_m^\alpha(s,y)-\sigma_m^\alpha(s,\xi)\big)(- \Delta)^{\frac \alpha 2} u_m(s,y) dy.
\end{eqnarray*}
Consider again the case $ p=q=r=\infty$. Since $Du_m\in \bL^\infty([0,T],\bB_{\infty,\infty}^{\theta-1-\varepsilon})$, \textcolor{black}{we thus have that} $-(-\Delta)^{ \alpha/ 2}u_m \in \bL^\infty([0,T],\bB_{\infty,\infty}^{\theta-\alpha-\varepsilon}) $, where $\theta-\alpha-\varepsilon =-1+\gamma-\varepsilon$. Still by duality one has to control \textcolor{black}{the  norm of the term} $ D_x p_\alpha^{\xi,m}(t,s,x,y)$ $\big(\sigma_m^\alpha(s,y)-\sigma_m^\alpha(s,\xi)\big)$ in the Besov space $\bB_{1,1}^{1-\gamma+\varepsilon} $. Since $\gamma> (3-\alpha)/2 $ and $D_x p_\alpha^{\xi,m}(t,s,x,y) \in \bB_{1,1}^{1- 1/\alpha}$\!\!, this will be the case provided the coefficient $\sigma  \in \bL^\infty([0,T],\bB_{\infty,\infty}^{1-\gamma+\varepsilon})$ for $\varepsilon $ small enough observing that $1-\gamma+\varepsilon\textcolor{black}{<}(\alpha-1)/2$. 

Note that, in comparison with the result obtained in  \cite{ling:zhao:19}, the above threshold is precisely the one appearing in \cite{ling:zhao:19} in this specific case. The general matrix case for $\sigma$ is more involved. It requires in \cite{ling:zhao:19} the Bony decomposition. We believe it could also be treated through the duality approach considered here but postpone \textcolor{black}{this} discussion to further research. In the scalar case, the analysis for general  $ p,q,r,\gamma$ as in Theorem \ref{THEO_WELL_POSED} could be performed similarly.
\end{REM}

\section{Dynamics of the \emph{formal} SDE \eqref{SDE}}\label{SEC_RECON_DYN}

In this part, we aim at proving Theorem \ref{THEO_DYN} and Corollary \ref{INTEG_STO}.
\textcolor{black}{We restrict here to the pure jump case $\alpha\in (1,2) $, since the diffusive one was already  considered in  \cite{dela:diel:16}. We adapt here their procedure to the current framework}. 

In subsection \ref{SHAPE}, we first recover the noise through {\color{black} an enlarged} martingale problem \textcolor{black}{(point (i) of Proposition \ref{PROP_REG_PARTIELLE} below)}, then recover a drift as the difference between the {\color{black} Martingale} solution and the noise obtained before and estimate its contribution \textcolor{black}{(point (ii) of Proposition \ref{PROP_REG_PARTIELLE} below)}. With this contribution at hand, we show that the drift decomposes as a \textcolor{black}{principal} part plus a remainder which has \textcolor{black}{a \textit{negligible}} contribution (point (iii) of Proposition \ref{PROP_REG_PARTIELLE} below). Then, we recall in Subsection \ref{YOUNG} how the general construction of the stochastic Young integral from  \cite{dela:diel:16} translates in our setting. Eventually, we derive in subsection \ref{SEC_DYNAMICS} the dynamics associated with the solution of the Martingale Problem and define the class of processes to which an associated It\^o's formula holds. This last part thus conclude the proof of Theorem \ref{THEO_DYN} and Corollary \ref{INTEG_STO}. 



\subsection{Shape of the drift}\label{SHAPE}
\begin{PROP}
\label{PROP_REG_PARTIELLE} Let $\alpha \in (1,2)$. For any initial point $x\in \R^d$, one can find a probability measure $\mathbf P^\alpha$ on  $\mathcal D([0,T], \R^{2d}) $ s.t. the canonical process $(X_t, \mW_t)_{t\in [0,T]} $ satisfies the following properties:
\begin{trivlist}
\item[(i)]  Under $\mathbf P^\alpha$, the law of $(X_t)_{t\ge 0}$ is a solution of the Martingale  Problem associated with data ($L^\alpha,F,x)$, $x \in \R^d$ and  the law of $(\mW_t)_{t\ge 0} $ corresponds to the one of a $d$-dimensional stable process with generator $L^\alpha$. 

\item[(ii)] For any $1 \leq \mathfrak q < \alpha $, there exists a constant $ C:=C(\alpha,p,q,r,\gamma,\mathfrak q)$ s.t. for any  $0\le v<s\le T$:
\begin{equation}
\label{REG_DRIFT_FOR_DYN}
\E^{\mathbf P^\alpha}[| X_{s}-  X_v-( \mW_{s}- \mW_v)|^{\mathfrak q}]^{\frac 1{\mathfrak q}}\le C (s-v)^{\frac 1\alpha+\frac{\theta-1}{\alpha}},
\end{equation}

\item[(iii)] Let $(\F_v)_{v\ge 0}:=\big(\sigma ( (X_w,\mW_ w)_{0\le w \le v}  ) \big)_{v\ge 0} $ denote the filtration generated by the couple $(X,\mW)$. For any $0\le v<s\le T $, it holds that:
$$\E^{\mathbf P^\alpha}[ X_{s}- X_v|\F_v] =\mathfrak f(v,X_v,s-v)=\E^{\mathbf P^\alpha}[u(v,X_v)-u(s,X_v)|\F_v],$$
with $\mathfrak f(v,X_v,s-v):=u(v,X_v)-X_v $, where $u$ is the mild solution of the Cauchy problem $\mathscr C(F,L^\alpha,0,x,s)$ \textcolor{black}{(note that the dependence of $\mathfrak f(v,X_v,s-v)$ on $s$ is precisely through the Cauchy problem $\mathscr C(F,L^\alpha,0,x,s)$)}.



Furthermore, the following decomposition holds:
\begin{eqnarray}
{\mathfrak f}(v,X_v,s-v)&=&\mathscr F(v,X_v,s-v)+{\mathscr R}(v,X_v,s-v),\notag \\
|\mathscr F(v,X_v,s-v)|&=&\Big|\int_v^s d\textcolor{black}{w} \int_{\R^d}dy F(w,y)  p_\alpha(\textcolor{black}{w}-s,y-X_v)\Big|\notag\\
&\le& C\|F\|_{\bL^r([0,T],\bB_{p,q}^{-1+\gamma})} (s-v)^{\frac 1\alpha+\frac{\theta-1}{\alpha}},
\notag \\
|{\mathscr R}(v,X_v,s-v)|&\le & C(s-v)^{1+\varepsilon'},\ \varepsilon'>0.\label{THE_CONTROLS_FOR_THE_DRIFT}
\end{eqnarray}
\end{trivlist}
\end{PROP}

{\color{black}
\begin{REM}
The above proposition gives a first information on the shape of the drift. Indeed, using the decomposition
$$X_{t+h}-X_t= \E[X_{t+h}-X_t|\mathcal F_t] + X_{t+h}-X_t- \E[X_{t+h}-X_t|\mathcal F_t],$$
one can see that the infinitesimal increment of the \textcolor{black}{canonical process} $X$ involves a drift part (first term in the above r.h.s.) and a martingale part (second and third terms in the above r.h.s.). The main point being now that 
$$\E[X_{t+h}-X_t|\mathcal F_t] = \mathfrak f(t,X_t,h) = \mathscr F(t,X_t,h)  +{\mathscr R}(t,X_t,h) =  \mathscr F(t,X_t,h) + \mathcal O(h^{1+\varepsilon'}),$$
meaning that only the first term $\mathscr F$ matters \emph{i.e.} the infinitesimal dynamics involves, as a drift, the mollified version of the initial one along the density of the driving noise.
\end{REM}}

\begin{proof}

\begin{trivlist}
\item[\textit{(i)}] Coming back to point \emph{(i)} in Section \ref{SDE_2_PDE} we have that the couple  $\big((X_t^m, \mW_t^m)_{t\in [0,T]}\big)_{m \ge 0}$ is tight (pay attention that the stable noise $\mW^m$ feels the mollifying procedure as it is obtained through solvability of the Martingale Problem) so that it converges, along a subsequence, to the couple $(X_t, \mW_t)_{t\in [0,T]}$.

\item[\textit{(ii)}] Let $0\le v<s $. Let $u_m = (u_m^1,\ldots,u_m^d)$ where each $u_m^k$, $k$ in $\{1,\ldots,d\}$ is chosen as the solution the Cauchy problem $\mathscr C(F_m,L^\alpha,0,x_k,s)$ where
$x_k$ is the $k^{{\rm th}}$ coordinate of $x=(x_1,\ldots,x_d)\in \R^d$. We have
$$X_s^m-X_v^m = u_m(s,X_s^m) - u_m(\textcolor{black}{s},X_v^m) =  u_m(s,X_s^m) - u_m(v,X_v^m) +u_m(v,X_{\textcolor{black}{v}}^m) - u_m(\textcolor{black}{s},X_v^m).$$
Let us now notice that
$$\mW_s^m-\mW_v^m = \int_v^s \int_{\R^d \backslash\{0\}}x \tilde N^m(dw,dx),$$
so that, from It\^o's formula
\begin{eqnarray}
&&X_s^m-X_v^m\notag\\
\ \ \ \  &=&M_{v,s}^{s,m}(\alpha,u_m,X^m)+[u_m (v,X_v^m)-u_m(s,X^m_v)]\label{Del_trans}\\
&=& \int_v^s \int_{ 
\R^d \backslash\{0\} 
} \{u_m(w,X^m_{w^-}+x) - u_m(w,X_{w^-}^m) \}\tilde N^m(dw,dx)\notag\\
&&+[u_m(v,X_v)-u_m(s,X_v)]\notag\\
&=&\mW_s^m-\mW_v^m +[u_m(v,X_v^m)-u_m(s,X_v^m)] \notag\\
&&+ \int_v^s \int_{ |x|\le 1} \{u_m(w,X_{w^-}^m+x) - u_m(w,X_{w^-}^m) -x\}\tilde N^m(dw,dx)\notag\\
&&+
\int_v^{\textcolor{black}{s}}\int_{|x|\ge 1}\{u_m(w,X_{w^-}^m+x) - u_m(w,X_{w^-}^m) -x\}\tilde N^m(dw,dx).\notag\\
&=:& \mW_s^m-\mW_v^m +[u_m(v,X_v^m)-u_m(s,X_v^m)]+ \mathcal M_S^m(v,s)+ \mathcal M_L^m(v,s).\notag
\end{eqnarray}
From the smoothness properties of $u_m$ established in \textcolor{black}{Proposition \ref{PROP_PDE_MOLL}} (in particular $|u^s_m(v,X_v^m)-u^s_m(s,X_v^m)]|\leq C(s-v)^{\theta/\alpha}$ and the gradient is uniformly bounded) we have
\begin{eqnarray}
\textcolor{black}{|}\mathcal U(w,X_{w^-}^m,x) \textcolor{black}{|}\!\!&:=&\!\! \big|u_m(w,X_{w^-}^m+x) - u_m(w,X_{w^-}^m) -x \big|\notag\\
\!\!&=&\!\! \Big|\int_0^1 d\lambda (D u_m(w,X_{w^-}^m+\lambda x)-I) \cdot x\Big|
\leq  C(s-w)^{\frac{\theta-1}{\alpha}} |x|, \label{ESTI_COUP_BASSE_POISSON}
\end{eqnarray}
recalling that for all $z$ in $\R^d$, $u_m(s,z) = z$ \textcolor{black}{so that $Du_m(s,z) = I$, and using estimate \eqref{CTR_SCHAUDER_LIKE}}. Note that $\big(\mathcal M_S^m(v,s)\big)_{0\leq v < s \leq T}$ and $\big(\mathcal M_L^m(v,s)\big)_{0\leq v < s \leq T}$ are respectively $\bL^2$ and $\bL^{\mathfrak q}$ martingales associated respectively with the ``small'' and ``large'' jumps. Let us first handle the ``large'' jumps. We have by the \textcolor{black}{Burkholder-Davies-Gundy (BDG)} inequality that
$$\E\big[|\mathcal M_L^m(v,s)|^{\textcolor{black}{\mathfrak q}}\big] \leq C_\ell \E\big[[\mathcal M_L^m]_{(v,s)}^{\frac {\mathfrak q}2}\big],$$
where $[\mathcal M_L^m]_{(v,s)}$ denotes the corresponding bracket given by the expression $\sum_{v\leq w \leq s} |\mathcal U(w,X_{w^-}^m,\Delta \mW_w^m)|^2\mathbf{1}_{|\Delta \mW_w^m|\ge 1}$. Using the linear growth of $\mathcal U$ w.r.t. its third variable (uniformly w.r.t. the second one) from \eqref{ESTI_COUP_BASSE_POISSON} together with the fact that $\mathfrak q/2\leq 1$ we obtain 
\begin{eqnarray*}
&&\Big(\sum_{v\leq w \leq s} | \mathcal U(w,X_{w^-}^m,\Delta \mW_w^m)|^2\textcolor{black}{\mathbf{1}_{|\Delta \mW_w^m|\ge 1}}\Big)^{\mathfrak q/2} \notag\\
&\leq& C (s-w)^{q\frac{\theta-1}{\alpha}} \Big(\sum_{v\leq w \leq s} | \Delta \mW_w^m|^2\textcolor{black}{\mathbf{1}_{|\Delta \mW_w^m|\ge 1}}\Big)^{\mathfrak q/2}\\
& \leq& C(s-w)^{\mathfrak q\frac{\theta-1}{\alpha}} \sum_{v\leq w \leq s} |\Delta \mW_w^m|^q\textcolor{black}{\mathbf{1}_{|\Delta \mW_w^m|\ge 1}}.
\end{eqnarray*} 
We then readily get from the compensation formula that
\begin{eqnarray*}
\E\big[|\mathcal M_L^m(v,s)|^{\mathfrak q}\big] &\leq& C (s-w)^{1 +\mathfrak  q\frac{\theta-1}{\alpha}}\int |x|^q \textcolor{black}{\mathbf{1}_{|x|\ge 1}} \nu(dx) \leq C'(s-w)^{1 +\mathfrak  q\frac{\theta-1}{\alpha}}\\
&\le& C'(s-w)^{\frac{q}{\alpha} + \mathfrak q\frac{\theta-1}{\alpha}}.
\end{eqnarray*}
We now deal with the ``small'' jumps and split them w.r.t. their characteristic scale writing
\begin{eqnarray*}
\mathcal M_S^m(v,s) &=& \mathcal M_{S,1}^m(v,s)+\mathcal M_{S,2}^m(v,s) \\
&=: &\int_v^{\textcolor{black}{s}}\int_{|x| > (s-v)^{\frac 1\alpha}}\textcolor{black}{\mathbf{1}_{|x|\le 1}}\mathcal U(w,X_{w^-}^m,x) \tilde N^m(dw,dx)\\
&& + \int_v^{\textcolor{black}{s}}\int_{ |x| \le (s-v)^{\frac 1\alpha}}\textcolor{black}{\mathbf{1}_{|x|\le 1}}\mathcal U(w,X_{w^-}^m,x) \tilde N^m(dw,dx).
\end{eqnarray*}

In the off-diagonal regime (namely for $\mathcal M_{S,1}^m(v,s)$), we do not face any integrability problem w.r.t. the L\'evy measure. The main idea consists then in using first \textcolor{black}{the BDG} inequality, then the compensation formula and \eqref{ESTI_COUP_BASSE_POISSON}, and eventually \textcolor{black}{usual convexity arguments} together with the compensation formula again to obtain
\begin{eqnarray*}
\E[|\mathcal M_{S,1}^m(v,s)|^{\mathfrak q}] &=& \E\left[\left| \int_v^s\int_{|x| > |s-v|^{\frac 1\alpha}}\textcolor{black}{\mathbf{1}_{|x|\le 1}} \mathcal U(w,X_{w^-}^m,x)\tilde N^m(dr,dx)\right|^{\mathfrak q} \right]\\
&\leq& C_q  \E\left[ \left(\sum_{v\leq w \leq s} |\mathcal U(w,X_{w^-}^m,\Delta \mW_w^m)|^2\mathbf{1}_{\textcolor{black}{1>}|\Delta \mW_w^m| > |v-s|^{\frac 1\alpha}}\right)^{\frac {\mathfrak q}2}\right]\notag\\
&\leq& C_q    (s-v)^{1+\mathfrak q\frac{\theta-1}{\alpha}} \int_{\textcolor{black}{1>}|x| > |v-s|^{\frac 1\alpha}
} \big|x\big|^{\mathfrak q}  \nu(dx)\notag\\
&\leq&  C_q |v-s|^{\frac {\mathfrak q}\alpha +\frac{\theta-1}{\alpha}}.
\end{eqnarray*}
In the diagonal regime (i.e. for $\mathcal M_{S,2}^m(v,s)$) we use the BDG inequality and \eqref{ESTI_COUP_BASSE_POISSON} to recover integrability w.r.t. the L\'evy measure and then use the \textcolor{black}{additional} integrability to obtain better estimate. \textcolor{black}{Namely}:
\begin{eqnarray}
\E[|\mathcal M_{S,2}^m(v,s)|^{\mathfrak q}] &=& C \E\left[ \left| \int_v^s \int_{|x|\le |v-s|^{\frac 1\alpha}\textcolor{black}{\wedge 1}
}  \mathcal U(w,X_{w^-}^m,x) \tilde N^m(dw,dx)\right|^{\mathfrak q}\right]\notag\\
&\leq& C_{\mathfrak q}  \left( \int_v^s \int_{|x|\le |v-s|^{\frac 1\alpha}\textcolor{black}{\wedge 1}
} \big|\mathcal U(w,X_{w^-}^m,x)\big|^2 dw \nu(dx)\right)^{\frac {\mathfrak q} 2}\notag\\
&\leq& C_{\mathfrak q} \left((s-v)^{1+2\frac{\theta-1}{\alpha}} \int_{|x|\le |v-s|^{\frac 1\alpha}\textcolor{black}{\wedge 1}
} \big|x\big|^2 \nu(dx)\right)^{\frac {\mathfrak q} 2}\notag\\
&\le &C_{\mathfrak q}  (s-v)^{\frac {\mathfrak q}\alpha + \mathfrak q\frac {\theta-1}\alpha}.\notag
\end{eqnarray}
Using the above estimates on the $\mathfrak q$-moments of $\mathcal M_{L}^m(v,s)$, $\mathcal M_{S,1}^m(v,s)$ and $\mathcal M_{S,2}^m(v,s)$ the statement follows passing to the limit in $m$, thanks to Proposition \ref{PROP_PDE_MOLL}.

\item[\textit{(iii)}] Letting $(\F_v^m)_{v\ge 0}:=\big(\sigma ( (X_w^m,\mW_ w^m)_{0\le w \le v}  ) \big)_{v\ge 0} $, restarting from \eqref{Del_trans} and taking the conditional expectation w.r.t. $\F^m$ yields
\begin{eqnarray*}
\E[ X_s^m- X_v^m |\F_v^m]&=&\E[u_m(v,X_v^m)-u_m(s,X_v^m)|\F_v^m]=u_m(v,X_v^m) - X_v^m.
\end{eqnarray*}
Passing to the {\color{black}(weak)} limit in $m$, it can be deduced that from Proposition \ref{PROP_PDE_MOLL} that
$$\E[ X_{s}- X_v|\F_v] =u(v,X_v)-X_v = : \mathfrak f(v,X_v,s-v),$$
where $u$ is the mild solution of $\mathscr C(F,L^\alpha,0,x,s)$.
From the mild definition of $u$ in Theorem \ref{THE_PDE} we obtain that \textcolor{black}{for all $(w,y)\in [s,v]\times \R^d $:}
\textcolor{black}{
\begin{eqnarray*}
Du(w,y) &=&  \int_{\R^d} dy' \{y' \otimes Dp_\alpha(s-w,y'-y)\} \\
&&+ \int_w^s dw' \int_{\R^d} dy  [Du(w',y') \textcolor{black}{\cdot} F(w',y')] \otimes Dp_\alpha(w'-w,y'-y)\\
&=&   I + \int_w^s dw' \int_{\R^d} dy'  [Du(w',y')  \textcolor{black}{\cdot} F(w',y')] \otimes Dp_\alpha(w'-w,y'-y),
\end{eqnarray*}
}
integrating by parts to derive the last inequality. We thus get:
\begin{eqnarray}
&&\E[ X_s- X_v |\F_v]\notag\\
&=&
u(v,X_v)-u(s,X_v)\notag\\
&=&
\int_v^s dw \int_{\R^d}dy Du(w,y)F(w,y) p_\alpha(w-v,y-X_v)\notag\\
&=&
\int_v^s dw \int_{\R^d}dy F(w,y)  p_\alpha(w-\textcolor{black}{v},y-X_v)\notag\\
&&+ \int_v^s \!\!\!dw \!\!\int_{\R^d}\!\!\!\!dy\!\!  \int_w^s\!\!\!\! dw' \!\!\int_{\R^d}\!\!\!\!dy'   \big[[Du(w',y')  \textcolor{black}{\cdot} F(w',y')] \otimes \textcolor{black}{D}_y  p_\alpha(w'-w,y'-y)\big] F(w,y)\notag\\
&&\qquad \times  p_\alpha(w-v,y-X_v),\notag\\\label{dvp}
\end{eqnarray}
where we have again plugged the mild formulation of $Du$ {\color{black} from \eqref{repsolPDE}}. Let us first prove that the first term in the above has the right order. Thanks to Lemma \ref{LEM_BES_NORM} (with $\eta=0$ and $\Psi =\rm{Id}$ therein) we obtain that:
{\color{black}
\begin{eqnarray}
&&{\mathscr F}(v,X_v,s-v)\notag\\
&:=&\Big|\int_v^s d\textcolor{black}{w} \int_{\R^d}dy  F(\textcolor{black}{w},y)  p_\alpha(\textcolor{black}{w}-\textcolor{black}{v},y-X_v)\Big|\notag\\
&\le& C\|F\|_{\bL^r([0,T],\bB_{p,q}^{-1+\gamma})} (s-v)^{1-(\frac 1r +\frac d{p\alpha}+\frac{1-\gamma}\alpha)}\notag\\
&\le & C\|F\|_{\bL^r([0,T],\bB_{p,q}^{-1+\gamma})} (s-v)^{\frac 1\alpha+\big[1-\frac 1\alpha-(\frac 1r +\frac d{p\alpha}+\frac{1-\gamma}\alpha)\big]}\notag\\
&\le & C\|F\|_{\bL^r([0,T],\bB_{p,q}^{-1+\gamma})} (s-v)^{\frac 1\alpha+\frac{\theta-1}{\alpha}}.
\label{PREAL_CTR_GOOD_CONTROL_FOR_YOUNG}
\end{eqnarray}
}

Let us now prove that the second \textcolor{black}{term} in the r.h.s. of \eqref{dvp} is a negligible perturbation.  
Setting
\begin{eqnarray*}
\textcolor{black}{\psi}_{v,\textcolor{black}{w},s}(y) &:=&   p_\alpha(\textcolor{black}{w}-v,y-X_v)\int_{\textcolor{black}{w}}^s d\textcolor{black}{w}' \int_{\R^d}dy'  \\
&&\times [Du(\textcolor{black}{w}',y')  \textcolor{black}{\cdot} F(\textcolor{black}{w}',y')] \otimes \textcolor{black}{D}_y  p_\alpha(\textcolor{black}{w}'-r,y'-y)\\
&=&p_\alpha(\textcolor{black}{w}-v,y-X_v) D\mathfrak r(\textcolor{black}{w},y),
\end{eqnarray*}
we write:
\begin{eqnarray*}
{\mathscr R}(v,X_v,s-v):=\textcolor{black}{\int_v^s d\textcolor{black}{w} \int_{\R^d}dy \textcolor{black}{\psi}_{v,\textcolor{black}{w},s}(y) F(\textcolor{black}{w},y).}
\end{eqnarray*}
We thus have the following estimate:
\begin{equation}\label{DEF_REMAINDER_DRIFT}
|{\mathscr R}(v,X_v,s-v)|\le \|F\|_{\bL^r([0,T],\bB_{p,q}^{-1+\gamma})}\|\psi_{v,\cdot,s}(\cdot)\|_{\bL^{r'}([0,T],\bB_{p',q'}^{1-\gamma})}.
\end{equation}
Let us now consider the thermic part of  $ \|\psi_{v,\cdot,s}(\cdot)\|_{\bL^{r'}([0,T],\bB_{p',q'}^{1-\gamma})}$, \textcolor{black}{which can be split again into a lower and an upper part, as in \eqref{DECOMP_UPP_DOWN}.  We first deal with the upper part and then \textcolor{black}{with} the lower one.} With the same previous notations\footnote{\textcolor{black}{Pay attention that, in order to absorb some singularities we cannot here directly appeal to Lemma \ref{LEM_BES_NORM} but simply  exploit some $\textcolor{black}{\bL}^\infty$ of $D\textcolor{black}{\mathfrak r}(t,\cdot)$ in terms of $(T-t)^{\frac \theta\alpha} $}.}:
\begin{eqnarray}
&&\Big({\mathcal T}_{p',q'}^{1-\gamma}(\psi_{v,\textcolor{black}{w},s}(\cdot))\Big|_{[(\textcolor{black}{w}-v),1]}\Big)^{q'}\notag\\
&\le& C(\textcolor{black}{w}-v)^{-\frac{1-\gamma}{\alpha}q'}\|D\mathfrak r(\textcolor{black}{w},\cdot)\|_{\infty}^{q'}\|p_\alpha(\textcolor{black}{w}-v,\cdot,-X_v)\|_{\bL^{p'}}^{q'}\notag\\
&\le & C(\textcolor{black}{w}-v)^{-\frac{1-\gamma}{\alpha}q'}(s-\textcolor{black}{w})^{\frac{(\theta-1)}\alpha q'}(\textcolor{black}{w}-v)^{-\frac{d}{\alpha p}q'},
\end{eqnarray}
using \eqref{upp_bound_Du} and \eqref{INT_LP_DENS_STABLE} for the last inequality. Hence,
\begin{equation}
\Big( \int_v^s d\textcolor{black}{w} \Big({\mathcal T}_{p',q'}^{1-\gamma}(\psi_{v,\textcolor{black}{w},s}(\cdot))\Big|_{[(\textcolor{black}{w}-v),1]}\Big)^{r'}\Big)^{1/r'} 
\le  C(s-v)^{\frac 1{r'}+\frac{\theta-1}{\alpha}-\frac{d}{\alpha p}-\frac{1-\gamma}{\alpha}}.\label{BD_REMAINDER_COUPURE_HAUTE}
\end{equation}
Observe that, for this term to be a remainder on small time intervals, we need:
$$\frac 1{r'}+\frac{\theta-1}{\alpha}-\frac{d}{\alpha p}-\frac{1-\gamma}{\alpha}>1 \iff \gamma-1+\theta-1-\frac{d}{ p}-\frac \alpha r >0.$$
Recalling the definition of $\theta $ in \eqref{DEF_THETA}, we obtain the condition:
\begin{equation}\label{COND_TO_BE_A_REMAINDER}
\gamma>\frac{3-\alpha+\frac {2d}p+\frac{2\alpha}r}{2} .
\end{equation}
This stronger condition appears only in the case where one is interested in expliciting exactly the dynamics in terms of a drift which actually writes as the mollified version of the initial one along the density of the driving noise (regularizing kernel). Note that if one chooses to work in a bounded setting, i.e. for $p=r=\infty $, \eqref{COND_TO_BE_A_REMAINDER} again corresponds to the \textcolor{black}{condition} appearing in  Theorem \ref{THEO_WELL_POSED}.

Let us now deal with the \textcolor{black}{lower part of the thermic characterization. Using \textcolor{black}{a} cancellation argument, r}estarting from  \textcolor{black}{\eqref{upp_bound_Du_hold} and \eqref{SENSI_STABLE}, exploiting as well \eqref{THE_BD_GRAD_EN_THETA_1_EPS}} and \eqref{CTR_BETA_GREENPART}, we get for $\beta=\theta-1-\varepsilon $:
\begin{eqnarray}\label{Holder_prod_AGAIN}
&& \left|D \mathfrak r(\textcolor{black}{w},y)p_\alpha (\textcolor{black}{w}-v,y-x) - D  \mathfrak r(\textcolor{black}{w},z)p_\alpha(\textcolor{black}{w}-v,z-x)\right|\\
&\leq & C\left[ \left(\|D \mathfrak r(\textcolor{black}{w},\cdot)\|_{ \textcolor{black}{\dot \bB^\beta_{\infty,\infty}}} + \frac{\|D  \mathfrak r(\textcolor{black}{w},\cdot)\|_{\textcolor{black}{\bL^{\infty}}
}}{(r-v)^{\frac{ \beta}\alpha}}\right)\right.\notag\\
&&\left.\times\left(q_\alpha(\textcolor{black}{w}-v,y-x) +q_\alpha(\textcolor{black}{w}-v,z-x) \right)\right] 
|y-z|^\beta\notag\\
&\leq &   C\Big( (s-\textcolor{black}{w})^{\frac \varepsilon \alpha}+\frac{(s-\textcolor{black}{w})^{\frac{\theta-1}\alpha}}{(\textcolor{black}{w}-v)^{\frac \beta\alpha}}
\Big)
\left(q_\alpha(\textcolor{black}{w}-v,y-x) +q_\alpha(\textcolor{black}{w}-v,z-x) \right) \notag\\
&&\times |y-z|^\beta,\notag
\end{eqnarray}
recalling also \eqref{upp_bound_Du} for the last inequality \textcolor{black}{and denoting by $\|\cdot\|_{\dot \bB_{\infty,\infty}^\beta} $ the homogeneous Besov norm (H\"older modulus of order $\beta $)}. 
Hence:
\begin{eqnarray}
&&\Big({\mathcal T}_{p',q'}^{1-\gamma}(\psi_{v,\textcolor{black}{w},s}(\cdot))\Big|_{[0,(\textcolor{black}{w}-v)]}\Big)^{q'}\notag\\
&\le&\frac{C}{(\textcolor{black}{w}-v)^{(\frac{d}{p\alpha})q'}}\int_0^{\textcolor{black}{w}-v} \frac{d\bar v}{\bar v}\bar v^{(\frac{\gamma-1+\beta}{\alpha})q'}\Big( (s-\textcolor{black}{w})^{\frac \varepsilon \alpha}+\frac{(s-\textcolor{black}{w})^{\frac{\theta-1}\alpha}}{(\textcolor{black}{w}-v)^{\frac \beta\alpha}}
\Big)^{q'},\notag\\
&&\textcolor{black}{\bigg(}\int_v^s d\textcolor{black}{w} \Big({\mathcal T}_{p',q'}^{1-\gamma}(\psi_{v,\textcolor{black}{w},s}(\cdot))\Big|_{[0,(\textcolor{black}{w}-v)]}\textcolor{black}{\Big)^{r'}\bigg)^{1/r'}}\notag\\
&\le&\Big(\int_v^s d\textcolor{black}{w} (\textcolor{black}{w}-v)^{(\frac{\gamma-1+\beta}{\alpha}-\frac{d}{p\alpha})r'}\Big( (s-\textcolor{black}{w})^{\frac \varepsilon \alpha}+\frac{(s-\textcolor{black}{w})^{\frac{\theta-1}\alpha}}{(\textcolor{black}{w}-v)^{\frac \beta\alpha}}
\Big)^{r'} \Big)^{1/r'} \notag\\
&\le& C(s-v)^{\frac 1{r'}+(\frac{\gamma-1+\beta}{\alpha}-\frac{d}{p\alpha})+\frac{\varepsilon}{\alpha}}= C(s-v)^{\frac 1{r'}+(\frac{\gamma-1+\theta-1}{\alpha}-\frac{d}{p\alpha})},\label{BD_REMAINDER_COUPURE_BASEE}
\end{eqnarray}
which precisely gives a contribution homogeneous to the one of \eqref{BD_REMAINDER_COUPURE_HAUTE}.We eventually derive that, under the condition \eqref{COND_TO_BE_A_REMAINDER}, the remainder in \eqref{DEF_REMAINDER_DRIFT} is s.t. there exists $\varepsilon':= -1/r + [(\gamma-1+\theta-1)/\alpha]-[d/(p\alpha)]>0$
\textcolor{black}{for which}
\begin{equation}
|{\mathscr R}(v,X_v,s-v)|\le  C (s-v)^{1+\varepsilon'},\qquad  C:=C(\|F\|_{\bL^r([0,T],\bB_{p,q}^{-1+\gamma})}).
\end{equation}
\end{trivlist}
\end{proof}

\subsection{The non-linear stochastic Young integral}\label{YOUNG}
Having derived the shape of the drift, 
let us try to sum up how such a construction can be adapted in our setting. As in Section 4.4.1 of \cite{dela:diel:16}, we introduce in a generic way the process $(A(s,t))_{0\leq s \leq t \leq T}$ as \textcolor{black}{for any $0\le t \le t+ h \le t+ h' \le T$}, $(i) A(t,t+h) = X_{t+h}-X_t$ or $(ii) A(t,\textcolor{black}{t+h}) = \mW_{t+h}-\mW_t$ or $(iii) A(t,\textcolor{black}{t+h}) = \mathfrak f(t,X_t,h)$. We then claim that the following estimates hold: there exists $\varepsilon_0 \in (0,1-1/\alpha]$, $\varepsilon_1,\varepsilon_1' >0$ such that  for any $1\leq \mathfrak q <\alpha$ there exists a constant $C:=C(p,q,r,\gamma,\mathfrak q,T)>0$ such that
\begin{eqnarray}\label{Esti_Inter}
\E^{\frac{1}{\mathfrak q}}[|\E[A(t,t+h)|\mathcal F_t]|^{\mathfrak q}] &\leq& C h^{\frac 1\alpha + \varepsilon_0},\notag\\
\E^{\frac{1}{\mathfrak q}}[|A(t,t+h)|^{\mathfrak q}] &\leq& C h^{\frac 1\alpha},\notag\\
\E^{\frac 1{\mathfrak q}}\big[|\E[A(t,t+h) + A(t+h,t+h') - A(t,t+h')|\mathcal F_t]|^{\mathfrak q}\big]&\leq& C(h')^{1+\varepsilon_1},\notag\\
\E^{\frac 1{\mathfrak q}}[|A(t,t+h) + A(t+h,t+h') - A(t,t+h')|^{\mathfrak q}]&\leq& C(h')^{\frac 1\alpha+\varepsilon_1'}.
\end{eqnarray}

Then, we aim at \textcolor{black}{defining} for any $T>0$ the stochastic integral $\int_0^T \psi_s A(t,t+dt)$, for processes  $(\psi_s)_{s\in [0,T]} $ in $\mathcal H_{\mathfrak q'}^{(1-1/\alpha)-\varepsilon_2}$ (see \eqref{def:predproc} for the definition) with
$\mathfrak q'\ge 1$ such that $1/\mathfrak q'+1/\mathfrak q=1/\ell$, 
for any ${\color{black}1\le }\ell<\alpha$ and $0<\varepsilon_2<\varepsilon_0$, as an $\bL^\ell$ limit of the associated Riemann sum: for $\Delta=\{0=t_0<t_1,\ldots,t_N=T\}$
\begin{equation}
 S(\Delta) := \sum_{i=0}^{N-1}\psi_{t_i}A(t_i,t_{i+1}) \to  \int_0^T \psi_{\textcolor{black}{t}} A(t,t+dt),\quad \text{in } \bL^\ell,
\end{equation}
which justifies the fact that such an integral is called $\bL^\ell$ stochastic-Young integral by the Authors. To do so, the main idea in \cite{dela:diel:16} consists in splitting the process $A$ as the sum of a drift and a martingale:
\begin{eqnarray}\label{DECOMP_A}
A(t,t+h) &= &A(t,t+h)-\E[A(t,t+h)|\mathcal F_t] + \E[A(t,t+h)|\mathcal F_t] \notag\\
& :=&  M(t,t+h)+ R(t,t+h),
\end{eqnarray}
and define $\bL^\ell$-stochastic-Young integral w.r.t. each of these terms. We then have

\begin{THM}[Theorem 16 of \cite{dela:diel:16}]\label{THEO_DEL_DIEL}
There exists $C=C(q,q',p,q,r,\gamma)>0$ such that, given two subdivisions $\Delta \subset \Delta '$ of $[0,T]$, such that $\pi(\Delta) < 1$,
\begin{equation}
\| S(\Delta)-S(\Delta')\|_{\bL^\ell} \leq C\max\{T^{1/\alpha},T\} (\pi(\Delta))^\eta,
\end{equation}
where $\pi(\Delta)$ denotes the step size of the subdivision $\Delta$ and with $\eta = \min\{\textcolor{black}{\varepsilon_0}-\varepsilon_2,\varepsilon_1,\varepsilon_1' \}$.
\end{THM}
\begin{proof}
The main \textcolor{black}{point} consists in noticing that the proof in \cite{dela:diel:16} remains valid in our setting (for parameter \textcolor{black}{$\ell=p$} therein) and that the only difference is the possible presence of jumps. To \textcolor{black}{handle that}, the \textcolor{black}{key idea is then to split} the martingale part (which in our current framework may \textcolor{black}{involve} jumps) into two parts: an $\bL^2$-martingale (which includes the compensated small jumps) and an $\bL^{\textcolor{black}{\ell}}$-martingale (which includes the compensated large jumps). The first part can be handled using \textcolor{black}{the BDG} inequality (and this is what is done in \cite{dela:diel:16}) and the other part by using the compensation formula (such a strategy is somehow classical in the pure-jump setting and has been implemented to prove point \emph{(ii)} in Proposition \ref{PROP_REG_PARTIELLE} above). 
\end{proof}

Thus, we obtain that for any fixed $t$ in $[0,T]$ we are able to define an additive (on $[0,T]$) integral $\int_0^t \psi_s A(s,s+ds)$. The main point consists now in giving a meaning on this quantity as a process (i.e. that all the time integrals can be defined simultaneously). In the current pure-jump setting, we rely on the Aldous criterion, whereas in the diffusive framework of \cite{dela:diel:16}, the Kolmogorov continuity criterion was used. Thanks to Theorem \ref{THEO_DEL_DIEL}, one has
\begin{equation}
\Big\|\int_t^{t+h} \psi_sA(s,s+ds) - \psi_t A(t,t+\textcolor{black}{h}) \Big\|_{\bL^\ell} \leq C h^{\frac 1\alpha + \eta},
\end{equation}
so that one can apply Proposition 34.9 in Bass \cite{bass:11} and Proposition 4.8.2 in Kolokoltsov \cite{kolo:11} to the sequence $\big(\int_0^{t} \psi_s A(s,s+ds)\big)_{s \leq t}$ and deduce that the limit is stochastically continuous.\\

\subsection{Building the dynamics: Proofs of Theorem \ref{THEO_DYN} and Corollary \ref{INTEG_STO}}\label{SEC_DYNAMICS}
We here follow Section 4.6 of \cite{dela:diel:16}. Let us first emphasize that \eqref{Esti_Inter} hold in the case $(i)$ and $(ii)$ mentioned above from Proposition \ref{PROP_REG_PARTIELLE} and Theorem \ref{THE_PDE} (the two last estimates are equals to $0$, since the process $A$ is additive).
 
 we can thus define the process $\big(\int_0^t \psi_s dX_s\big)_{0 \leq t \leq T}$ for any progressively measurable $(\psi_s)_{0 \leq s \leq T}$ {\color{black} in $\mathcal H^{(1-1/\alpha)-\varepsilon_2}_{\mathfrak q'}$ (see \eqref{def:predproc}), $1/\mathfrak q'+1/\mathfrak q=1/\ell$, $1\le \mathfrak q,\ell<\alpha$} with $\varepsilon_2<(\theta-1)/\alpha$. Setting 
$$R(t,t+h) = \E[X_{t+h}-X_t|\mathcal F_t],\quad  M(t,t+h) =X_{t+h}-X_t- \E[X_{t+h}-X_t|\mathcal F_t],$$
the construction of the stochastic Young integral sketched above (see as well subsections 4.4 and 4.5 of \cite{dela:diel:16}) allows to define as well the processes $\big(\int_0^t \psi_s R(s,s+ds )\big)_{0 \leq t \leq T}$ and $\big(\int_0^t \psi_s M(s,s+ds )\big)_{0 \leq t \leq T}$ and the following relation holds
$$\Big(\int_0^t \psi_s dX_s\Big)_{0 \leq t \leq T} = \Big(\int_0^t \psi_s R(s,s+ds )\Big)_{0 \leq t \leq T}+\Big(\int_0^t \psi_s M(s,s+ds )\Big)_{0 \leq t \leq T}.$$
%
%
%
%
Thanks to Proposition \ref{PROP_REG_PARTIELLE} we have that, actually $ \big(\int_0^t \psi_s R(s,s+ds )\big)_{0 \leq t \leq T} =  \big(\int_0^t \psi_s \mathfrak f(s,X_s,ds)\big)_{0 \leq t \leq T},$ so that the r.h.s. is well defined. Also, we have that  $\big(\int_0^t \psi_s (R(s,s+ds ) - \mathscr F(s,X_s,ds))\big)_{0 \leq t \leq T} = \big(\int_0^t \psi_s \mathscr R(s,X_s,ds)\big)_{0 \leq t \leq T}$ is well defined and is null since the bound appearing in the increment of the l.h.s. is greater than one. Hence,
$$\big(\int_0^t \psi_s \mathfrak f(s,X_s,ds)\big)_{0 \leq t \leq T} = \big(\int_0^t \psi_s \mathscr F(s,X_s,ds)\big)_{0 \leq t \leq T}.$$
On the other hand, we have that $\big(\int_0^t \psi_s M(s,s+ds)\big)_{0 \leq t \leq T}$ is well defined as well and that $\big(\int_0^t \psi_s M(s,s+ds) - d\mW_t\big)_{0 \leq t \leq T} =\big(\int_0^t \psi_s \hat M(s,s+ds)\big)_{0 \leq t \leq T} $ where
$$\hat M(t,t+h)  = X_{t+h}-X_t - (\mW_{t+h}-\mW_t) - \E[X_{t+h}-X_t - (\mW_{t+h}-\mW_t)|\mathcal F_t],$$
is an $\bL^{\mathfrak q}$ martingale with $\mathfrak q$ moment bounded by $C_{\mathfrak q} h^{\mathfrak q[1+ (\theta-1)/\alpha]}$ so that it is null as well, meaning that when reconstructing the drift as above, we indeed get that only the ``original'' noise part in the dynamics matters. In other words, for any $(\psi_s)_{0 \leq s \leq T}$ in $\mathcal H^{1-1/\alpha-\varepsilon_2}_{\mathfrak q'}$, with $\varepsilon_2<(\theta-1)/\alpha$,
$$\int_0^t \psi_s dX_s = \int_0^t \psi_s \mathscr F(s,X_s,ds )+\int_0^t \psi_s d\mW_s.$$

\section{Weak formulation and further properties of the drift}\label{futher_prop} 

\subsection{Weak solutions}\label{WEAKSOL}
In this part, we mainly prove Theorem \ref{THEO_STRONG}. Note first that the existence of a weak solution is a consequence of Theorem \ref{THEO_DYN}. It thus only remain to prove weak uniqueness for any $d \ge 1$ to prove Theorem \ref{THEO_STRONG}-(i) (see the corresponding point below) and pathwise uniqueness for $d=1$ to prove Theorem \ref{THEO_STRONG}-(ii) (see the corresponding point below as well). In any cases, we will need to expand a weak solution along the sequence of classical solution $(u_m)_{m\ge 1}$ of the Cauchy problem $\mathscr C(F_m,L^\alpha,f,g,T)$, where $F_m$ is a smooth approximation of $F$ in the sense of Remark \ref{APPROX} and for some smooth functions $f,g$ through It\^o's formula. We are therefore led to check whenever the stochastic integrals $\big(\int_0^t Du_m(s,Y_s)  dY_s\big)_{0 \leq t \leq T}$ can be defined as an $\mathbb L^1$-stochastic Young integral in the sense of Definition \ref{Stochastic_young}. This is the purpose of the next two lemmas. In the first one, we prove that one may define a stochastic calculus w.r.t. the weak solution (\emph{i.e.} w.r.t. quantities in \eqref{DYNAMICS-DEF}), proving thus the last assertion of Theorem \ref{THEO_STRONG}. In the second one, we prove that one can expand the solution of the mollified PDE from Proposition \ref{PROP_PDE_MOLL} along the weak solution through It\^o's formula.

\begin{lem}\label{Def_proc} Assume that the parameters $\alpha,p,q,r$ and $\gamma$ satisfy a \emph{good relation for the dynamics} \eqref{good_relation_dyn} Then, the processes $\big(\int_0^t \psi_s dY_s\big)_{0 \leq t \leq T}$ and  $\big(\int_0^t \psi_s \mathscr F(s,Y_s,ds)\big)_{0 \leq t \leq T}$ are well defined for any progressively measurable process $(\psi_s)_{0\le s \le T}$ in $\mathcal H^{1-1/\alpha-\varepsilon_2}_{\mathfrak q'}$ for all $0<\varepsilon_2<(\theta-1)/\alpha$ and $\mathfrak q' \in \big([\alpha/(\alpha-1)],\infty\big]$.
\end{lem}
\begin{proof}
Starting from the very definition (see Definition \ref{WEAK-DEF}) of a weak solution $(Y,\mZ)$, we readily get from \eqref{PREAL_CTR_GOOD_CONTROL_FOR_YOUNG} that 
\begin{equation}\label{esti_weak_drift}
\forall 1\le \mathfrak q<\alpha,\quad \forall 0\le t <t+h\le T,\  \E^{\frac 1{\mathfrak q}}[|\E[\mathscr F(t,Y_t,h)|\mathcal F_t]|^{\mathfrak q}] + \E^{\frac 1{\mathfrak q}}[|\mathscr F(t,Y_t,h)|^{\mathfrak q}] \le C h^{1/\alpha + [(\theta-1)/\alpha]}.
\end{equation} 
From the construction in Subsection \ref{YOUNG}, this implies in turn that we can define the process $\big(\int_0^t \psi_s dY_s\big)_{0 \leq t \leq T}$ and thus the process $\big(\int_0^t \psi_s \mathscr F(s,Y_s,ds)\big)_{0 \leq t \leq T}$ by the methodology of Subsection \ref{SEC_DYNAMICS}, for any progressively measurable process $(\psi_s)_{0\le s \le T}$ in $\mathcal H^{1-1/\alpha-\varepsilon_2}_{\mathfrak q'}$ for any $0<\varepsilon_2<(\theta-1)/\alpha$ and $\mathfrak q'$ such that $1/\mathfrak q'+1/\mathfrak q=1$ with $1\le \mathfrak q<\alpha$.
\end{proof}

\begin{lem}\label{Ito_proc}Assume that the parameters $\alpha,p,q,r$ and $\gamma$ satisfy a \emph{good relation for the dynamics} \eqref{good_relation_dyn} and let $(Y,\mZ)$ be a weak solutions of \eqref{SDE} in the sense of Definition \ref{WEAK-DEF}. Then,
$$
\bigg(\int_0^t Du_m(s,Y_s)  dY_s\bigg)_{0 \leq t \leq T} \text{ and so }\bigg(\int_0^t Du_m(s,Y_s) \mathscr F(s,Y_s,ds)\bigg)_{0 \leq t \leq T},
$$
where $u_m$ denotes the solution of the Cauchy problem $\mathscr C(F_m,L^\alpha,f,g,T)$ with $F_m$ a smooth approximation of $F$ in the sense of Remark \ref{APPROX} and $f,g$ are smooth functions, are well defined as $\bL^1$-stochastic Young integral in the sense of Definition \ref{Stochastic_young}.
\end{lem}

\begin{proof}
Thanks to the previous lemma it remains to check that there exists $\varepsilon_2$ in $\big(0,[(\theta-1)/\alpha]\big)$ and $\mathfrak q'$ in $\big([\alpha/(\alpha-1)],\infty\big]$ such that $(D u_m(s,Y_s))_{s\in [0,T]}$ belongs to $\mathcal H^{1-1/\alpha-\varepsilon_2}_{\mathfrak q'}$. From \eqref{esti_weak_drift}, we deduce
\begin{equation}\label{TBP}
\forall 1\le \mathfrak q < \alpha,\quad\forall s\neq t \in [0,T],\quad  \E^{\frac 1{\mathfrak q}}[|Y_t - Y_s|^{\mathfrak q}] \le C |s-t|^{\frac 1\alpha}.
\end{equation}
The point is now to notice that, from Proposition \ref{PROP_PDE_MOLL}, we have
$$\forall \mathfrak q'\ge1,\ \exists C_{\mathfrak q'}>0:\ \forall t\neq s \in [0,T],\quad |D u_m(s,Y_s)- Du_m(t,Y_t)|^{\mathfrak q'} \le C_{\mathfrak q'}\Big\{ (s-t)^{\mathfrak q'\frac{\theta-1}{\alpha}} + |Y_t-Y_s|^{\mathfrak q'\rho}\Big\},$$
for any $\rho < \theta-1$.
Set now  $\beta := 1-1/\alpha-\varepsilon_2$ where we recall that $\varepsilon_2\in \big(0, [(\theta-1)/\alpha]\big)$. Thus, if 
$$\exists \varepsilon_2, \rho/\alpha\in \big(0, [(\theta-1)/\alpha]\big),\ \exists  \mathfrak q'\in \big([\alpha/(\alpha-1)],\infty\big] \ {\rm s.t.} \quad  \beta < \rho/\alpha,\ (*)  \text{ and }  \ \mathfrak q'\rho < \alpha,\ (**)$$
we can use \eqref{TBP} together with previous estimate to obtain 
\begin{equation}\label{ESTI_Q'_GRAD}
\exists C'_{\mathfrak q'}>0:\ \|D u_m(\cdot,Y)\|_{\mathcal H^{\beta}_{\mathfrak q'}} := \sup_{t\neq s\in [0,T]} \Bigg\{\|D u_m(s,Y_s)\|_{\bL^{\mathfrak q'}(\tilde \Omega)} + \left\|\frac{|D u_m(s,Y_s)-Du_m(t,Y_t)|}{|s-t|^{\beta}}\right\|_{\bL^{\mathfrak q'}(\tilde \Omega)}\Bigg\} \le C'_{\mathfrak q'},
\end{equation}
from which we deduce, thanks to Lemma \ref{Def_proc}, that both processes 
$$\big(\int_0^t Du_m(s,Y_s)  dY_s\big)_{0 \leq t \leq T}\ {\rm and}\ \big(\int_0^t Du_m(s,Y_s) \mathscr F(s,Y_s,ds)\big)_{0 \leq t \leq T}$$ are well defined. It thus now remains to prove $(*)$ and $(**)$ to conclude the proof. 

From the very definition of $\beta$, $(*)$ rewrites $\beta-\rho/\alpha <0 \Leftrightarrow [(\alpha-1)/\alpha] - \varepsilon_2 - \rho/\alpha <0$. Hence, a sufficient condition for $(*)$ to hold is to prove that there exists $\tilde \rho \in \big(0, 2[(\theta-1)/\alpha]\big)$ such that $[(\alpha-1)/\alpha]- \tilde \rho <0$ (and thus to choose $\varepsilon_2=\rho/\alpha= \tilde \rho/2$) . Notice now from the very definition of $\theta$ in \eqref{DEF_THETA} and the \emph{good relation for the dynamics} \eqref{good_relation_dyn} that we have $2[(\theta-1)/\alpha] > [(\alpha-1)/\alpha]$. Hence, for any choice of the parameters satisfying \eqref{good_relation_dyn}, one can find such a $\tilde \rho$. More precisely,  for any choice of the parameters satisfying \eqref{good_relation_dyn}, there exists $0<\varepsilon <<1$ such that $\tilde \rho=2[(\theta-1)/\alpha]-2\varepsilon$ implies that $(*)$ holds. This allows to choose $\varepsilon_2 = \rho/\alpha =[(\theta-1)/\alpha]-\varepsilon$. It remains to check wether such a choice allows to obtain $\mathfrak q' \in \big([\alpha/(\alpha-1)],\infty\big]$ so that $(**)$ holds. Choose $0<\eta<[(d/p+\alpha/r+\varepsilon)/(\alpha-1)]$ and let $\mathfrak q' = [\alpha/(\alpha-1)]+\eta$, then, $(**)$ holds. Indeed, we have $\mathfrak q'\rho < \alpha \Leftrightarrow \mathfrak q'\rho/\alpha-1<0$ and $ \mathfrak q'\rho/\alpha-1 = [(\theta-\alpha)/(\alpha-1)] - \varepsilon[\alpha/(\alpha-1)] + \eta \rho$. As $[(\theta-\alpha)/(\alpha-1)] = [(\gamma-1-d/p-\alpha/r)/(\alpha-1)] < [(-d/p-\alpha/r)/(\alpha-1)]$ we obtain $\mathfrak q'\rho/\alpha-1 <  [(-d/p-\alpha/r)/(\alpha-1)] - \varepsilon[\alpha/(\alpha-1)] + \eta \rho <0$, which concludes the claim.
\end{proof}

\bigskip
\noindent\emph{(i) Weak uniqueness in any dimension:  proof of point (i) of Theorem \ref{THEO_STRONG}.} Having this result at hand, one can now expand any weak solution of the \emph{formal} SDE \eqref{SDE} along the solution of the Cauchy problem $\mathscr C(F_m,L^\alpha,g,f,T)$ through It\^o's formula for any smooth $f,g$ to obtain for any $t$ in $[0,T]$,
\begin{eqnarray}\label{ITO_u_WEAK}
u_m(t,Y_t) = u_m(0,x) + \int_0^t f(s,Y_s) ds + \int_0^t Du_m(s,Y_s) [\mathscr F(s,Y_s,ds) - F_m(s,Y_s) ds] + M_{\textcolor{black}{0},t}(\alpha,u_m,Y)
\end{eqnarray}
where $M_{\textcolor{black}{0},t}(\alpha,u_m,Y)$, defined by \eqref{def_de_m_alpha} (up to the substitution of $X^m$ by $Y$ therein), is a true martingale thanks to Proposition \ref{PROP_PDE_MOLL}. From Proposition \ref{PROP_DRIFT} and again Proposition \ref{PROP_PDE_MOLL} one
can pass to the (weak) limit  in \eqref{ITO_u_WEAK} to get  that for any $t$ in $[0,T]$,
\begin{eqnarray}\label{ITO_u_WEAK_AFTER_LIMIT}
u(t,Y_t) - u(0,x) - \int_0^t f(s,Y_s) ds =  M_{\textcolor{black}{0},t}(\alpha,u,Y),
\end{eqnarray}
where $u$ is the solution of the Cauchy problem $\mathscr C(\textcolor{black}{F},L^\alpha,g,f,T)$ with $g \in \mathcal C^1(\R^d,\R)$ with $Dg \in \bB_{\infty,\infty}^{\theta-1}(\R^d,\R^d)$, where $\theta$ is given by \eqref{DEF_THETA} and $M_{\textcolor{black}{0},t}(\alpha,u,Y)$ is again a true martingale.

Taking now the expectation, one gets   that for any $t$ in $[0,T]$,
\begin{eqnarray*}
-u(0,x) = \E[u(t,Y_t)] + \E[\int_0^t f(s,Y_s) ds].
\end{eqnarray*}
Choosing $g\equiv 0$ and $t=T$ we obtain that the \textcolor{black}{left} hand side does not depend on the specific choice of $Y$, so that uniqueness in law follows. This concludes the proof of  point (i) in Theorem \ref{THEO_STRONG}.\qed

\begin{REM}\label{equivWeakMart}
Observe that the right hand side of \eqref{ITO_u_WEAK_AFTER_LIMIT} is a $\P$-martingale and one can use usual arguments to build a probability measure on the associated canonical space $\Omega_\alpha$ from $\P$ and $Y$ for which finite dimensional marginal coincide. Hence, the Martingale formulation, in the sense of Definition \ref{DEF_MPB}, holds. In other words, existence of a weak solution implies the existence of a Martingale solution. Also, as a consequence of the previous arguments, it is plain to check that uniqueness of the Martingale solution implies weak uniqueness. 

\end{REM}

\bigskip
\noindent\emph{(ii) Pathwise uniqueness in dimension one:  proof of point (ii) of Theorem \ref{THEO_STRONG}.}
The aim of this part is to prove {\color{black} Theorem \ref{THEO_STRONG}}-(ii), adapting to this end the proof of Proposition 2.9 in \cite{athr:butk:mytn:18} to our current inhomogeneous and parabolic (for the auxiliary PDE concerned) framework.
Let us consider $(X^1,\mW) $ and $(X^2,\mW)$ two weak solutions of the \emph{formal} SDE \eqref{SDE} in the sense of Definition \ref{WEAK-DEF}. Let also $u_m$ be the solution of the Cauchy problem $\mathscr C(F_m,L^\alpha,-F_m,0,T)$. Thanks to Lemma \ref{Ito_proc}, one can apply It\^o's formula on $(X^i_t + u_m(t,X_t^i))$, $i\in \{1,2\}$ to obtain for any $t$ in $[0,T]$ the two corresponding It\^o-Zvonkin transforms 
$$X_t^{Z,m,i} := X_t^i-u_m(t,X_t^i) = x-u_m(0,x) + \mathcal{W}_t- M_{0,t}(\alpha,u_m,X^i) +  R_{0,t}(\alpha,F_m,\mathscr F, X^i), \ i\in\{1,2\},$$
 where $M_{0,t}(\alpha,u_m,X)$ is as in \eqref{def_de_m_alpha} with $X$ instead of $X^m$ therein and $R_{0,t}(\alpha,F_m,\mathscr F, X):= \int_0^t \mathscr F(s,X_s,ds) - F_m(s,X_s)ds$.
 
We point out that we here use the mollified PDE, keeping therefore the remainder term and dependence in $m$ for the martingale part. 
Of course, we will have to control the remainders, \textcolor{black}{which is precisely possible from Proposition \ref{PROP_DRIFT}}. From now on, we assume that $\alpha<2$. \textcolor{black}{The case $\alpha=2$ is indeed easier and can be handled following the arguments below}.\\

As a starting point, we now expand, with 
$$V_n: \R \ni x \mapsto \begin{cases} |x|,\ |x|\ge \frac 1n,\\
\frac 3{8n}+\frac 34 nx^2-\frac 18 n^3x^4,\ |x|\le \frac 1n,
\end{cases}
$$ 
the smooth approximation $V_n(X_t^{Z,m,1}-X_t^{Z,m,2})$ of $|X_t^{Z,m,1}-X_t^{Z,m,2}| $. For fixed $m,n$, thanks to Lemma \ref{Ito_proc}, we can apply It\^o's formula to obtain:
\begin{eqnarray}
&&V_n(X_t^{Z,m,1}-X_t^{Z,m,2})\notag\\
&=&V_n(0)+\int_0^t V_n'(X_{\textcolor{black}{s}}^{Z,m,1}-X_{\textcolor{black}{s}}^{Z,m,2}) \big[\mathscr F(s,X_s^1,ds) - F_m(s,X_s^1)ds\notag\\
&&-(\mathscr F(s,X_s^2,ds) - F_m(s,X_s^2)ds)\big]\notag\\
&&+\int_0^t \textcolor{black}{\int_{\R\backslash\{0\}}}[V_n(X_s^{Z,m,1}-X_s^{Z,m,2}+h_m(\textcolor{black}{s},X_s^{1},X_s^{2},r))-V_n(X_s^{Z,m,1}-X_s^{Z,m,2})] \notag\\
&&\qquad \times \tilde N(ds,\textcolor{black}{dr})\notag\\
&&+\int_{0}^t\int_{|r|\ge 1} \psi_n(X_s^{Z,m,1}-X_s^{Z,m,2},h_m(\textcolor{black}{s},X_s^{1},X_s^{2},r))\nu (dr) ds\notag\\
&&+\int_{0}^t\int_{|r|\le 1} \psi_n(X_s^{Z,m,1}-X_s^{Z,m,2},h_m(\textcolor{black}{s},X_s^{1},X_s^{2},r))\nu (dr) ds\notag\\
 &=:&\frac 3{8n}+\textcolor{black}{\Delta \mathcal R}_{0,t}^{m,n}+\Delta M_{0,t}^{m,n}+\Delta C_{0,t,L}^{m,n}+\Delta C_{0,t,S}^{m,n},\notag\\\label{ITO_FINAL}
\end{eqnarray}
recalling that $X_0^{Z,m,1}=X_0^{Z,m,2} $, using the definition of $V_n$ and denoting for all $(\textcolor{black}{s},x_1,x_2,r)\in \textcolor{black}{[0,t]\times} \R^3 $:
\begin{eqnarray}
h_m(\textcolor{black}{s},x_1,x_2,r)&=&u_m(\textcolor{black}{s},x_1+r)-u_m(\textcolor{black}{s},x_1)-[u_m(\textcolor{black}{s},x_2+r)-u_m(\textcolor{black}{s},x_2)],\label{DEF_HM}\\
\psi_n(x_1,r)&=&V_n(x_1+r)-V_n(x_1)-V_n'(x_1)r.\notag
\end{eqnarray}
The point is now to take the expectations in \eqref{ITO_FINAL}. Since $\Delta M_{0,t}^{m,n}$ is a martingale, we then readily get $\E[\Delta M_{0,t}^{m,n}]=0 $.  On the other hand, since $|V_n'(x)|\le 2 $, we also have from Proposition \ref{PROP_DRIFT} that:
\begin{equation}\label{L1_LIMITE_RESTE}
\E[|\textcolor{black}{\Delta \mathcal R}_{0,t}^{m,n} |] \underset{m}{\to} 0.
\end{equation}
It now remains to handle the compensator terms. For the \textit{large} jumps, we readily write:
\begin{eqnarray}\label{PATHWISE_BIG_JUMPS}
\E[| \Delta C_{0,t,L}^{m,n}|]&\le& 2\|V_n'\|_\infty\|Du_m\|_{\bL^\infty(\bL^{\infty})}\int_0^t\E[|X_s^{1}-X_s^{2}|]ds\notag\\
&\le& C\int_0^t\E[|X_s^{1}-X_s^{2}|]ds,
\end{eqnarray}
\textcolor{black}{observing that $ |h_m(\textcolor{black}{s},x_1,x_2,r)|\le 2\|Du_m\|_{\bL^\infty(\bL^{\infty})} |x_1-x_2|$. Also,}  from Proposition \ref{COR_ZVON_THEO}, $\|Du_m\|_{\bL^\infty(\bL^{\infty})} \le C_T\underset{T\rightarrow 0}{\longrightarrow} 0$  uniformly in $m$ (as the terminal condition of the PDE is 0). 
In particular, for $T$ small enough one has $\|Du_m\|_{\bL^\infty(\bL^{\infty})}\le 1/4 $ and
\begin{eqnarray}
|x_1-u_m(t,x_1)-(x_2-u_m(t,x_2))|&\ge& |x_1-x_2|-|u_m(t,x_1)-u_m(t,x_2)|\notag\\
&\ge& |x_1-x_2|(1-\|Du_m\|_{\bL^\infty(\bL^{\infty})})\notag\\
&\ge& \frac 34 |x_1-x_2|.\label{DOM_1}
\end{eqnarray}
Hence, 
\begin{equation}\label{FROM_X_TO_XZ}
|h_m(\textcolor{black}{s},X_s^{1},X_s^{2},r)|\le 2\|Du_m\|_{\bL^\infty(\bL^{\infty})} |X_s^{1}-X_s^{2}|\le \frac 23 |X_s^{Z,m,1}-X_s^{Z,m,2}|.
\end{equation} 
Therefore, if $ |X_s^{Z,m,1}-X_s^{Z,m,2}|\ge 3/n$, \textcolor{black}{we have for any $r$ either that $X_s^{Z,m,1}-X_s^{Z,m,2}+ h_m(\textcolor{black}{s},X_s^{1},X_s^{2},r) \ge 1/n$ if $\textcolor{black}{X}_s^{Z,m,1}-X_s^{Z,m,2}\ge 3/n$, {\color{black} or}  $X_s^{Z,m,1}-X_s^{Z,m,2}+ h_m(\textcolor{black}{s},X_s^{1},X_s^{2},r) \le -1/n$ if $\textcolor{black}{X}_s^{Z,m,1}-X_s^{Z,m,2}\le -3/n$. It is thus} readily seen that $\psi_n(X_s^{Z,m,1}-X_s^{Z,m,2},h_m(\textcolor{black}{s},X_s^{1},X_s^{2},r))=0 $. We thus have:
\begin{eqnarray}
&&|\E[C_{0,t,S}^{m,n}]|\notag\\
&=&\Big|\E[\int_{0}^t\int_{|r|\le 1} \I_{|X_s^{Z,m,1}-X_s^{Z,m,2}|\le \frac 3n}\psi_n(X_s^{Z,m,1}-X_s^{Z,m,2},h_m(\textcolor{black}{s},X_s^{1},X_s^{2},r))\nu (dr) ds]\Big|\notag\\
&\le&
Cn \E[\int_{0}^t\int_{|r|\le 1} \I_{|X_s^{Z,m,1}-X_s^{Z,m,2}|\le \frac 3n}|h_m(\textcolor{black}{s},X_s^1,X_s^2,r)|^2\nu(dr)ds],\notag\\\label{BD_EXPLO}
\end{eqnarray}
using for the last inequality the definition of $V_n$ which gives that there exists \textcolor{black}{$C$ s.t. for all $x, y\in \R$, $|\psi_n(x,y)| \le Cn|y|^2$}. We now use the definition of $h_m$ and the smoothness of $u_m$ in order to balance the explosive contribution in $n$ and to keep an exponent of $r$ which allows to integrate the small jumps. From \eqref{DEF_HM} and usual interpolation techniques (see e.g. Lemma 5.5 in \cite{athr:butk:mytn:18} or Lemma 4.1 in \cite{prio:12}) we get:
\begin{eqnarray*}
&&|h_m(\textcolor{black}{s},X_s^{1},X_s^{2},r)|\le \|u_m\|_{\bL^\infty(\bB_{\infty,\infty}^{\textcolor{black}{\eta}})}|X_s^1-X_s^2|^{\eta_1} r^{\eta_2},
\end{eqnarray*}
{\color{black}$(\eta_1,\eta_2)\in (0,1)^2, \eta_1+\eta_2=\eta<\theta$. The point is now to apply the above identity with $\eta_1$ large enough in order to get \textcolor{black}{rid} of the explosive term in \eqref{BD_EXPLO} (i.e. $\textcolor{black}{\eta_1}>1/2$) and with $\eta_2$ sufficiently large in order to guarantee the integrability of the L\'evy measure (i.e. $\textcolor{black}{\eta_2>\alpha/2}$). From the very definition of $\theta$, see \eqref{DEF_THETA}, the constraint $1/2+\alpha/2 < \theta$ is satisfied as soon as $\gamma>[3-\alpha+2d/p+2\alpha/r]/2$, which is precisely the condition ensured when the parameters satisfy a \emph{good relation for the dynamics}, see \eqref{good_relation_dyn}. In such a case, for any choice of the parameters $\alpha, p,q,r,\gamma$, one can find $0<\tilde \varepsilon<<1$ such that $\eta_1 = 1/2 + \tilde \varepsilon/2$, $\eta_2 = \alpha/2 + \tilde \varepsilon/2$ and $\eta_1+\eta_2 <\theta$.}

Hence,
\begin{eqnarray}
|\E[C_{0,t,S}^{m,n}]|&\le& Cn \E\left[\int_{0}^t\int_{|r|\le 1} \I_{|X_s^{Z,m,1}-X_s^{Z,m,2}|\le \frac 3n}|X_s^1-X_s^2|^{1+\tilde \varepsilon}r^{ \alpha+\tilde \varepsilon}\frac{dr}{r^{1+\alpha}}ds\right]\notag\\
&\le& Cn \E\left[\int_{0}^t \I_{|X_s^{Z,m,1}-X_s^{Z,m,2}|\le \frac 3n}|X_s^{Z,m,1}-X_s^{Z,m,2}|^{1+\tilde \varepsilon}ds\right]\notag\\
&\le& Cn^{-\tilde \varepsilon},\label{CTR_SMALL_JUMPS_PTHW}
\end{eqnarray}
using \textcolor{black}{\eqref{DOM_1} and the definition of $(X^{Z,m,i})_{i\in \{1,2\}} $} for the last but one inequality. Plugging \eqref{CTR_SMALL_JUMPS_PTHW}, \eqref{PATHWISE_BIG_JUMPS} into \eqref{ITO_FINAL} (taking therein the expectations) and recalling that $\E[\Delta M_{0,t}^{m,n}] =0 $, eventually yields:
\begin{eqnarray*}
\E[V_n(X_t^{Z,m,1}-X_t^{Z,m,2})]\le \frac{3}{8n}+\E[|\textcolor{black}{\Delta \mathcal R}_{0,t}^{m,n}|]+C\int_0^t \E[|X_s^1-X_s^2|] ds+\frac{C}{n^{\tilde \varepsilon}}.
\end{eqnarray*}
Passing to the limit, first in $m$ recalling that $\E[|\textcolor{black}{\Delta \mathcal R}_{0,t}^{m,n}|]\underset{m}{\rightarrow }0$ uniformly in $n$, gives (from the smoothness properties of $ (u_m)_{m\ge 1}$ in Proposition  \ref{PROP_PDE_MOLL}, see also point \emph{(ii)} in Section \ref{SDE_2_PDE}):
\begin{eqnarray*}
\E[V_n(X_t^{Z,1}-X_t^{Z,2})]&\le& \frac{3}{8n}+C\int_0^t \E[|X_s^1-X_s^2|] ds+\frac{C}{n^{\tilde \varepsilon}}, \\X_t^{Z,i}&:=&X_t^i-u(t,X_t^i),\ i\in \{1,2\}.
\end{eqnarray*}
Take now the limit in $n$ and write from \eqref{DOM_1} (which also holds replacing $u_m$ by $u$):
\begin{eqnarray*}
\frac 34\E[|X_t^{1}-X_t^{2}|]&\le& \E[|X_t^{Z,1}-X_t^{Z,2}|]\le C\int_0^t \E[|X_s^1-X_s^2|] ds,\\
\end{eqnarray*}
which readily gives from the Gronwall Lemma  $\E[|X_t^{1}-X_t^{2}|]=0.$ \textcolor{black}{This concludes the proof for $T$ small enough. One may then iterate the argument on small \textcolor{black}{time} intervals to extend the result for any arbitrary $T>0$ and then on the whole positive real line.}\qed

  \subsection{Further properties of the drift}\label{futher_prop_drift}
We here give the proof of Proposition \ref{PROP_DRIFT}. For both the martingale and weak solution, proof of point (i) is obvious from the definition of the dynamics. Point (ii) follows from applying It\^o's formula on the the sequence of classical solution $(u_m)_{m\ge 0}$ of the Cauchy problem $\mathscr C(F_m,L^\alpha,-F_m,0,T)$, where $F_m$ is a smooth approximation of $F$ in the sense of Remark \ref{APPROX}, along the Martingale or the weak solutions, which is licit from Lemma \ref{Ito_proc} (note that this Lemma only use the very definition of the dynamics, so that it holds for the Martingale solution as well). We can then conclude this point by passing to the limit in $m$ together with Proposition \ref{COR_ZVON_THEO}. We now prove points (iii) and (iv).

Let $(F_m)_{m \in \mathbb N^*}$ satisfying 
$$\lim_{m\to \infty} \| F-F_m\|_{\bL^r([0,T],\bB_{p,q}^{-1+\gamma}(\R^d))} =0\footnote{See Remark \ref{APPROX} for the cases when $p$ and/or $r$ are/is $+\infty$. }.$$
We aim at proving that, for either the Martingale or the weak solutions one has 
for all $t$ in $[0,T]$,
\begin{equation}\label{TBPPROP7}
\lim_{m \to \infty} \left\| \int_0^t \psi_{s} \mathscr F(s,X_s,ds) - \int_0^t \psi_{s} F_m(s,X_s)ds \right\|_{\bL^\ell\textcolor{black}{} }= 0,
\end{equation}
with $\ell=1$ in the case of a weak solution (while being not recalling, this last fact will be implicitly assumed in the following). We want to investigate:
\begin{equation}
\lim_{m\to \infty} \E\left|\int_0^t \psi_s \mathscr F(s,X_s,ds) - \int_0^{t} \psi_s F_m(s,X_s) ds\right|^\ell,
\end{equation}
for any $\psi \in \mathcal H^{1-1/\alpha-\varepsilon_2}_{\mathfrak q'}$, $\varepsilon_2 \in \big(0,[(\theta-1)/\alpha]\big)$, $\mathfrak q' \in \big([\alpha/(\alpha-1)],\infty\big]$. Coming back to the definition of such integrals, this means that we want to control
\begin{eqnarray*}
\lim_{m\to \infty} {\color{black}{\lim_{N \to \infty}}}\E\left| \sum_{i=0}^{N-1}\psi_{t_i} \int_{t_i}^{t_{i+1}}ds\left\{ \int dy F(s,y)p_\alpha(s-t_i,y-X_{t_i}) - F_m(t_i,X_{t_i})\right\}\right|^\ell.
\end{eqnarray*}
We have the following decomposition:
\begin{eqnarray*}
&&\lim_{m\to \infty} {\color{black}{\lim_{N \to \infty}}}\E\left| \sum_{i=0}^{N-1}\psi_{t_i} \int_{t_i}^{t_{i+1}}ds\right.\\
&&\times \left. \left\{ \int dy F(s,y)p_\alpha(s-t_i,y-X_{t_i}) - F_m(t_i,X_{t_i})\right\}\right|^\ell\\
&\leq &\lim_{m\to \infty}{\color{black}{\lim_{N \to \infty}}} \E\Bigg| \sum_{i=0}^{N-1}\psi_{t_i} \int_{t_i}^{t_{i+1}}ds\\
&&\times \Bigg\{ \int dy [F(s,y)-F_m(s,y)]p_\alpha(s-t_i,y-X_{t_i})\Bigg\}\Bigg|^\ell\\
&&+\lim_{m\to \infty} {\color{black}{\lim_{N \to \infty}}}\E\Bigg| \sum_{i=0}^{N-1}\psi_{t_i} \int_{t_i}^{t_{i+1}}ds \int dy [F_m(s,y)-F_m(t_i,X_{t_i})]\\
&&\times p_\alpha(s-t_i,y-X_{t_i})\Bigg\}\Bigg|^\ell\\
&:=& \lim_{m\to \infty} {\color{black}\lim_{\pi(\Delta)\to 0}} \|S_m^1(\Delta)\|_{\bL^\ell} + \lim_{m\to \infty}{\color{black}\lim_{\pi(\Delta)\to 0}} \|S_m^2(\Delta)\|_{\bL^\ell}
\end{eqnarray*}
with the previous notations. Note that  $\lim_{m \to \infty}\|S_m^1(\Delta)\|_{\bL^\ell} = 0$, uniformly \textcolor{black}{w.r.t.} $\Delta$ and that from estimate \eqref{PREAL_CTR_GOOD_CONTROL_FOR_YOUNG} and by construction (see Subsection \ref{YOUNG}) for each $m$, $S_m^1(\Delta)$ tends to some $S_m^1$ in $\bL^\ell$ as $ \pi(\Delta) \to 0$. One can hence \textcolor{black}{swap} both limits and \textcolor{black}{therefore} deduce that $$\lim_{m\to \infty} \lim_{\pi(\Delta)\to 0}\|S_m^1(\Delta)\|_{\bL^\ell} =  \lim_{\pi(\Delta)\to 0} \lim_{m\to \infty} \|S_m^1(\Delta)\|_{\bL^\ell} = 0.$$ For the second term, we note that, using Minkowski's and then H\"older's inequalities due to the regularity of $F_m$ (using e.g. its $\bB_{\infty,\infty}^\beta(\bB_{\infty,\infty}^\beta)$ norm, for some $\beta >0$) that 
\begin{eqnarray*}
&&\E^{1/\ell}\left|\sum_{i=0}^{N-1}\psi_{t_i}\int_{t_i}^{t_{i+1}}ds \int dy [F_m(s,y)-F_m(t_i,X_{t_i})]p_\alpha(s-t_i,y-X_{t_i})\right|^\ell \\
&\leq& C_m \sum_{i=0}^{N-1} (t_{i+1}-t_{i})^{(1 + \frac \beta\alpha)},
\end{eqnarray*}
so that $\lim_{m\to \infty} \lim_{\pi(\Delta)\to 0}\|S_m^2(\Delta)\|_{\bL^\ell} =0$. This proves (iii). To prove (iv), it suffices to notice that in such a case, the term $\|S_m^1(\Delta)\|_{\bL^\ell}$ defined above is 0.\qed

\appendix 
\section{Proof of Lemma  \ref{LEM_BES_NORM}} \label{SEC_APP_TEC}
We start with the proof of estimate \eqref{Esti_BES_NORM}. Having in mind the thermic characterization of the Besov norm \eqref{strong_THERMIC_CAR_DEF_STAB}, the main point consists in establishing suitable controls on the thermic part of \eqref{strong_THERMIC_CAR_DEF_STAB} (i.e. the second term in the r.h.s. therein) viewed as the map
\begin{equation*}
s \mapsto \mathcal{T}_{p',q'}^{1-\gamma}[\Psi(s,\cdot ) \mathscr D^\eta p_\alpha(s-t,\cdot-x)].
\end{equation*}
\textcolor{black}{Splitting the interval $[0,1] $ in function of the current time increment $s-t $ (meant to be small) considering $[0,1]=[0,s-t]\cup]s-t,1] $ (low and high cut-off), we write}:
\begin{eqnarray}\label{DEF_HIGH_LOW_CO}
&&\Big(\mathcal{T}_{p',q'}^{1-\gamma}[\Psi(s,\cdot) \mathscr D^\eta p_\alpha(s-t,\cdot-x)]\Big)^{q'}\notag\\
&=&\int_0^1 \frac{dv}{v} v^{(1-\frac{1-\gamma}{\alpha})q'} \| \p_v \tilde p_\alpha(v,\cdot) \star \big(\Psi(s,\cdot) \mathscr D^\eta p_\alpha(s-t,\cdot-x)\big)\|_{\bL^{p'}}^{q'}\notag\\
&=& \int_0^{(s-t)^{}} \frac{dv}{v} v^{(1-\frac{1-\gamma}{\alpha})q'} \| \p_v\tilde p_\alpha(v,\cdot) \star \big(\Psi(s,\cdot) \mathscr D^\eta p_\alpha(s-t,\cdot-x)\big)\|_{\bL^{p'}}^{q'}\notag\\
&&+ \int_{(s-t)^{}}^1 \frac{dv}{v} v^{(1-\frac{1-\gamma}{\alpha})q'} \| \p_v\tilde p_\alpha(v,\cdot) \star \big(\Psi (s,\cdot) \mathscr D^\eta p_\alpha(s-t,\cdot-x)\big)\|_{\bL^{p'}}^{q'}\notag\\
&=:&\Big(\mathcal{T}_{p',q'}^{1-\gamma}[\Psi(s,\cdot) \mathscr D^\eta p_\alpha(s-t,\cdot-x)]|_{[0,(s-t)^{}]}\Big)^{q'}\notag\\
&&+\Big(\mathcal{T}_{p',q'}^{1-\gamma}[\Psi(s,\cdot) \mathscr D^\eta p_\alpha(s-t,\cdot-x)]|_{[(s-t)^{},1]}\Big)^{q'}. 
\end{eqnarray}

For the high cut-off, the singularity induced by the differentiation of the heat kernel in the thermic part is always integrable. Hence using $\bL^1-\bL^{p'}$ convolution inequalities we have
\begin{eqnarray*}
&&\Big(\mathcal{T}_{p',q'}^{1-\gamma}[\Psi (s,\cdot) \mathscr D^\eta p_\alpha(s-t,\cdot-x)]|_{[(s-t),1]}\Big)^{q'}\\
&\leq&\int_{(s-t)^{}}^1 \frac{dv}{v} v^{(1-\frac{1-\gamma}{\alpha})q'} \| \p_v \tilde p_\alpha(v,\cdot) \|_{\bL^1}^{q'} \| \Psi(s,\cdot) \mathscr D^\eta p_\alpha(s-t,\cdot-x)\|_{\bL^{p'}}^{q'}.
\end{eqnarray*}
\textcolor{black}{From \eqref{SENS_SING_STAB} and similarly to \eqref{INT_LP_DENS_STABLE}, we have}
\begin{eqnarray*}
\|\mathscr D^\eta p_\alpha(s-t,\cdot-x)\|_{\textcolor{black}{\bL^{p'}}}
&\le& \frac{\bar C_{p'}}{(s-t)^{\frac{d}{\alpha p}+\frac{|\eta|}{\alpha}}}.
\end{eqnarray*}
We thus obtain
\begin{eqnarray}
&&\Big(\mathcal{T}_{p',q'}^{1-\gamma}[\Psi (s,\cdot) \mathscr D^\eta p_\alpha(s-t,\cdot-x)]|_{[(s-t),1]}\Big)^{q'}\notag\\
&\leq&\textcolor{black}{\|\Psi (s,\cdot)\|_{\bL^{\infty}}^{q'}}\frac{C}{(s-t)^{(\frac d{p\alpha}+\frac {\eta}{\alpha})q'}} \int_{(s-t)^{}}^1 \frac{dv}{v} \frac{1}{v^{\frac{1-\gamma}{\alpha}q'}}\notag \\
&\leq & \frac{C\|\Psi\|_{\bL^\infty(\bL^\infty)}^{q'}}{(s-t)^{\left[\frac{1-\gamma}{\alpha}+\frac d{p\alpha}+\frac {\eta}{\alpha}\right]q'}}.\label{ESTI_COUP_HAUTE}
\end{eqnarray}

To deal with the low cut-off of the thermic part, we need to smooth\textcolor{black}{en} the singularity induced by the differentiation of the heat kernel of the thermic characterization. Coming back to the very definition \eqref{DEF_HIGH_LOW_CO} of this term, we note that
\begin{eqnarray}\label{centering_BESOV_COUPURE_BASSE}
&&\| \partial_v \tilde p_\alpha(v,\cdot) \star \Psi(s,\cdot) \mathscr D^\eta p_\alpha(s-t,\cdot-x)\|_{\bL^{p'}}\\
&=&\Big(\int_{\R^d } dz |\int_{\R^d}dy \partial_v \tilde p_\alpha(v,z-y) \Psi(s,\cdot)\mathscr D^\eta p_\alpha(s-t,y-x)|^{p'} \Big)^{1/p'}\notag\\
&=&\Big(\int_{\R^d } dz \Big|\int_{\R^d}dy \partial_v \tilde p_\alpha(v,z-y)\notag\\
&&\times\Big[\Psi(s,\cdot)\mathscr D^\eta p_\alpha(s-t,y-x)-\Psi(s,\cdot)\mathscr D^\eta p_\alpha(s-t,z-x)\Big]\Big|^{p'} \Big)^{1/p'}\notag.
\end{eqnarray}
To smooth\textcolor{black}{en} the singularity, one then needs to establish a suitable control on the H\"older moduli of the product $\Psi(s,\cdot) \mathscr D^{\eta}p_\alpha (s-t,\cdot-x)$. We claim that for all $(t<s,x)$ in $[0,T]^2 \times \R^d$, for all $(y,z)$ in $(\R^{\textcolor{black}{d}})^2$:
\begin{eqnarray}\label{Holder_prod}
&& \left|\Psi(s,y) \mathscr D^\eta p_\alpha (s-t,y-x) - \Psi(s,z) \mathscr D^\eta p_\alpha(s-t,z-x)\right|\\
&\leq & C\bigg[ \Big(\frac{\|\Psi(s,\cdot)\|_{\dot\bB^\beta_{\infty,\infty}}}{(s-t)^{\frac{\eta}\alpha}} + \frac{\|\Psi(s,\cdot)\|_{\bL^{\infty}
}}{(s-t)^{\frac{\eta + \beta}\alpha}}\Big)\notag\\
&&\times\left(q_\alpha(s-t,y-x) +q_\alpha(s-t,z-x) \right)\bigg] |y-z|^\beta\notag\\
&\leq &   \frac{C}{(s-t)^{\frac{\eta + \beta}\alpha}}
\|\Psi(s,\cdot)\|_{\bB^\beta_{\infty,\infty}} 
\left(q_\alpha(s-t,y-x) +q_\alpha(s-t,z-x) \right) |y-z|^\beta.\notag
\end{eqnarray}
This readily gives, using $\bL^1-\bL^{p'}$ convolution estimates and \eqref{INT_LP_DENS_STABLE}, that
\begin{eqnarray}
&&\ \ \ \ \Big(\mathcal{T}_{p',q'}^{1-\gamma}[\Psi(s,\cdot) \mathscr D^\eta p(s-t,\cdot-x)]|_{[0,(s-t)]}\Big)^{q'} 
\\
&\leq&\!\!\!\frac{C\textcolor{black}{\|\Psi(s,\cdot)\|_{\bB^\beta_{\infty,\infty} }^{q'} }}{(s-t)^{\left[\frac{d}{p\alpha}+\frac{\eta}{\alpha}+\frac{\beta}{\alpha}\right]q'}}\int_{0}^{s-t} \frac{dv}{v} v^{(1-\frac{1-\gamma}{\alpha}-1+\frac{\beta}{\alpha})q'} 
\leq
\frac{C\textcolor{black}{\|\Psi(s,\cdot)\|_{\bB^\beta_{\infty,\infty}}^{q'} } }{(s-t)^{\left[\frac{d}{p\alpha}+\frac{\eta}{\alpha}+\frac{\beta}{\alpha} + \frac{1-\gamma-\beta}{\alpha}\right]q'}}\notag.\label{ESTI_COUP_BASSE}
\end{eqnarray}
\textcolor{black}{Putting together estimates \eqref{ESTI_COUP_HAUTE} and \eqref{ESTI_COUP_BASSE} into \eqref{DEF_HIGH_LOW_CO} yields the estimate \eqref{Esti_BES_NORM} in Lemma \ref{LEM_BES_NORM}}.

\begin{REM}[On the control of the first term in the r.h.s. \eqref{strong_THERMIC_CAR_DEF_STAB}]\label{GESTION_BESOV_FIRST} This term is easily handled by the $\bL^{p'}$ norm of the product $\Psi(s,\cdot ) \mathscr D^\eta p_\alpha(s-t,\cdot-x)$ and hence on $\bL^{p'}$ norm of $\mathscr D^\eta p_\alpha$ times the $\bL^\infty$ norm of $\Psi$. This, in view of \eqref{INT_LP_DENS_STABLE}, clearly brings only a negligible contribution in comparison with the one of the thermic part.
\end{REM}

To conclude with \eqref{Esti_BES_NORM}, it remains to prove \eqref{Holder_prod}. From \eqref{SENSI_STABLE} (see again the proof of Lemma 4.3 in \cite{huan:meno:prio:19} for details), we claim that there exists $C$ s.t. for all $\beta'\in (0,1] $ and all $(x,y,z)\in (\R^d)^2 $,
\begin{eqnarray}
&&|\mathscr D^{\eta} p_\alpha(s-t,z-x)- \mathscr D^{\eta} p_\alpha(s-t,y-x)|\notag\\
&\le& \frac{C}{(s-t)^{\frac{\beta'+\eta}{\alpha}}} |z-y|^{\beta'} \Big( q_\alpha(s-t,z-x)+q_\alpha(s-t,y-x)\Big).\label{CTR_BETA}
\end{eqnarray}
Indeed, \eqref{CTR_BETA} is direct if $|z-y|\ge [1/2] (s-t)^{1/\alpha} $ (off-diagonal regime). It suffices to exploit the bound \eqref{SENSI_STABLE} for $\mathscr D^{\eta} p_\alpha(s-t,y-x) $ and $\mathscr D^{\eta} p_\alpha(s-t,z-x) $ and to observe that $\big(|z-y|/(s-t)^{1/\alpha}\big)^{\beta'}\ge 1 $. If now $|z-y|\le [1/2] (s-t)^{1/\alpha} $ (diagonal regime), it suffices to observe from \eqref{CTR_DER_M} that, with the notations of the proof of Lemma \ref{SENS_SING_STAB} (see in particular \eqref{DECOMP_G_P}), for all $\lambda\in [0,1] $:
\begin{eqnarray}
&&|\mathscr D^{\eta} \textcolor{black}{D} p_M(s-t,y-x+\lambda(y-z))|\notag\\
&\le&
\frac{C_m}{(s-t)^{\frac{\eta+1}\alpha}}p_{\bar M}(s-t,y-x-\lambda(y-z))\notag\\
&\le&
\frac{C_m}{(s-t)^{\frac{\eta+1+d}\alpha}}\frac{1}{\Big( 1+\frac{|y-x-\lambda(z-y)|}{(s-t)^{\frac 1\alpha}} \Big)^{m}} \notag\\
&\le&
\frac{C_m}{(s-t)^{\frac{\eta+1+d}\alpha}}\frac{1}{\Big( \frac 12+\frac{|y-x|}{(s-t)^{\frac 1\alpha}} \Big)^{m}}\le 2\frac{C_m}{(s-t)^{\frac {\eta+1}\alpha}} p_{\bar M}(s-t,y-x).
\label{MIN_JUMP}
\end{eqnarray}
Therefore, in the diagonal case \eqref{CTR_BETA} follows from \eqref{MIN_JUMP} and \eqref{DECOMP_G_P} writing $|\mathscr D^{\eta}p_\alpha(s-t,z-x)- \mathscr D^{\eta} p_\alpha(s-t,y-x)|\le \int_0^1 d\lambda   |\mathscr D^{\eta} D p_\alpha(s-t,y-x+\lambda(y-z)) \cdot (y-z)| \le 2C_m(s-t)^{-[(\textcolor{black}{\eta+1})/\alpha]} q_{\alpha}(s-t,y-x)|z-y|\le \tilde C_m (s-t)^{-[ (\textcolor{black}{\eta+\beta'})/\alpha]} q_{\alpha}(s-t,y-x)|z-y|^{\beta'}$ for all $\beta' \in [0,1] $ (exploiting again that $|z-y|\le [1/2] (s-t)^{1/\alpha} $ for the last inequality). We conclude \textcolor{black}{the proof of \eqref{Holder_prod}} noticing that for all $s$ in $(0,T]$ the map $\R^d \ni y \mapsto \Psi (s,y)$ is $\beta$-H\"older continuous and choosing $\beta'=\beta$ in the above estimate.\\

We now prove \eqref{Esti_BES_HOLD}. Splitting again the thermic part of the Besov norm into two parts (high and low cut-off) we write
\begin{eqnarray*}
&&\Big(\mathcal{T}_{p',q'}^{1-\gamma}[ \Big(\Psi (s,\cdot)\big(\mathscr D^\eta p_\alpha(s-t,\cdot-x) - \mathscr D^\eta p_\alpha(s-t,\cdot-x')\big)]\Big)^{q'}\\
&=&\int_0^1 \frac{dv}{v} v^{(1-\frac{1-\gamma}{\alpha})q'} \notag\\
&&\times\| \p_v\tilde p_\alpha(v,\cdot) \star \Big(\Psi(s,\cdot)\big(\mathscr D^\eta p_\alpha(s-t,\cdot-x) - \mathscr D^\eta p_\alpha(s-t,\cdot-x')\big) \Big)\|_{\bL^{p'}}^{q'}\\
&=& \int_0^{(s-t)^{}} \frac{dv}{v} v^{(1-\frac{1-\gamma}{\alpha})q'} \notag\\
&&\times\| \p_v\tilde p_\alpha(v,\cdot) \star \Big(\Psi (s,\cdot)\big(\mathscr D^\eta p_\alpha(s-t,\cdot-x) - \mathscr D^\eta p_\alpha(s-t,\cdot-x')\big) \Big)\|_{\bL^{p'}}^{q'}\\
&&+ \int_{(s-t)^{}}^1 \frac{dv}{v} v^{(1-\frac{1-\gamma}{\alpha})q'}\notag\\
&&\times \| \p_v\tilde p_\alpha(v,\cdot) \star \Big(\Psi(s,\cdot)\big(\mathscr D^\eta p_\alpha(s-t,\cdot-x) - \mathscr D^\eta p_\alpha(s-t,\cdot-x')\big) \Big)\|_{\bL^{p'}}^{q'}\\
&=:&\Big(\mathcal{T}_{p',q'}^{1-\gamma}[\Big(\Psi(s,\cdot)\big(\mathscr D^\eta p_\alpha(s-t,\cdot-x) - \mathscr D^\eta p_\alpha(s-t,\cdot-x')\big) \Big)]|_{[0,(s-t)^{}]}\Big)^{q'}\\
&&+\Big(\mathcal{T}_{p',q'}^{1-\gamma}[\Big(\Psi (s,\cdot)\big(\mathscr D^\eta p_\alpha(s-t,\cdot-x) - \mathscr D^\eta p_\alpha(s-t,\cdot-x')\big) \Big)]|_{[(s-t)^{},1]}\Big)^{q'}.\notag 
\end{eqnarray*}

Proceeding as we did before for the high cut-off and using \eqref{CTR_BETA}, we have for any $\beta'$ in $[0,1]$:
\begin{eqnarray*}
&&\Big(\mathcal{T}_{p',q'}^{1-\gamma}[\Big(\Psi (s,\cdot)\big(\mathscr D^\eta p_\alpha(s-t,\cdot-x) - \mathscr D^\eta p_\alpha(s-t,\cdot-x')\big) \Big)]|_{[(s-t)^{},1]}\Big)^{q'}\\
&\leq&\int_{(s-t)^{}}^1 \frac{dv}{v} v^{(1-\frac{1-\gamma}{\alpha})q'} \|\p_v \tilde p_\alpha(v,\cdot) \|_{\bL^1}^{q'}\notag\\
&&\times \| \Big(\Psi (s,\cdot)\big(\mathscr D^\eta p_\alpha(s-t,\cdot-x) -\mathscr D^\eta p_\alpha(s-t,\cdot-x')\big) \Big)\|_{\bL^{p'}}^{q'}\\
&\leq&\frac{C\|\Psi (s,\cdot)\|_{\bL^\infty}^{\textcolor{black}{q'}}}{(s-t)^{(\frac d{p\alpha}+\frac {\eta+\beta'}{\alpha})q'}} \int_{(s-t)^{}}^1 \frac{dv}{v} \frac{1}{v^{\frac{1-\gamma}{\alpha}q'}} |x-x'|^{\beta' \textcolor{black}{q'}}\\
&\leq & \frac{C\textcolor{black}{\|\Psi (s,\cdot)\|_{\bL^\infty}^{q'}}}{(s-t)^{\left[\frac{1-\gamma}{\alpha}+\frac d{p\alpha}+\frac {\eta+\beta'}{\alpha}\right]q'}}|x-x'|^{\beta' \textcolor{black}{q'}}.
\end{eqnarray*}
To deal with the low cut-off, we proceed as we did for \eqref{centering_BESOV_COUPURE_BASSE} in order to smooth\textcolor{black}{en} the singularity induced by the differentiation of the thermic kernel. We are hence \textcolor{black}{led} to control the H\"older moduli of $\Psi(s,\cdot)\Big(\mathscr D^\eta {p}_\alpha(s-t,\cdot-x)-\mathscr D^\eta {p}_\alpha(s-t,\cdot-x')\Big)$. We claim that for any $\beta'$ in $(0,1]$ and all $(t<s,x)$ in $[0,T]^2 \times \R^d$, we have that for all $(y,z)$ in $(\R^d)^2$:
\begin{eqnarray}
&& \bigg|\Psi (s,y)\Big(\mathscr D^\eta {p}_\alpha(s-t,y-x)-\mathscr D^\eta {p}_\alpha(s-t,y-x')\Big) \notag\\
&&\quad - \Psi(s,z)\Big(\mathscr D^\eta {p}_\alpha(s-t,z-x)-\mathscr D^\eta {p}_\alpha(s-t,z-x')\Big)\bigg|\notag\\
&\leq &   \frac{C}{(s-t)^{\frac{\eta + \beta+\beta'}\alpha}}\|\Psi(s,\cdot)\|_{\bB^\beta_{\infty,\infty}}\Big(q_\alpha(s-t,y-x) +q_\alpha(s-t,z-x)\notag\\
&&+q_\alpha(s-t,y-x') +q_\alpha(s-t,z-x') \Big) 
 |y-z|^\beta|x-x'|^{\beta'}.\label{Holder_prod-2}
\end{eqnarray}
Repeating the computations in \eqref{centering_BESOV_COUPURE_BASSE} and using the above estimate, we obtain that:
\begin{eqnarray*}
&&\Big(\mathcal{T}_{p',q'}^{1-\gamma}[\Big(\Psi(s,\cdot)\big(\mathscr D^\eta p_\alpha(s-t,\cdot-x) - \mathscr D^\eta p_\alpha(s-t,\cdot-x')\big) \Big)]|_{[0,(s-t)^{}]}\Big)^{q'}\\
&\leq&\frac{C \|\Psi(s,\cdot)\|_{\bB_{\infty,\infty}^\beta}^{\textcolor{black}{q'}}}{(s-t)^{\left[\frac{d}{p\alpha}+\frac{\eta+\beta'}{\alpha}+\frac{\beta}{\alpha}\right]q'}}\int_0^{(s-t)^{}} \frac{dv}{v} v^{(1-\frac{1-\gamma}{\alpha}-1+\frac{\beta}{\alpha})q'} |x-x'|^{\beta'\textcolor{black}{q'}} \\
&\leq& \frac{C \|\textcolor{black}{\Psi(s,\cdot)}\|_{\bB_{\infty,\infty}^\beta}^{\textcolor{black}{q'}}}{(s-t)^{\left[\frac{d}{p\alpha}+\frac{\eta+\beta'}{\alpha} + \frac{1-\gamma}{\alpha}\right]q'}}|x-x'|^{\textcolor{black}{\beta' q'}},
\end{eqnarray*}
provided
\begin{equation}\label{cond_beta_gamma}
\beta+\gamma>1.
\end{equation}

It thus remains to prove \eqref{Holder_prod-2}. It directly follow\textcolor{black}{s} from \eqref{CTR_BETA} that:
\begin{eqnarray}
&& \bigg|\Psi (s,y)\Big(\mathscr D^\eta {p}_\alpha(s-t,y-x)-\mathscr D^\eta {p}_\alpha(s-t,y-x')\Big)\notag\\
&& \quad - \Psi (s,z)\Big(\mathscr D^\eta {p}_\alpha(s-t,z-x)-\mathscr D^\eta {p}_\alpha(s-t,z-x')\Big)\bigg|\notag\\
&\le& \|\Psi (s,\cdot)\|_{\dot \bB_{\infty,\infty}^\beta} |z-y |^{\beta} \frac{C}{(s-t)^{\frac{\eta+\beta'}{\alpha}}}|x-x|^{\beta'}\notag\\&&
\quad \times \big(q_\alpha(s-t,y-x)+q_\alpha(s-t,y-x')\big)\label{CTR_INTERMEDIAIRE_POUR_Holder_prod-2}\\
&&+\|\Psi(s,\cdot)\|_{\bL^\infty}\Big|\big( \mathscr D^\eta {p}_\alpha(s-t,y-x)-\textcolor{black}{\mathscr D^\eta} {p}_\alpha(s-t,y-x')\big)\notag\\
&&-\big(\mathscr D^\eta {p}_\alpha(s-t,z-x)-\mathscr D^\eta {p}_\alpha(s-t,z-x')\big) \Big|.\notag
\end{eqnarray}
Setting:
\begin{eqnarray*}
\Delta(s-t,x,x',y,z)
&:=&\Big|\big( \mathscr D^\eta {p}_\alpha(s-t,y-x)-\mathscr D^\eta {p}_\alpha(s-t,y-x')\big)\\
&&-\big(\mathscr D^\eta {p}_\alpha(s-t,z-x)-\mathscr D^\eta {p}_\alpha(s-t,z-x')\big) \Big|,
\end{eqnarray*}
it now remains to control this term. Precisely,
\begin{trivlist}
\item[-] If $|x-x'|\ge (s-t)^{1/\alpha}/4 $, we write:
\begin{eqnarray}
&&\Delta(s-t,x,x',y,z) \label{HD_1}\\
&\le& \big|\mathscr D^\eta {p}_\alpha(s-t,y-x)-\mathscr D^\eta {p}_\alpha(s-t,z-x) \big|\notag\\
&&+\big|\mathscr D^\eta {p}_\alpha(s-t,y-x')-\mathscr D^\eta {p}_\alpha(s-t,z-x') \big|\notag\\
&\underset{\eqref{CTR_BETA}}{\le}& \frac{C}{(s-t)^{\frac {\eta+\beta}\alpha}} |y-z|^{\beta}\big( q_\alpha(s-t,y-x)+q_\alpha(s-t,y-x')\notag\\
&&+q_\alpha(s-t,z-x)+q_\alpha(s-t,z-x')\big)\notag\\
&\le& \frac{4C}{(s-t)^{\frac {\eta+\beta+\beta'}\alpha}} |y-z|^{\beta}|x-x'|^{\beta'}\big( q_\alpha(s-t,y-x)+q_\alpha(s-t,y-x')\notag\\
&&+q_\alpha(s-t,z-x)+q_\alpha(s-t,z-x')\big).\notag
\end{eqnarray}
\item[-] If $|z-y|\ge (s-t)^{1/\alpha}/4 $, we write symmetrically:
\begin{eqnarray}
&&\Delta(s-t,x,x',y,z) \label{HD_2}\\
&\le& \big|\mathscr D^\eta {p}_\alpha(s-t,y-x)-\mathscr D^\eta {p}_\alpha(s-t,y-x') \big|\notag\\
&&+\big|\mathscr D^\eta {p}_\alpha(s-t,z-x)-\mathscr D^\eta{p}_\alpha(s-t,z-x') \big|\notag\\
&\underset{\eqref{CTR_BETA}}{\le}& \frac{C}{(s-t)^{\frac {\eta+\beta'}\alpha}} |x-x'|^{\beta'}\big( q_\alpha(s-t,y-x)+q_\alpha(s-t,y-x')\notag\\
&&+q_\alpha(s-t,z-x)+q_\alpha(s-t,z-x')\big)\notag\\
&\le& \frac{4C}{(s-t)^{\frac {\eta+\beta+\beta'}\alpha}} |y-z|^{\beta}|x-x'|^{\beta'}\big( q_\alpha(s-t,y-x)+q_\alpha(s-t,y-x')\notag\\
&&+q_\alpha(s-t,z-x)+q_\alpha(s-t,z-x')\big).\notag
\end{eqnarray}
\item[-] If $|z-y|\le (s-t)^{1/\alpha}/4 $ and $|x-x'|\le (s-t)^{1/\alpha}/4 $, we get:
\begin{eqnarray}
&&\Delta(s-t,x,x',y,z)\label{D}\\
&\le&  \int_0^1 d\lambda\int_0^1d\mu |\textcolor{black}{D}_x^{\textcolor{black}{2}} \mathscr D^\eta p_\alpha(s-t,z-x'+\mu(y-z)-\lambda(x-x'))|\notag\\
&&\times |x-x'||z-y|\notag\\
&\le&  \frac{C}{(s-t)^{\frac{\eta+\beta+\beta'}{\alpha}}} |y-z|^\beta |x-x'|^{\beta'}  \big( q_\alpha(s-t,y-x)+q_\alpha(s-t,y-x')\notag\\
&&+q_\alpha(s-t,z-x)+q_\alpha(s-t,z-x')\big)\notag
\end{eqnarray}
proceeding as in \eqref{MIN_JUMP}  and exploiting \eqref{DECOMP_G_P} for the last identity. Plugging \eqref{D}, \eqref{HD_2} and \eqref{HD_1} into \eqref{CTR_INTERMEDIAIRE_POUR_Holder_prod-2} eventually yields the control \eqref{Holder_prod-2}.
\end{trivlist}
\qed

\section*{Acknowledgements}
This work has been funded by the Russian Science Foundation project (project N$^{\circ}$ 20-11-20119).

	\bibliographystyle{alpha}

\bibliography{bibli}

\end{document}